\documentclass[11pt]{article}
\usepackage[a4paper,margin=1in]{geometry}
\usepackage{mathtools}
\usepackage{bm}
\usepackage{natbib}
\usepackage{graphicx}
\usepackage{verbatim}
\usepackage{aliascnt}
\usepackage{hyperref,cleveref}       
\usepackage{url}            
\usepackage{booktabs}       
\usepackage{caption}
\usepackage{subcaption}
\usepackage{mathrsfs,amsfonts,bbm}       
\usepackage{nicefrac}       
\usepackage{microtype}      
\usepackage{enumerate}
\usepackage[inline,shortlabels]{enumitem}
\usepackage{amsmath}
\usepackage{algorithm}
\usepackage{algorithmic}
\usepackage{listings}
\usepackage{xcolor}
\usepackage{multirow}
\usepackage{mmll}
\usepackage{amssymb}

\usepackage{times}
\usepackage{extarrows}
\usepackage{algorithm}
\usepackage{algorithmic}
\usepackage{xifthen}
\usepackage[x11names, svgnames, dvipsnames]{xcolor}
\newcommand{\nonlocal}{\mathcal{M}}
\newcommand{\stp}{\mathcal{T}}
\newcommand{\E}[2][]{\ifthenelse{\isempty{#2}}{\mathbb{E}\left[#1\right]}{\mathbb{E}_{#1}\left[#2\right]}}
\newcommand{\dd}[1]{\mathrm{d}#1}
\newcommand{\CCal}{\mathcal{C}}
\newcommand{\fcls}[2][]{\ifthenelse{\isempty{#2}}{\mathcal{C}\left(#1\right)}{\mathcal{C}^{#1}\left(#2\right)}}
\newcommand{\pcls}[2][]{\ifthenelse{\isempty{#2}}{\mathcal{P}\left(#1\right)}{\mathcal{P}_{#1}\left(#2\right)}}
\newcommand{\lips}[1][]{\ifthenelse{\isempty{#1}}{\mathscr{L}}{\mathscr{L}_{#1}}}
\newcommand{\test}{\mathscr{F}}
\newcommand{\filt}{\mathbb{F}}
\newcommand{\FCal}{\mathcal{F}}
\newcommand{\pbm}{\mathbb{P}}

\newcommand{\LCal}{\mathcal{L}}

\newcommand{\ACal}{\mathcal{A}}
\newcommand{\kl}[2]{D_{{\rm KL}}\left(#1\|#2\right)}
\newcommand{\RCal}{\mathcal{R}}
\newcommand{\subG}[2][\kappa_1,\kappa_2]{\ifthenelse{\isempty{#2}}{{\rm subG}_x\left(#1\right)}{{\rm subG}_{#2}\left(#1\right)}}
\newcommand{\keywords}[1]{\par\noindent\textbf{Keywords:} #1}

\newcommand{\revise}[1]{\textcolor{black}{#1}}
\newcommand{\rerevise}[1]{\textcolor{black}{#1}}
\usepackage{array}
\crefname{definition}{definition}{definitions}
\Crefname{definition}{Definition}{Definitions}

\title{A Two-fold Randomization Framework\\ for Impulse Control Problems \footnote{The work of the first and the third authors was partially funded by the DSAI and Department startup fund for early career faculty. The work of the second author was funded by National Natural Science Foundation of China
(No.12471425).}}
\author{Haoyang Cao\footnote{Department of Applied Mathematics, Data Science and AI Institute, Mathematical Institute for Data Science, Johns Hopkins University, Baltimore, MD, USA (hycao@jh.edu).}, Yuchao Dong\footnote{School of Mathematical Sciences, Tongji University, Shanghai, China (ycdong@tongji.edu.cn).},
Zhouhao Yang\footnote{Department of Applied Mathematics, Johns Hopkins University, Baltimore, MD, USA (zyang145@jh.edu).}
}
\date{}

\begin{document}

\maketitle

\begin{abstract}
We propose and analyze a randomization scheme for a general class of impulse control problems. The solution to this randomized problem is characterized as the fixed point of a compound operator which consists of a regularized nonlocal operator and a regularized stopping operator. This approach allows us to derive a semi-linear Hamilton-Jacobi-Bellman (HJB) equation. Through an equivalent randomization scheme with a compound Poisson  measure, we establish a verification theorem that implies the uniqueness of the solution. Via an iterative approach, we prove the existence of the solution. The existence--and--uniqueness result ensures the randomized problem is well-defined. We then demonstrate that our randomized impulse control problem converges to its classical counterpart as the randomization parameter \(\blambda\) vanishes. This convergence, combined with the value function's \(\mathcal{C}^{2,\alpha}_{loc}\) regularity, confirms our framework provides a robust approximation and a foundation for developing learning algorithms. Under this framework, we propose an offline reinforcement learning (RL) algorithm. Its policy improvement step is naturally derived from the iterative approach from the existence proof, which enjoys a geometric convergence rate. We implement a model-free version of the algorithm and numerically demonstrate its effectiveness using a widely-studied example. The results show that our RL algorithm can learn the randomized solution, which accurately approximates its classical counterpart. A sensitivity analysis with respect to the volatility parameter \(\sigma\) in the state process effectively demonstrates the exploration--exploitation tradeoff.
\end{abstract}

\section{Introduction}
Impulse control models systems in which actions are infrequent but high-impact: a controller waits and then intervenes at selected times to shift the state by a chosen jump size. 
The policy usually takes the form of \(\upsilon=\{(\tau_n,\xi_n)\}_{n\ge1}\), where \(\{\tau_n\}_{n\geq1}\) denotes a sequence of stopping times indicating the timing of inventions, and \(\{\xi_n\}_{n\geq1}\) denotes the associated strategic jumps. Under a given policy \(\upsilon\), the controlled state dynamics is given by 
\[
\dd X_t=b(X_{t-})\dd t+\sigma(X_{t-})\dd W_t+\sum_{n\ge1}\xi_n\delta(t-\tau_n).
\]
The objective is to minimize the discounted total cost,
\[\psi(x)=\inf_{\upsilon}J(x;\upsilon)=\inf_\upsilon\EE\left[\int_0^\infty e^{-rt}f(X_{t-})\dd t+\sum_{n\ge1}e^{-r\tau_n}l(\xi_n)\Big|X_{0-}=x\right],\]
with a discount rate \(r>0\), a running cost \(f\), and an intervention cost \(l\). Compared to regular control, impulse control is a more versatile mathematical framework as it incorporates discontinuities into the state processes. It also generalizes optimal stopping or the singular control problems, as impulse interventions can incur a nonzero fixed cost \cite{HT1983}. As a result, various problems in mathematical finance, engineering and management can be modeled as impulse control problems, including cash management \cite{constantinides1978existence}, inventory controls \cite{Harrison1983, Sulem1986}, transaction cost in portfolio analysis \cite{Eastham1988, Morton1995, Korn1998, Korn1999, Bielecki2000, Oksendal2002}, insurance model \cite{Jeanblanc-Picque1995, Candenillas2006}, liquidity risk \cite{LyVath2007}, exchange rates \cite{jeanblanc1993impulse, Mundaca1998, Bertola2016}, real options \cite{Triantis1990, Mauer1994,BensoussanRoignant}, energy market \cite{basei2019optimal}, and so on. 

In contrast to the popularity as a modeling tool, impulse control problems face technical difficulties in deriving and analyzing the optimal strategies. In general, the analysis of impulse control problems is often through the associated Hamilton-Jacobi-Bellman (HJB) equation of the quasi-variational inequality (QVI) type,
\[
\min\left\{\frac{\sigma^2}{2}\psi''+b\psi'-r\psi+f, \nonlocal\psi-\psi\right\}=0,\]
where \(\nonlocal\) is a nonlocal operator such that
\(\nonlocal\psi(x)=\inf_{\xi\in\RR}\{\psi(x+\xi)+l(\xi)\}\); see, for instance, \cite{bensoussan1984impulse,oksendal2007applied} for more details on the QVIs. Due to the very presence of this nonlocal operator and correspondingly the discontinuous nature of the controlled process, it takes delicate treatments of the value function, especially its regularity conditions, so that popular methods such as the smooth-fit principle can be applied; see for instance \cite{guo2009smooth,davis2010impulse,bayraktar2013impulse}. One of the best known strategies would be the threshold-type \((d,D,U,u)\) derived for the one-dimensional cash management model in \cite{constantinides1978existence}: the controller intervenes only if the state process exits the continuation region \((d,u)\), and in this case, the state is immediately adjusted to the closer threshold between \(D\) and \(U\); see also \cite{alvarez2004class} for more explicitly solvable cases. It is particularly challenging to {explicitly} derive the value function and the corresponding optimal strategy under a generic impulse control framework. 

To {numerically} solve impulse control problems instead, several numerical methods for the corresponding HJB-QVIs have been developed. Many of these methods are based on iterative partial differential equation (PDE) frameworks. The most common route is through penalization combined with policy/Howard iteration: one solves a sequence of penalized HJB equations and iterates to a fixed point (see, e.g., \cite{bensoussan1984impulse,reisinger2019penalty,reisinger2020error}). Other widely used approaches include direct-control or semi-Lagrangian discretizations with policy iteration and matrix-monotonicity arguments \cite{azimzadeh2016weakly,azimzadeh2018convergence}, finite-difference (FD)/finite-element (FE) schemes that provide $L^\infty$-type error bounds for (non)coercive QVIs \cite{boulbrachene1998noncoercive,boulbrachene2005pointwise}, implicit time-stepping for finite-horizon formulations \cite{ieda2015implicit}, and probabilistic representations via backward stochastic differential equations (BSDEs) that enable Monte-Carlo solvers for certain classes \cite{kharroubi2010backward}. While effective for low dimensional state space, these methods face serious limitations in modern applications: grid-based PDE schemes suffer from the curse of dimensionality and delicate monotonicity/stencil requirements; penalty formulations introduce stiffness and parameter tuning (often degrading accuracy near free boundaries); FD/FE error analyses can hinge on regularity or coercivity not guaranteed in practice; and BSDE solvers may require nested simulation and restrictive jump structures. 

As an alternative way to bypass the aforementioned obstacles in traditional numerical methods, reinforcement learning (RL) has been widely studied in recent years to solve optimal stochastic control problems under the continuous setting. The essence of RL is to learn about the possibly unknown environment from either the historical trajectory data (offline) or the interactions (online), and then to solve a decision making problem according to some objective function. Therefore, it is crucial to balance between learning and decision-making, also known as the ``exploration--exploitation'' tradeoff. To encourage exploration, a randomized optimal control framework is necessary. For regular control problems, a common randomization scheme is through an entropy regularization over the policy; see, for instance, \cite{wang2020reinforcement}. Due to the newly-introduced Shannon's entropy term, the controller is encouraged to {explore} the entire trajectory space instead of only {exploiting} the optimal one under the deterministic setting. In \cite{wang2020reinforcement}, a thorough analysis is conducted on a linear-quadratic problem, where the exploration--exploitation tradeoff is full characterized through the optimal randomized feedback strategy as a Gaussian distribution. The analysis of the entropy-regularized framework under a general setting is conducted in \cite{tang2022exploratory}. Under this framework, {model-free} updates have been developed in continuous-time temporal-difference learning and $q$-learning \cite{jia2022policy,jia2023q}, where the control randomization provides a policy-gradient estimator and has a variance reduction effect in the meantime. 

To design learning algorithms for optimal stopping problems, a proper randomized framework essentially increases the frequency of receiving the otherwise extremely sparse reward signals associated with the stopping. One way to achieve that is to randomize the survival process \(\{\mathbbm{1}\{\tau>t\}\}_{t\geq0}\) for any given stopping time \(\tau\). For instance, in \cite{dong2024randomized}, the survival process is modeled as \(\{e^{-\int_0^t\pi_sds}\}_{t\geq0}\) that is controlled by the nonnegative survival rate \(\pi=\{\pi_t\}_{t\geq0}\). With a running normalized entropy regularization \(\cR(\pi_t)\), the randomized optimal stopping problem is solved via analyzing the associated semi-linear HJB equation. A corresponding offline RL algorithm is proposed.
In \cite{dianetti2024exploratory}, the survival process is modeled as \(\{1-\xi_t\}_{t\geq0}\) where \(\xi=\{\xi_t\}_{t\geq0}\) is a nondecreasing singular control process valued on \([0,1]\). With a cumulative relative entropy regularization on \(\xi\), the randomized optimal stopping problem is formulated as a singular control problem whose state space is augmented with the current survival probability \(y\in[0,1]\); its solution can be derived through its obstacle-type HJB equation. Correspondingly, two versions of RL algorithms, one model-based and the other model-free, are proposed to learn the decision boundary characterizing the optimal bang-bang control. In both \cite{dong2024randomized} and \cite{dianetti2024exploratory}, the respective randomized optimal stopping frameworks via the survival process are shown to converge to their classical counterparts as the parameter \(\lambda\) modulating the randomization decreases to \(0\).

Extending these randomization ideas to impulse control problems is nontrivial, as control policy concerns with not just one but a sequence of stopping times together with the associated jumps. In \cite{denkert2025control}, a special case of impulse control problem is considered where the policy \(\upsilon\) is modeled as a marked point process with an associated (uncontrolled) Poisson random measure; see \cite{kharroubi2010backward} for more details. Instead of randomizing the survival processes for each stopping time, the randomization in \cite{denkert2025control} is through a controlled intensity process that tilts the probability measure. Under the intensity policy, the policy gradient is derived and a corresponding actor-critic RL algorithm is then proposed. 

\subsection*{Our Work}
In this work, we seek  a randomization scheme for a general class of impulse control problems whose control policies are beyond the form of marked point processes. 
The idea is inspired by the fixed-point characterization of the impulse control value function \(\psi\) under the operator \(\stp=T\circ\nonlocal\), such that for any proper test function \(\phi\), 
\[\stp\phi(x)=T[\nonlocal\phi](x)=\inf_{\tau}\E{\int_0^\tau e^{-rt}f(Y_{t-})dt+e^{-r\tau}\nonlocal\phi(Y_{\tau-})\biggl|Y_{0-}=x},\ \ \forall x\in\RR,\]
subject to \(dY_t=b(Y_t)dt+\sigma(Y_t)dW_t\); see \cite{bensoussan1984impulse,guo2009smooth}.
We propose a two-fold randomization scheme for impulse control problem. First, we randomize the nonlocal operator into \(\nonlocal^\lambda\) for some \(\lambda>0\) such that for any suitable test function \(\phi\), 
\[\nonlocal^\lambda\phi(x)=\inf_{\mu,\mu\ll\Phi_x}\E[\xi\sim\mu]{\phi(x+\xi)+l(\xi)}+\lambda\kl{\mu}{\Phi_x},\ \ \forall x\in\RR,\]
where \(\Phi_x\) denotes the Gaussian distribution \(\cN(-x,1)\), and \(\kl{\mu}{\Phi_x}\) denotes the Kullback-Leibler divergence between \(\mu\) and \(\Phi_x\); here, \(\kl{\mu}{\Phi_x}\) effectively acts as the entropy-regularizer that encourages the exploration of jumps.
Then, we randomize the optimal stopping operator into \(T^{\lambda}\) for some \(\lambda>0\), such that for any suitable test function \(\phi\), 
\[T^{\lambda}\phi(x)=\inf_{\pi_t\geq0}\E{\int_0^\infty e^{-\int_0^tr+\pi_sds}(f(Y_{t})-\lambda\cR(\pi_t)+\pi_t\phi(Y_t))dt\biggl|Y_{0-}=x},\ \ \forall x\in\RR,\]
where \(\cR(\pi_t)\) is the normalized entropy as in \cite{dong2024randomized} to encourage the exploration of intervention times. Given any \(\blambda=(\lambda_1,\lambda_2)\in(0,+\infty)^2\), the value function \(\psi^{\blambda}\) of the randomized impulse control problem is the fixed point for the compound operator \(\stp^{\blambda}=T^{\lambda_1}\circ\nonlocal^{\lambda_2}\); the intensity-jump distribution pair \((\pi^*,\mu^*)\) denotes the corresponding optimal optimal policy if \(\psi^{\blambda}\) is attained at \((\pi^*,\mu^*)\).

Based on the randomized framework introduced above, we first derive the associated HJB equation for the value function \(\psi^{\blambda}\). Then, using an equivalent randomization method via a Poisson compound measure, we provide a verification theorem. This verification theorem guarantees the uniqueness of the solution to the randomized impulse control problem. To establish the existence of the solution, we propose a constructive proof via an iterative approach. We then show that the iterates of value functions admit a pointwise limit, which is exactly a fixed point for the compound operator \(\stp^{\blambda}\); indeed, the randomized impulse control value function \(\psi^{\blambda}\) is shown to be a \(C^{2,\alpha}_{\mathrm{loc}}\) classical solution to the corresponding HJB equation, and it can be attained at a unique optimal randomized policy \((\pi^*,\mu^*)\). With some additional boundedness assumptions, we establish the convergence of  \(\psi^{\blambda}\) to the corresponding classical impulse control value function \(\psi\) as \(\blambda\) approaches \((0,0)\) upon taking a proper subsequence.

The iterative approach in the proof of existence for \(\psi^{\blambda}\) naturally gives rise to the corresponding RL algorithm, where the iterates correspond to the policy improvement steps whose convergence to the randomized value function is guaranteed by this existence proof. Under the same set of bounded assumptions, we can further derive a geometric convergence rate. We then design a temporal–difference (TD)-based algorithm. It provides an offline model-free learning procedure to solve the randomized impulse control problem. We implement and evaluate this RL algorithm on a classical linear model that has been widely explored in the impulse–control literature; see \cite{constantinides1978existence,jeanblanc1993impulse,guo2009smooth}, for instance.
The experiments empirically verify the convergence of the randomized problem to the classical one as \((\lambda_1,\lambda_2)\) approaches \((0,0)\). We also provide a sensitivity analysis with respect to the volatility parameter \(\sigma\) in the linear model under the randomized setting. The exploration--exploitation trade off is demonstrated through the impact of \(\sigma\) on value functions, intervention intensities as well as jump distributions.

The paper is organized as follows: 
In \Cref{sec:formulation}, we introduce a randomization framework for the impulse-control problem using some newly-defined entropy-regularized operators, and state the definition of the solution to this randomized problem. In \Cref{sec:prelim}, we discuss some preliminary properties for the aforementioned randomized operators. We present the HJB characterization as well as a verification theorem for the randomized impulse control problem in \Cref{sec:HJB-vrf}. This verification theorem guarantees the uniqueness of the solution to this randomized solution.
In \Cref{sec:existence}, we propose an iterative approach to prove the existence of the solution and provide the regularity of the value function.  In \Cref{sec:rand2classical}  we establish the convergence of randomized impulse control value function to its classical counterpart as the regularization effect disappears.
Based on the randomization framework, we develop an offline RL algorithm in 
\Cref{sec:rl}: we first prove geometric convergence of a policy-improvement scheme, and then introduce a temporal–difference method to implement a model-free version of the RL algorithm, and finally report experimental results on a widely studied linear problem. All the detailed proofs for the theoretical results in \Cref{sec:prelim,sec:HJB-vrf,sec:existence,sec:rand2classical,sec:rl} are deferred to \Cref{app:a}; additional experimental results are documented in \Cref{app:b}.

\subsection*{Notations} 
We specify the following list of notations for the rest of this paper.
\begin{itemize}
    \item Given any \(k\in\NN\), let \(\CCal^k=\fcls[k]{\RR}\) be the class of \(k\)-times-differentiable functions on \(\RR\). When \(k=0\), we use \(\CCal=\CCal(\RR)\) to denote the class of continuous functions.
    \item Given any \(\alpha\in(0,1]\) and any \(k\in\NN\), let \(\CCal^{k,\alpha}=\fcls[k,\alpha]{\RR}\) denote the \(\alpha\)-H\"{o}lder space,
    \[\CCal^{k,\alpha}\coloneqq\left\{f\in\CCal^k:\|f\|_{\CCal^{k,\alpha}}\coloneqq\|f\|_{\CCal^k}+\left[f^{(k)}\right]_\alpha<\infty\right\},\]
    where \(\|f\|_{\fcls[k]{K}}\coloneqq\sup\limits_{0\leq l\leq k,x\in \RR}|f^{(l)}(x)|\) and \(\left[f^{(k)}\right]_\alpha\coloneq \sup\limits_{x,y\in\RR,x\neq y}\frac{\left|f^{(k)}(x)-f^{(k)}(y)\right|}{|x-y|^\alpha}\), for any function \(f\in\CCal^k\) and \(f^{(l)}\) being the \(l\)-th order derivative of \(f\) for any \(l=0,\dots,k\). Given any precompact subset \(K\) of \(\RR\) denoted by \(K\Subset \RR\), let \(\fcls[k,\alpha]{K}\) denote the \(\alpha\)-H\"{o}lder space restricted to domain \(K\). Let \(\CCal^{k,\alpha}_{loc}=\CCal^{k,\alpha}_{loc}\left(\RR\right)\) be the locally \(\alpha\)-H\"{o}lder space,
    \[\CCal^{k,\alpha}_{loc}\coloneqq\left\{f:\RR\to\RR:f\in\fcls[k,\alpha]{K},\,\forall K\Subset\RR\right\}.\]
    For simplicity, when \(k=0\), we use \(\CCal^\alpha\) and \(\CCal^\alpha_{loc}\), with \(\|\cdot\|_{\CCal}\) denoted as \(\|\cdot\|_\infty\).
    \item Let \(\lips\) be the set of Lipschitz-continuous functions,
    \[\lips\coloneqq\left\{f:\RR\to\RR:[f]_1<\infty\right\}.\]
    Denote \(\|f\|_{\lips}\coloneqq\sup_{x\in\RR}\frac{|f(x)|}{1+|x|}\) for any \(f\in\lips\). Given any \(L\in[0,\infty)\), let \(\lips[L]\) be the set of \(L\)-Lipschitz functions,
    \[\lips[L]\coloneqq\left\{f\in\lips:[f]_1\leq L\right\}.\]
    \item Let \(\test\) denote the set of test functions \(\phi:\RR\to\RR\) such that \(\inf_{x\in\RR}\phi(x)\in\RR\) and \(\underline{\beta}\coloneqq\inf\{\beta\geq0:\sup_{x\in\RR}\frac{|\phi(x)|}{1+|x|^\beta}<+\infty\}<+\infty\). 
    \item Let \(\pcls{\RR}\) be the set of probability measures on \(\RR\). For any \(p\geq1\), let \(\pcls[p]{\RR}\) be the  set of probability measures with finite \(p\)-th moment,
    \[\pcls[p]{\RR}\coloneqq\left\{\mu\in\pcls{\RR}:\|\mu\|_p^p\coloneq\int_{\RR}|x|^p\mu(dx)<\infty\right\}.\]
\end{itemize}

\section{Randomized Impulse Control Problem}\label{sec:formulation}
In this section, we introduce the mathematical formulation of the randomized impulse control problem. Consider a filtered probability space \((\Omega,\FCal,\pbm,\filt=\{\FCal_t\}_{t\geq0})\) supporting a standard one-dimensional Brownian motion \(W=\{W_t\}_{t\geq0}\), where the filtration \(\filt\) satisfies the usual conditions. If there is no additional specification of probability measures, we will use \(\EE\) to denote the expectation with respect to \(\pbm\). 

On \((\Omega, \FCal,\pbm,\filt)\), the uncontrolled state process \(X=\{X_t\}_{t\geq0}\) is assumed to satisfy the following dynamic,
\begin{equation}\label{eq:sde-uncontrolled}
\mathrm dX_t = b\left(X_t\right)\mathrm dt + \sigma \left(X_t\right)\mathrm dW_t,
\end{equation}
where \(X_0\sim\mu_0\in\pcls[2]{\RR}\) such that \(X_0\perp W\) and \(X_0\in\FCal_0\). Let \(\LCal\) denote the infinitesimal generator associated with \eqref{eq:sde-uncontrolled} such that for any given test function \(\phi\in\CCal^2\), 
\begin{equation}
    \label{eq:inf-gen}
    \LCal\phi(x)=b(x)\phi'(x)+\frac{\sigma(x)^2\phi''(x)}{2},\quad\forall x\in\RR.
\end{equation}
We first specify the following assumptions for the state dynamic \eqref{eq:sde-uncontrolled}.
\begin{assumption}
    \label{ass:dynamic}
    The drift \(b\) and the volatility \(\sigma\) in \eqref{eq:sde-uncontrolled} satisfy the following conditions.
    \begin{enumerate}[label=\textup{(\roman*)}]
        \item There exists \(L>0\) such that 
        \(|b(x)-b(y)|+|\sigma(x)-\sigma(y)|\leq L|x-y|,\quad\forall x,y\in\RR.\)
        \item There exists \(\sigma_0>0\) such that \(|\sigma(x)|\geq\sigma_0,\quad\forall x\in\RR.\)
    \end{enumerate}
\end{assumption}
Under Assumption~\ref{ass:dynamic}, \eqref{eq:sde-uncontrolled} admits a strong solution adapted to \(\filt\) such that \(\EE[X_t^2]<\infty\) for all \(t\geq0\); see, for instance, \cite[Theorem 2.9]{karatzas1998brownian}. In addition, we have the following technical lemma which is a direct result of applying It\^{o}'s formula, Gr\"{o}nwall's inequality and Cauchy-Schwarz inequality.
\begin{lemma}
    \label{lem:l1-error-propagation}
    Let \(X^x=\{X_t^x\}_{t\geq0}\) satisfy \eqref{eq:sde-uncontrolled} with \(X_0=x\) for any \(x\in\RR\). Then, for any \(x,y\in\RR\), \(\E{\sup_{t\geq0}e^{-\left(L+\frac12L^2\right)t}|X^x_t-X^y_t|}\leq |x-y|\).
\end{lemma}

\subsection{Review of the Classical Impulse Control Problem}\label{subsec:classical_impulse_control}
In the literature of the classical impulse control problem without randomization (see, for instance, \cite{guo2009smooth}), the objective can be seen as the following optimization problem,
\begin{equation}
    \label{eq:obj-classical}
    \psi(x)\coloneqq\inf_{\upsilon\in\Upsilon}J(x;\upsilon),\quad\forall x\in\RR.
\end{equation}
Here, \(\Upsilon\) denotes the set of admissible impulse control policies, \(\upsilon=\bigl\{(\tau_n,\xi_n)\bigr\}_{n\ge 1}\), such that \(\tau_n\) is an \(\filt\)-stopping time and the random jump \(\xi_n\in\FCal_{\tau_n}\) for any \(n\in\NN^+\), with \(\tau_n\uparrow\infty\) as \(n\to\infty\), \(\pbm\)-almost surely. Under any \(\upsilon\in\Upsilon\), the controlled state process satisfies
\begin{equation}\label{eq:sde-controlled}
\mathrm dX_t
= b(X_{t-})\mathrm dt + \sigma(X_{t-})\mathrm dW_t
+ \sum_{n\ge 1} \xi_n\delta(t-\tau_n),
\end{equation}
where \(X_{0-}\sim\mu_0\). Here, \(\delta\) denotes the Dirac mass at 0. The cost function \(J(x;\upsilon)\) in \eqref{eq:obj-classical} for any \(x\in\RR\) under any given \(\upsilon\in\Upsilon\) is defined as
\begin{equation}\label{eq:performance}
J(x;\upsilon)
=\mathbb E^{x}\biggl[
\int_{0}^{\infty} e^{-rt}f(X_{t-})\mathrm dt
+\sum_{n\ge 1} e^{-r\tau_n}l(\xi_n)
\biggr],
\end{equation}
with discount rate \(r>0\), running cost \(f:\mathbb R\to\mathbb R\), and cost of intervention \(l:\mathbb R\to\mathbb R\). We then introduce the following assumptions related to the cost function.
\begin{assumption}\label{assump:basic}
The running cost \(f\), the cost of intervention \(l\), and the discount rate \(r\) satisfy the following conditions.
\begin{enumerate}[label=\textup{(\roman*)},ref=Assumption \ref{assump:basic}-(\roman*)]
    \item The running cost \(f:\RR\to[0,+\infty)\) satisfies \(f(0)=0\) and there exists $L_f>0$ such that \(f\in\lips[L_f]\), that is,
    \begin{equation}
    |f(x)-f(y)|\leq L_f|x-y|, \quad \forall x,y\in \RR.
    \end{equation}
    \item The cost of intervention \(l:\RR\to\RR\) satisfies 
    \begin{enumerate}
        \item \(\exists K>0\) such that \(\inf_{x\in\RR}l(x)=l(0)=K\) and \(l(x)+l(y)\geq l(x+y)+K\) for any \(x,y\in\RR\);
        \item \(l(\xi)\to+\infty\) as \(|\xi|\to+\infty\);
        \item \(\exists L_l>0\) such that \(l\in\lips[L_l]\), that is, 
        \[|l(x)-l(y)|\leq L_l|x-y|\quad\forall x,y\in\RR.\]
    \end{enumerate}
    \item\label{discount} Denoting \(G\coloneqq L+\frac{1}{2}L^2\), the discount factor \(r\) satisfies $r-G\in\left(0,\frac{L_f}{L_l}\right)$. 
\end{enumerate}
\end{assumption}
Let \(\nonlocal\) be a nonlocal operator such that for any test function \(\phi:\RR\to\RR\),
\begin{equation}
    \label{eq:nonloc}
    \nonlocal\phi(x)\coloneqq\inf_{\xi\in\RR}\left[\phi(x+\xi)+l(\xi)\right],\quad \forall x\in\RR.
\end{equation}
Classical results
\cite{bensoussan1984impulse,guo2009smooth}
show that the value function $\psi$ in \eqref{eq:obj-classical} is the unique viscosity solution of the following Hamilton-Jacobi-Bellman (HJB) equation of the quasi-variational-inequality type,
\begin{equation}\tag{HJB-QVI}\label{eq:inf-impulse-classical}
\min\left\{\cL\phi(x)-r\phi(x)+f(x),\cM\phi(x)-\phi(x)\right\}=0,
\quad\forall x\in\RR.
\end{equation}
Moreover, $\psi\in\lips\cap\CCal^2(\RR\setminus\partial\mathfrak{C})$, where \(\partial\mathfrak{C}\) denotes the boundary of the continuation region \(\mathfrak{C}\coloneqq\{x\in\RR:\cL\psi(x)-r\psi(x)+f(x)=0\}\). The corresponding optimal classical impulse control policy \(\upsilon^*=\{(\tau^*_n,\xi^*_n)\}_{n\geq1}\) is given by 
\begin{align}
    &\tau_1^*=\inf\{t\geq0:X_{t-}\not\in\mathfrak{C}\};\quad \tau_{n}^*=\inf\{t>\tau_{n-1}^*:X_{t-}\not\in\mathfrak{C}\},\,n\geq2.\label{eq:classical-stp}\\
    &\xi_n^*=\xi^*(X_{\tau_n^*-})\in\arg\min_{\xi\in\RR}\psi(X_{\tau_n^*-}+\xi)+l(\xi),\quad n\geq1.\label{eq:classical-jump}
\end{align}
In \eqref{eq:classical-jump}, \(\xi^*:\ACal\to\RR\) denotes an optimal jump when the current state is in the action region \(\ACal=\RR\setminus\mathfrak{C}\).

Apart from characterizing the optimal impulse control policy via \eqref{eq:inf-impulse-classical}, an equivalent characterization has also been analyzed in \cite{bensoussan1984impulse,guo2009smooth} via the following compound optimal stopping operator \(\stp\), where for any test function \(\phi:\RR\to\RR\) that is bounded from below, 
\begin{equation}
    \tag{classical}\label{eq:inf-classical-impulse-control}
    \stp\phi(x)\coloneqq\inf_{\tau\in\mathbb T}\E[x]{\int_0^\tau e^{-rt}f(X_{t-})\dd{t}+e^{-r\tau}\nonlocal\phi(X_{\tau-})},
\end{equation}
subject to \eqref{eq:sde-uncontrolled}, where \(\mathbb T\) denotes the set of \(\filt\)-stopping times such that \(\pbm(\tau<\infty)=1\) for any \(\tau\in\mathbb T\), and \(\E[x]{\cdot}=\E{\cdot|X_{0-}=x}\). \revise{In particular, the value function \(\psi\) in \eqref{eq:obj-classical}, and $\stp\psi$ as defined in \eqref{eq:inf-classical-impulse-control}, both solve the same \eqref{eq:inf-impulse-classical} with $\nonlocal\phi$ replaced by $\nonlocal\psi$. It can be shown that $\psi$ is a fixed point for \(\stp\), that is, \(\psi=\stp\psi\) on \(\RR\).} Correspondingly, from \eqref{eq:classical-stp}, we know that \(\stp\psi(x)\) defined by \eqref{eq:inf-classical-impulse-control} admits an optimal stopping time with respect to \eqref{eq:sde-uncontrolled},
\begin{equation}
    \label{eq:classical-opt-jump-alt}
    \tau^*=\inf\{t\geq0:X_{t-}\not\in\mathfrak{C}\}\in\mathbb T.
\end{equation}

\subsection{Randomization of Impulse Controls}\label{subsec:randomized_impulse_control}
\revise{The deterministic nature of the optimal classical impulse control policy in \eqref{eq:classical-stp}--\eqref{eq:classical-jump} lies in the fact that  both the intervention time and the jump are deterministic functions of the state process. This deterministic ``wait–then–intervene'' structure is natural when the environment $(b,\sigma,f,l)$ is fully specified, but it becomes restrictive when some or all of these components are only partially known and must be inferred from data. In such settings, classical grid-based HJB--QVI solvers are difficult to deploy, whereas RL methods are more natural. }

\revise{Inspired by \eqref{eq:inf-classical-impulse-control}, we propose a two-fold randomization scheme for impulse controls.  The classical policy induces sparse and highly structured feedback: interventions occur only at a finite set of deterministic boundary points and the chosen jumps bring the state process to corresponding optimal positions, which yields limited exploration and weak learning signals. By contrast, allowing randomness in both the jump distribution and the intervention time produces smoother, more frequent feedback and an intrinsic exploration mechanism; this is crucial for balancing exploration and exploitation in the subsequent RL framework.}

\revise{The following two randomized operators are the key ingredients for the formulation of the randomized impulse control problem in \Cref{subsec:scheme}.}

\paragraph{Randomization of the jump}
\revise{Our Assumption~\ref{assump:basic} on the structure of the cost functions ensures that the classical value function is approximately ``V-shaped'', where the origin is near the unique minimum and inside the continuation region. Inspired by this observation, we adopt a reference jump distribution so that the state process is driven towards the origin where
the value function is near its minimum.}

More specifically, for any \(x\in\RR\), let \(\Phi_x\) denote the Gaussian distribution \(\cN(-x,1)\) with mean \(-x\) and variance \(1\). Intuitively, \(\Phi_x\) denotes a simple randomized jump strategy at \(x\) such that the jump should send the state back to the origin subject to a standard normal disturbance. 

For any \(\mu\in\pcls{\RR}\) such that \(\mu\ll\Phi_x\), recall the Kullback-Leibler (KL) divergence
\[\kl{\mu}{\Phi_x}\coloneqq\E[{\zeta\sim\Phi_x}]{\rho(\zeta)\log \rho(\zeta)}\geq0,\]
where \(\rho\) is the Radon-Nikodym derivative, \(\rho=\frac{d\mu}{d\Phi_x}\). Given any \(\lambda>0\), define the randomized nonlocal operator \(\nonlocal^\lambda\) such that, for any test function \(\phi:\RR\to\RR\) that is bounded from below,
\begin{equation}
    \label{eq:rand-nonloc}
    \nonlocal^\lambda\phi(x)\coloneqq\inf_{\mu\in\pcls{\RR},\mu\ll\Phi_x}\int_{\RR}\left(\phi(x+\xi)+l(\xi)\right)\mu(\dd{\xi})+\lambda\kl{\mu}{\Phi_x},\quad\forall x\in\RR.
\end{equation}

\rerevise{In fact, our proofs rely on $\Phi_x$ only through three properties: (1) a strictly
positive Lebesgue density for the KL divergence in \eqref{eq:rand-nonloc} to be
well-defined; (2) finite log-moment-generating functions
$\log \mathbb{E}_{\zeta \sim \Phi_x}[e^{-(\varphi(x+\zeta)+l(\zeta))/\lambda}]$
for any Lipschitz $\varphi$ (Lemma~\ref{lem:rand-nonloc-attain}); and (3) a local
Lipschitz condition on $d\Phi_x/d\Phi_y$ in $x$, supporting
Lemma~\ref{lem:rand-nonloc-attain} and~\ref{lemma:nonlocal_convergence}.
Any location-scale family with sub-Gaussian tails satisfies all three; the Gaussian
choice is made here for the closed form \eqref{eq:rand-nonloc-value}.}


\paragraph{Randomization of the intervention time} 
Similar to \cite{dong2024randomized}, we borrow the terminology from Cox process \cite{jeanblanc2009mathematical} with $\pi_t\geq0$ being the intensity for any \(t\geq0\), and define the survival probability by time \(t\geq0\) as \(p_t=\exp\left\{-\int_0^t\pi_s\dd{s}\right\}\); or equivalently, the process \(p=\{p_t\}_{t\geq0}\) satisfies \(p_0=1\) and
\begin{equation}
    \label{eq:survival}
    dp_t=\dot{p}_tdt=-\pi_tp_tdt.
\end{equation}
Define \(\RCal:[0,+\infty)\to\RR\) as \(\RCal(x)=x-x\log{x}\) for all \(x>0\) and \(\RCal(0)=\lim_{x\to0^+}R(x)=0\). Notice that \(\RCal\) attains its maximum at and only at \(x=1\). For any \(\lambda>0\), define the randomized optimal stopping operator \(T^\lambda\) such that for any test function \(\phi\in\test\) and any \(x\in\RR\),
\begin{equation}
    \label{eq:rand-stp}
    T^\lambda\phi(x)=\inf_{\pi=\{\pi_t\}_{t\geq0},\pi_t\geq0}J^{\lambda,\phi}(1,x;\pi)
\end{equation}
subject to \eqref{eq:sde-uncontrolled}, where for any \(p\in[0,1]\),
\begin{equation}
    \label{eq:rand-stp-cost}
    J^{\lambda,\phi}(p,x;\pi)\coloneqq\E[x]{\int_0^\infty pe^{-\int_0^tr+\pi_s\dd{s}}\left(f(X_t)-\lambda\RCal(\pi_t)+\pi_t\phi(X_t)\right)\dd{t}},
\end{equation}
for any \(\pi=\{\pi_t\}_{t\geq0}\) with \(\pi_t\geq0\) for all \(t\geq0\). For simplicity, we abuse the notation slightly and denote \(J^{\lambda,\phi}(\cdot;\pi)=J^{\lambda,\phi}(1,\cdot;\pi)\) for any intensity process \(\pi\).

\subsection{The Two-Fold Randomization of Impulse Controls}\label{subsec:scheme}
Analogous to \eqref{eq:inf-classical-impulse-control}, for the randomized impulse control problem, we define the following compound optimal stopping operator \(\stp^{\blambda}\) under any given \(\blambda=(\lambda_1,\lambda_2)\in(0,+\infty)^2\) such that for any test function \(\phi\in\test\) and any \(x\in\RR\) and subject to \eqref{eq:sde-uncontrolled},
\begin{equation}\tag{randomized}\label{eq:inf-randomized-impulse-control}
    \begin{aligned}
&\stp^{\blambda}\phi(x)=T^{\lambda_1}\left[\nonlocal^{\lambda_2}\phi\right](x)\\
=&\inf_{\pi}\E[x]{\int_0^\infty e^{-\int_0^tr+\pi_s\dd{s}}\left(f(X_t)-\lambda_1\RCal(\pi_t)+\pi_t\nonlocal^{\lambda_2}\phi(X_t)\right)\dd{t}}.  
    \end{aligned}
\end{equation}

In light of the fixed-point characterization of the classical impulse control problem via \eqref{eq:inf-classical-impulse-control}, we have the following definition of the solution to the randomized impulse control problem via \eqref{eq:inf-randomized-impulse-control}.
\begin{definition}\label{defn:rand-impulse}
    A function \(\psi^{\blambda}:\RR\to\RR\) is said to be a value function for the randomized impulse control problem associated with \(\stp^{\blambda}\) in \eqref{eq:inf-randomized-impulse-control} under any given \(\blambda=(\lambda_1,\lambda_2)\in(0,+\infty)^2\), if \(\psi^{\blambda}\in\lips\) is bounded from below, and \(\psi^{\blambda}=\stp^{\blambda}\psi^{\blambda}\) on \(\RR\). Moreover, if \(\pi^*=\{\pi^*_t\}_{t\geq0}\), with \(\pi^*_t\geq0\) for all \(t\geq0\), and \(\mu^*=\{\mu^*_x\}_{x\in\RR}\), with \(\mu^*_x\in\{\mu\in\pcls{\RR}:\mu\ll\Phi_x\}\) for all \(x\in\RR\), satisfy
    \[\begin{cases}
        \nonlocal^{\lambda_2}\psi^{\blambda}(x)=\E[\xi\sim\mu^*_x]{\psi^{\blambda}(x+\xi)+l(\xi)}+\lambda_2\kl{\mu^*_x}{\Phi_x},\\
        \psi^{\blambda}(x)=\E[x]{\int_0^\infty e^{-\int_0^tr+\pi^*_s\dd{s}}\left(f(X_t)-\lambda_1\RCal(\pi^*_t)+\pi^*_t\nonlocal^{\lambda_2}\psi^{\blambda}(X_t)\right)\dd{t}},
    \end{cases}\;\; \forall x\in\RR,\]
    then we call \(\upsilon^{\blambda}=(\pi^*,\mu^*)\) an optimal randomized impulse control policy associated with \(\psi^{\blambda}\).
\end{definition}

\section{Preliminary Analysis on the Randomized Operators}\label{sec:prelim}
In this section, we first study the properties of the randomized operators introduced in \Cref{sec:formulation}. These properties are fundamental for developing the HJB characterization for the randomized problem  in \Cref{sec:HJB-vrf}, and for solving the randomized impulse control problem according to Definition~\ref{defn:rand-impulse} in \Cref{sec:existence}.

\subsection{The Randomized Nonlocal Operator}\label{subsec:rand-nonlocal}
First, we introduce several basic properties of the randomized nonlocal operator \(\nonlocal^\lambda\) in \eqref{eq:rand-nonloc} being applied to test functions \(\phi\in\test\).
\begin{lemma}
    \label{lem:rand-nonloc-basic}
    Fix any \(\lambda>0\).
    \begin{enumerate}[label=\textup{(\roman*)},ref= Lemma~\ref{lem:rand-nonloc-basic}-(\roman*)]
        \item\label{finite} For any \(\phi\in\test\), \(\nonlocal^\lambda\phi(x)\in\RR\) for all \(x\in\RR\).
        \item\label{global-non-exp} {For any \(\phi_1,\phi_2\in\test\) with \(\phi_1-\phi_2\in L^\infty(\RR)\), \(\left\|\nonlocal^\lambda\phi_1-\nonlocal^\lambda\phi_2\right\|_{L^\infty(\RR)}\leq\left\|\phi_1-\phi_2\right\|_{L^\infty(\RR)}\).}
        \item\label{concave} \(\nonlocal^\lambda\) is concave, i.e., for any \(\phi_1,\phi_2\in \test\) and any \(\eta\in[0,1]\),
        \[\nonlocal^\lambda\left[\eta\phi_1+(1-\eta)\phi_2\right]\geq\eta\nonlocal^\lambda\phi_1+(1-\eta)\nonlocal^\lambda\phi_2.\]
        \item\label{increase} \(\nonlocal^\lambda\) is nondecreasing, i.e., for any \(\phi_1,\phi_2\in \test\) with \(\phi_1\leq\phi_2\) on \(\RR\), \(\nonlocal^\lambda\phi_1\leq\nonlocal^\lambda\phi_2\).
    \end{enumerate}
\end{lemma}
The next results concern  \(\nonlocal^\lambda\) being applied to more restricted families of test functions, in particular, Lipschitz test functions. This additional restriction is also in alignment with the Lipschitz value function introduced in Definition~\ref{defn:rand-impulse}. {Moreover, a closed-form expression for our randomized operator is provided.} 
\begin{lemma}
    \label{lem:rand-nonloc-attain}
    Fix any \(\lambda>0\) and any test function \(\phi,\phi_1,\phi_2\in\lips\) that is bounded from below. Then, \(\nonlocal^\lambda\phi\in\lips[L_l]\), and
    \begin{equation}
        \label{eq:rand-nonloc-value}
        \nonlocal^\lambda\phi(x)=-\lambda\log\E[\zeta\sim\Phi_x]{e^{-\frac{\phi(x+\zeta)+l(\zeta)}{\lambda}}},\quad\forall x\in\RR.
    \end{equation}
    Moreover, there exist \(\kappa_1,\kappa_2>0\) with which we define \(\delta_t\coloneqq 2\kappa_1e^{-\kappa_2t^2}\) for all \(t\geq0\) and
    \begin{equation}
        \label{eq:sub-gaussian}
        \subG{}\coloneqq\left\{\nu\in\pcls{\RR}:\pbm(|x+\xi|>t)\leq \delta_t,\,\xi\sim\nu,\,\forall t\geq0\right\},
    \end{equation}
    such that 
    \[\nonlocal^\lambda\phi(x)=\E[\xi\sim\mu_x]{\phi(x+\xi)+l(\xi)}+\lambda\kl{\mu}{\Phi_x},\quad\forall x\in\RR,\]
    where \(\mu_x\in\subG{}\) with \(\frac{d\mu_x(\xi)}{d\Phi_x(\xi)}\propto e^{-\frac{\phi(x+\xi)+l(\xi)}{\lambda}}\).

    Furthermore, for any \(R>0\), there exists \(\mathscr{E}_R>0\) with \(\lim_{R\to\infty}\mathscr{E}_R=0\) such that 
    \begin{equation*}
    \|\nonlocal^{\lambda}\phi_1-\nonlocal^{\lambda}\phi_2\|_{L^\infty(\RR)}\leq \|\phi_1-\phi_2\|_{L^\infty([-R,R])}+\mathscr{E}_R.
    \end{equation*}
\end{lemma}
Finally, we have the following comparison results among \(\nonlocal^\lambda\)'s with various \(\lambda>0\) and \(\nonlocal\).
\begin{lemma}
    \label{lem:nonlocal-lambda-convergence-properties}
    The properties hold as follows:
    
    \begin{enumerate}[label=\textup{(\roman*)},ref = Lemma~\ref{lem:nonlocal-lambda-convergence-properties}-(\roman*)]
        \item \label{lem:rand-nonloc-monotone}     For any test function \(\phi\in\lips\) that is bounded from below, \(\nonlocal^{\lambda}\phi\leq \nonlocal^{\lambda'}\phi\) for all \(\lambda,\lambda'\in(0,+\infty)\) such that \(\lambda<\lambda'\).
        \item     \label{lem:compare-classical}
    For any \(\lambda>0\) and any \(\phi\in\test\), \(\nonlocal\phi\leq\nonlocal^\lambda\phi\) on \(\RR\).
        \item\label{lemma:nonlocal_convergence}
  For any test function \(\phi\in\lips\) that is bounded from below, \(\lim_{\lambda\to0^+}\|\nonlocal^{\lambda}\phi-\nonlocal\phi\|_{L^{\infty}(K)}=0\) for any \(K\Subset\RR\). 
    \end{enumerate}

\end{lemma}

\subsection{The Randomized Optimal Stopping Operator}\label{subsec:rand-stop}
In this section, we focus on the randomized optimal stopping operator \(T^\lambda\) being applied to any Lipschitz test function \(\phi\). Through a corresponding HJB characterization and a verification theorem, we show that \(T^\lambda\phi\) is the value function attained at a unique optimal intensity process \(\pi^*\).
\begin{lemma}
    \label{lem:rand-stp-lips}
    For any \(\lambda>0\), any \(L_\phi>0\), and any \(\phi\in\lips[L_\phi]\) that is bounded from below, \(T^{\lambda}\phi\in\lips[\max\{\frac{L_f}{r-G},L_\phi\}]\).
\end{lemma}

Note that when \(\phi\in\lips\) bounded from below, the following conventional optimal control problem is well-posed, where for any \(x\in\RR\) and any \(p\in[0,1]\),
\begin{equation}\label{eq:reg-control-test}
\left\{
\begin{aligned}
    v^{\lambda,\phi}(p,x)=&\inf_{\pi}\E{\int_0^\infty e^{-rt}p_t\left(f(X_t)-\lambda\RCal(\pi_t)+\pi_t\phi(X_t)\right)\dd{t}}\\
    \text{subject to }\dd X_t=&b(X_t)\dd t+\sigma(X_t)\dd{W_t},\ \ X_0=x,\\
    \dd p_t=&-\pi_tp_t\dd{t},\ \ p_0=p.
\end{aligned}
\right.
\end{equation}
Moreover, we observe that for any \(x\in\RR\), \(T^\lambda\phi(x)=v^{\lambda,\phi}(1,x)\).
By the dynamic programming principle, the HJB associated with \eqref{eq:reg-control-test} is given by 
\begin{equation}
    \label{eq:hjb-reg-full}
    (\cL-r)v(p,x)+pf(x)+p\inf_{\pi\geq0}\left\{\pi(x)\left[\phi(x)-\partial_pv(p,x)\right]-\lambda\cR(\pi(x))\right\}=0,
\end{equation}
for \(v:[0,1]\times\RR\to\RR\), \(\forall(p,x)\in(0,1)\times\RR\). Observe that \(v^{\lambda,\phi}(p,x)=pT^{\lambda}\phi(x)\). Then from \eqref{eq:hjb-reg-full}, we deduce the HJB associated with \(T^{\lambda}\phi(x)\) as follows,
\begin{equation}
    \label{eq:hjb-regular-test}
    0=(\cL-r)v+f+\inf_{\pi\geq0}\left\{\pi(\phi-v)-\lambda\cR(\pi)\right\}\\
        =(\cL-r)v+f-\lambda e^{-\frac{\phi-v}{\lambda}},
\end{equation}
for \(v:\RR\to\RR\), and the infimum is attained at \(\bar\pi(x)=e^{-\frac{\phi(x)-v(x)}{\lambda}}\) for any \(x\in\RR\), due to the strong concavity of \(\cR\). 

We have the following verification theorem, establishing $T^\lambda \phi$ as a solution to the above HJB.
\begin{theorem}
    \label{thm:classical-control-sol}
    Fix any \(\lambda>0\) and any \(\phi\in\lips\) that is bounded from below, then the value function \(T^\lambda\phi\) in \eqref{eq:rand-stp} is the unique classical solution to the HJB equation \eqref{eq:hjb-regular-test} in \(\CCal^2\cap\lips\). Moreover, the intensity process \(\pi^*=\{\pi^*_t\}_{t\geq0}\) with \(\pi^*_t=\bar\pi(X_t)\) for all \(t\geq0\), where
    \begin{equation}
        \label{eq:opt-ctrl-regular-test}
        \bar\pi(x)=e^{-\frac{\phi(x)-T^\lambda\phi(x)}{\lambda}},\qquad\forall x\in\RR,
    \end{equation}
    is the corresponding optimal control policy such that
    \[T^\lambda\phi(x)=J^{\lambda,\phi}(x;\pi^*),\quad\forall x\in\RR.\]
\end{theorem}

\subsection{The Compound Optimal Stopping Operator}
Composed with \(\nonlocal^{\lambda_2}\) and \(T^{\lambda_1}\), the properties presented in \Cref{subsec:rand-nonlocal,subsec:rand-stop} lead to the following properties of \(\stp^{\blambda}\). These properties lay the foundation for the validity of Definition~\ref{defn:rand-impulse}.
\begin{lemma}\label{lem:stp-lambda-properties}
The compound optimal stopping operator $\stp^{\blambda}$ satisfies the following properties:

\begin{enumerate}[label=\textup{(\roman*)},ref=Lemma~\ref{lem:stp-lambda-properties}-(\roman*)]
    \item     \label{lem:rand-increasing}
    \(\stp^{\blambda}\) is non-decreasing, i.e., for any \(\phi_1,\phi_2\in\test\) with \(\phi_1\leq\phi_2\), \(\stp^{\blambda}\phi_1\leq\stp^{\blambda}\phi_2\).
    \item      \label{lem:rand-lips}
    For any \(\blambda\in(0,+\infty)^2\) and any test function \(\phi\in\lips\) that is bounded from below, \(\stp^{\blambda}\phi\in\lips[\frac{L_f}{r-G}]\).
    \item     \label{lem:global-rand-nonexp}
    For any \(\blambda\in(0,+\infty)^2\) and any test function \(\phi_1,\phi_2\in\lips\) that are bounded from below, \(\|\stp^{\blambda}\phi_1-\stp^{\blambda}\phi_2\|_{L^\infty(\RR)}\leq\|\phi_1-\phi_2\|_{L^\infty(\RR)}\). 
\end{enumerate}

\end{lemma}
\ref{lem:rand-increasing} directly follows from \ref{increase} and its proof is therefore omitted.



\section{HJB Characterization of Randomized Impulse Control Problem and Verification Theorem}\label{sec:HJB-vrf}
In this section, we provide the HJB characterization \eqref{eq:inf_st_HJB} for the randomized impulse control solution in Definition \ref{defn:rand-impulse} under a fixed randomization parameter \(\blambda=(\lambda_1,\lambda_2)\in(0,+\infty)^2\). With the help of an alternative randomization scheme \eqref{eq:mdp-jump-diff}, we then provide a verification theorem for \eqref{eq:inf_st_HJB}, which leads to the uniqueness of solution in Definition~\ref{defn:rand-impulse}. The proofs in this section follow standard arguments using dynamic programming principle and It\^{o}'s formula, therefore they are omitted here and can be found in \cite[Appendix A.2]{cao2025two}.
\subsection{HJB Characterization}\label{subsec:HJB}
According to the compound optimal stopping operator defined in \eqref{eq:inf-randomized-impulse-control}, we have \(\stp^{\blambda}\phi=T^{\lambda_1}\left[\nonlocal^{\lambda_2}\phi\right]\) for any \(\phi\in\test\).
Following \eqref{eq:hjb-regular-test} and \eqref{eq:opt-ctrl-regular-test}, we can derive the HJB equation associated with \(T^{\lambda_1}\left[\nonlocal^{\lambda_2}\phi\right]\),
\begin{equation}
    \label{eq:hjb-aux}
0=(\cL-r)v+f+\inf_{\pi\geq0}\left\{\pi\left(\nonlocal^{\lambda_2}\phi-v\right)-\lambda_1\cR(\pi)\right\}\\
        =(\cL-r)v+f-\lambda_1e^{-\frac{\nonlocal^{\lambda_2}\phi-v}{\lambda_1}},
\end{equation}
for function \(v:\RR\to\RR\). Accordingly, the optimal feedback policy \(\pi^*\) is given by
\begin{equation}
    \label{eq:opt-ctrl-aux}
    \pi^*(x)=e^{-\frac{\nonlocal^{\lambda_2}\phi-\stp^{\blambda}\phi}{\lambda_1}}.
\end{equation}
By Definition~\ref{defn:rand-impulse}, if a function \(\psi^{\blambda}:\RR\to\RR\) is a value function for the randomized impulse control problem, then it must satisfy the fixed-point condition,
\(\psi^{\blambda}=\stp^{\blambda}\psi^{\blambda}=T^{\lambda_1}\left[\nonlocal^{\lambda_2}\psi^{\blambda}\right]\).
Thus, by \eqref{eq:hjb-aux}, we have the following HJB equation associated with \(\psi^{\blambda}\),
\begin{equation}
    \label{eq:inf_st_HJB}\tag{HJB}
    (\cL-r)v+f-\lambda_1e^{-\frac{\nonlocal^{\lambda_2}v-v}{\lambda_1}}=0,
\end{equation}
for \(v:\RR\to\RR\). From \eqref{eq:opt-ctrl-aux}, we know that the optimal feedback policy associated with \(\psi^{\blambda}\) is
\begin{equation}
    \label{eq:opt-ctrl}
    \pi^*(x)=e^{-\frac{\nonlocal^{\lambda_2}\psi^{\blambda}-\psi^{\blambda}}{\lambda_1}}.
\end{equation}
\begin{theorem}
    \label{thm:fx-pnt-hjb-sol}
    If a function \(V:\RR\to\RR\) is a value function for the randomized impulse control problem according to Definition~\ref{defn:rand-impulse}, then \(V\in\CCal^2\) is also a classical solution to \eqref{eq:inf_st_HJB}.
\end{theorem}

\subsection{Verification Theorem}
We will provide a verification theorem for \eqref{eq:inf_st_HJB} through an alternative randomization scheme of impulse control.

\subsubsection{Connection with Randomization via Controlled Compound Poisson Process}
In this section, we discuss an alternative approach to introduce randomness to the classical impulse control policy under some fixed \(\blambda=(\lambda_1,\lambda_2)\in(0,+\infty)^2\). 

Recall the optimal impulse control policy characterized by \eqref{eq:classical-stp}--\eqref{eq:classical-jump}. We can write the corresponding control process as a jump process \(Y^*=\{Y^*_t\}_{t\geq0}\), where 
\begin{equation}
    \label{eq:classical-jump-process}
    Y^*_t=\sum_{n=1}^{N^*_t}\xi^*_n,\ \ \text{where}\ \ N^*_t=\sum_{n\geq1}^\infty\mathbbm{1}\{\tau_n^*\leq t\},\quad\forall t\geq0.
\end{equation}
Here, \(\{(\tau^*_n,\xi^*_n)\}_{n\geq1}\) is given by \eqref{eq:classical-stp}--\eqref{eq:classical-jump}. The deterministic nature of \(Y^*\) is characterized as follows: in \eqref{eq:classical-jump-process}, the associated counting process \(N^*\) is represented by the sequence of arrival times \(\{\tau^*_n\}_{n\geq1}\), which is of a deterministic feedback form; the sequence of jumps \(\{\xi^*_n\}_{n\geq1}\) is also of a deterministic feedback form. 

Inspired by \cite{denkert2025control}, we embed the impulse mechanism via a controlled compound Poisson process. 
Consider \(\pi=\{\pi_t\}_{t\geq0}\) such that \(\pi_t\geq0\) for all \(t\geq0\), and let \(N^\pi=\{N^\pi_t\}_{t\geq0}\) be a nonhomogeneous Poisson process with intensity \(\pi\). Let \(\eta=\{\eta_n\}_{n=1}^\infty\) be a sequence of random jumps. Define the following jump process \(Y^{\pi,\bmu}=\{Y^{\pi,\bmu}_t\}_{t\geq0}\) such that 
\begin{equation}
    \label{eq:compound-poisson}
    Y^{\pi,\bmu}_t=\sum_{n=1}^{N^\pi_t}\eta_n,\qquad\forall t\geq0.
\end{equation}
Here, \(\bmu=\{\mu_t\}_{t\geq0}\) is a flow of probability measures in \(\pcls{\RR}\) such that for any open set \(U\subset \RR\) and any \(0\leq t_1<t_2\),
\[\EE\left[\sum_{t_1<s\leq t_2}\mathbbm{1}\left\{\Delta Y^{\pi,\bmu}_s\in U\right\}\right]=\EE\left[\int_{t_1}^{t_2}\int_{U}Y^{\pi,\bmu}(\dd s,\dd \xi)\right]=\int_{t_1}^{t_2}\mu_s(U)\pi_s\dd{s},\]
with some abuse of notation that \(Y^{\pi,\bmu}(ds,d\xi)\) denotes the corresponding random measure. We then consider the following controlled state process,
\begin{equation}
    \label{eq:jump-diff}
    X_{t}=X_{0-}+\int_{0}^{t} b(X_{s-})\dd{s}+\int_{0}^{t}\sigma(X_{s-})\dd W_{s}+\int_0^t\int_{\RR}\xi Y^{\pi,\bmu}(\dd s,\dd \xi),\quad\forall t\geq0.
\end{equation}
The infinitesimal generator with respect to \eqref{eq:jump-diff} is given by
\begin{equation}
    \label{eq:jump-diff-inf-gen}
    \begin{aligned}
        \hat{\cL}^{\pi,\bmu}\phi(t,x)&=\partial_t\phi(t,x)+\cL\phi(t,x)+\pi_t\E[\xi\sim\mu_t]{\phi(t,x+\xi)-\phi(t,x)}, 
    \end{aligned}
\end{equation}
for any test function \(\phi:[0,+\infty)\times \RR\), \((t,x)\in(0,+\infty)\times\RR\). 

Let \(\cA\) denote the set of admissible pairs \((\pi,\bmu)\) such that for all \(t\geq0\), \(\pi_t=\pi(t,X_{t-})\) for some \(\pi:[0,+\infty)\times\RR\to[0,+\infty)\), and \(\mu_t=\bmu[t,X_{t-}]\) such that \(\bmu[t,y]\in\pcls{\RR}\) with \(\bmu[t,y]\ll\Phi_y\) for all \((t,y)\in[0,+\infty)\times\RR\); for such \(\bmu\), let \(\brho[t,y]=\frac{d\bmu[t,y]}{d\Phi_y}\) and denote \(\brho=\{\rho_t\}_{t\geq0}\) with \(\rho_t=\brho[t,X_{t-}]\), for all \(t\geq0\) and \(y\in\RR\). Later we may use \(\bmu\) and \(\brho\) interchangeably. For any \((\pi,\bmu)\in\cA\), define the cost function \(J^{\blambda}\) with respect to the state process \eqref{eq:jump-diff} as follows,
\[\begin{aligned}
J^{\blambda}(x;\pi,\bmu)&=\EE_x\Big[\int_0^{\infty}e^{-rt}\Big\{f(X_t)-\lambda_1\cR(\pi_t)\\
&\hspace{50pt} +\int_{\RR}\left[l(\xi)-\lambda_2\cH(\mu_t;X_{t-})\right]Y^{\pi,\bmu}(\dd{t},\dd{\xi})\Big\}\dd{t}\Big],\quad\forall x\in\RR,
\end{aligned}\]
with \(\cH(\nu;x)=-D_{KL}(\nu\Vert\Phi_x)\) for any \(\nu\in\pcls{\RR}\) such that \(\nu\ll\Phi_x\). Observe that $\cR$ and $\cH$ serve as the entropy regularization to encourage randomization of the counting process and sequence of jumps, respectively. 
We are interested in the optimal control problem,
\begin{equation}
    \label{eq:value-jump-diff}
    \tilde\psi^{\blambda}(x)=\inf_{(\pi,\bmu)\in\cA}J^{\blambda}(x;\pi,\bmu),
\end{equation}
subject to \eqref{eq:jump-diff} with \(X_{0-}=x\) for any \(x\in\RR\). 

Notice that \(\kl{\bmu[t,y]}{\Phi_y}=\E[\zeta\sim\Phi_y]{\brho[t,y](\zeta)\log{\brho[t,y](\zeta)}}\) for all \(t\geq0\) and \(y\in\RR\). Then, the cost function can be rewritten as
\[\begin{aligned}
&J^{\blambda}(x;\pi,\bmu)\\
=&\E[x]{\int_0^{\infty}e^{-rt}\left\{f(X_t)-\lambda_1\cR(\pi_t)+\pi_t\left[\E[\xi\sim\mu_t]{l(\xi)}-\lambda_2\cH(\mu_{t};X_{t-})\right]\right\}\dd{t}}\\
=&\E[x]{\int_0^{\infty}e^{-rt}\left(f(X_t)-\lambda_1\cR(\pi_t)+\pi_t\E[\zeta\sim\Phi_{X_{t-}}]{\rho_{t}(\zeta)\left(l(\zeta)+\lambda_2\log{\rho_{t}(\zeta)}\right)}\right)\dd{t}}.
\end{aligned}\]
Therefore, the optimal control problem \eqref{eq:value-jump-diff} is equivalent to
\begin{equation}
    \label{eq:mdp-jump-diff}
    \begin{aligned}
       \tilde{\psi}^{\blambda}(x)&=\inf_{(\pi,\rho)\in\cA}\EE\Big[\int_0^{\infty}e^{-rt}\Big(f(X_t)-\lambda_1\cR(\pi_t)\\
       &\hspace{80pt}+\pi_t\E[\xi\sim\Phi_{X_{t-}}]{\rho_{t}(\xi)\left(l(\xi)+\lambda_2\log{\rho_{t}(\xi)}\right)}\Big)\dd{t}\Big],
    \end{aligned}
\end{equation}
subject to \eqref{eq:jump-diff} with \(X_{0-}=x\) for any \(x\in\RR\). 

We have the following equivalence between the fixed-point characterization of randomized impulse control given by Definition \eqref{defn:rand-impulse} and the control problem \eqref{eq:mdp-jump-diff} via their associated HJB's.
\begin{lemma}
    \label{lem:equiv-hjb}
    The HJB equation associated with \eqref{eq:mdp-jump-diff} coincides with \eqref{eq:inf_st_HJB}.
\end{lemma}

\subsubsection{Verification Theorem}
With the alternative randomization framework \eqref{eq:mdp-jump-diff} and the equivalence relation given by Lemma \ref{lem:equiv-hjb}, we present the following verification theorem for
\eqref{eq:inf_st_HJB}.
\begin{theorem}
    \label{thm:rand-verification}
    Suppose \(V\in\CCal^2\cap\lips\) is a classical solution to \eqref{eq:inf_st_HJB} that is bounded from below. Then \(V=\tilde{\psi}^{\blambda}\) as in \eqref{eq:mdp-jump-diff}.
\end{theorem}

\Cref{thm:fx-pnt-hjb-sol} and \ref{thm:rand-verification} immediately lead to the following uniqueness result.
\begin{theorem}
    \label{thm:unique}
    There exists at most one value function \(\psi^{\blambda}\in\lips\) with a finite lower bound according to Definition~\ref{defn:rand-impulse}.
\end{theorem}
Theorem \ref{thm:unique} guarantees that the solution to the randomized impulse control framework introduced in Definition \ref{defn:rand-impulse} is free from ambiguity.

\section{Solving Randomized Impulse Control Problem via an Iterative Approach}\label{sec:existence}
In this section, we provide a constructive proof for the existence of solution to the randomized impulse control problem specified in Definition \ref{defn:rand-impulse} under a fixed \(\blambda\in(0,+\infty)^2\). Its building block is an iterative approach mirroring the
classical Howard iteration \cite{howard1960dynamic}. The construction of the iterates consists of two key components, namely the ``policy evaluation'' and the ``policy improvement''. Tailored to our randomized impulse control problem, these two components are achieved through the compound optimal stopping operator \(\stp^{\blambda}\) defined in \eqref{eq:inf-randomized-impulse-control}. We show that this iterative approach leads to a limiting function, which is indeed a value function under Definition~\ref{defn:rand-impulse}. The proofs are deferred to Appendix~\ref{app:a-2}.

\subsection{Construction of the iterates}\label{subsec:iter-construction}
At the initial step \(n=0\), we adopt a conservative policy, simply not intervening at all. That is, \(\pi^0_t\equiv0\) for all \(t\geq0\). Correspondingly, the total cost \(\psi^{\blambda,0}:\RR\to\RR\) is defined as
\begin{equation}
    \label{eq:policy_improvement_psi0}
    \psi^{\blambda,0}(x) = \EE\left[\int_0^\infty e^{-rt}f(X_t)\dd t \Big| X_0=x\right],\quad \forall x\in\RR.
\end{equation}
Starting with \(\psi^{\blambda,0}\), we continue the construction recursively via \(\stp^{\blambda}\), namely for any \(n\in\NN^+\),
\begin{equation}\label{eq:policy_improvement_psin_value_definition}
\begin{aligned}
    \psi^{\blambda,n}(x)&=\stp^{\blambda}\psi^{\blambda,n-1}(x)\\
    &= \inf_\pi\E[x]{\int_0^\infty e^{-rt} p_t\left(f(X_t)+\nonlocal^{\lambda_2}\psi^{\blambda,n-1} (X_t) \pi_t-\lambda_1\cR(\pi_t) \right)\dd t},
\end{aligned}
\end{equation}
subject to \eqref{eq:sde-uncontrolled} and \eqref{eq:survival} with \(p_0=1\). 

Intuitively, the sequence $\{\psi^{\blambda,n}\}_{n\ge 0}$ is constructed via the compound randomized
optimal stopping operator $\stp^{\blambda}$. At each step, $\psi^{\blambda,n}$ and $\pi^n$
respectively represent the value function and the optimal stopping rule
for the randomized optimal stopping problem with terminal cost $\nonlocal^{\lambda_2}\psi^{\blambda,n-1}$.
Since an optimal stopping problem only solves for the optimal intervention time, not for the jump size, the policy iteration naturally
acts on the survival rate $\pi^n$.
By contrast, the optimal jump-size distribution $\mu$ is already available in closed
form \eqref{eq:rand-nonloc-value} thanks to the regularization via KL divergence, and hence no additional inner optimization over distributions is required.

Furthermore, we have the following property for the iterates \(\psi^{\blambda,n}\)'s.
\begin{lemma}\label{lemma:uniform_lipschitz}
For any \(n\in\NN\), $\psi^{\blambda,n}\in\lips[\frac{L_f}{r-G}]$.
\end{lemma}
From  Lemma~\ref{lemma:uniform_lipschitz}, we know that \(\psi^{\blambda,n},\nonlocal^{\lambda_2}\psi^{\blambda,n}\in\lips[\frac{L_f}{r-G}]\) with a natural lower bound \(0\) for all \(n\in\NN\). Therefore by Theorem~\ref{thm:classical-control-sol}, for \(n\geq1\), \(\psi^{\blambda,n}\) in \eqref{eq:policy_improvement_psin_value_definition} is the unique classical solution in \(\CCal^2\) to the following HJB equation,
\begin{equation}\label{eq:value_improvement}
\begin{aligned}
     0 &=\cL\psi^{\blambda,n}-r\psi^{\blambda,n}+f+\inf_\pi\left\{\left(\nonlocal^{\lambda_2}\psi^{\blambda,n-1}-\psi^{\blambda,n}\right)\pi-\lambda_1\cR(\pi)\right\}\\
      &=\cL\psi^{\blambda,n}-r\psi^{\blambda,n}+f-\lambda_1e^{-\frac{\nonlocal^{\lambda_2}\psi^{\blambda,n-1}(x)-\psi^{\blambda,n}(x)}{\lambda_1}}.
\end{aligned}
\end{equation}
In addition, the corresponding optimal policy $\pi^n$ is given by
\begin{equation}\label{eq:policy_improvement}
    \pi^{n}(x)=e^{-\frac{\nonlocal^{\lambda_2}\psi^{\blambda,n-1}(x)-\psi^{\blambda,n}(x)}{\lambda_1}};
\end{equation}
denote \(p^n_t=e^{-\int_0^t\pi^n_s\dd{s}}\) for any \(t\geq0\).

\subsection{Convergence of the Iterates}
The next result shows that adopting the policy in \eqref{eq:policy_improvement} is indeed a policy improvement step, and correspondingly, the policy evaluation step according to \eqref{eq:policy_improvement_psin_value_definition} will eventually converge to a Lipschitz-continuous function. 
\begin{theorem}\label{thm:policy_improvement}
For \(\{\psi^{\blambda,n}\}_{n\geq0}\) constructed in Section \ref{subsec:iter-construction}, 
\begin{enumerate}[label=\textup{(\roman*)},ref=\ref{thm:policy_improvement}-(\roman*)]
    \item\label{improve}$\psi^{\blambda,n+1}\leq\psi^{\blambda,n}$ for all \(n\in\NN\);
    \item\label{pnt-limit}\(\exists\hat{\psi}^{\blambda}\in\lips[\frac{L_f}{r-G}]\) such that \(\psi^{\blambda,n}\xrightarrow{n\to\infty}\hat{\psi}^{\blambda}\) pointwise. Moreover, $\psi^{\blambda,n}$ converges to $\hat{\psi}^{\blambda}$ locally uniformly.
\end{enumerate}
\end{theorem}
Next, we show that such a pointwise limit \(\hat{\psi}^{\blambda}\) actually corresponds to the randomized impulse control value function.
\begin{theorem}
    \label{thm:iter-value-fct}
    The pointwise limit \(\hat{\psi}^{\blambda}\) in Theorem~\ref{thm:policy_improvement} is a value function for the randomized impulse control problem as in Definition~\ref{defn:rand-impulse}. The associated optimal policy \((\pi^*=\{\pi^*_t\}_{t\geq0},\mu^*=\{\mu^*_x\}_{x\in\RR})\) is given by
    \(\mu^*_x\in\pcls{\RR}\) such that
    \[\rho^*_x(\xi)=\frac{d\mu^*_x(\xi)}{d\Phi_x(\xi)}=\frac{e^{-\frac{\hat{\psi}^{\blambda}(x+\xi)+l(\xi)}{\lambda_2}}}{\E[\zeta\sim\Phi_x]{e^{-\frac{\hat{\psi}^{\blambda}(x+\zeta)+l(\zeta)}{\lambda_2}}}},\quad\forall x\in\RR,\]
    and
    \(\pi^*_t=\bar\pi(X_t)\) for any \(t\geq0\) with
    \(\bar\pi=e^{-\frac{\nonlocal^{\lambda_2}\hat{\psi}^{\blambda}-\hat{\psi}^{\blambda}}{\lambda_1}}.\)
\end{theorem}
Theorem~\ref{thm:policy_improvement} and \ref{thm:iter-value-fct} guarantee that the iterative approach with initialization \eqref{eq:policy_improvement_psi0} and recursive relation \eqref{eq:policy_improvement_psin_value_definition} indeed converges to the value function of the randomized impulse control problem according to Definition \ref{defn:rand-impulse}. Together with the uniqueness result in Theorem~\ref{thm:unique}, we conclude that the randomized impulse control problem via the compound optimal stopping operator \eqref{eq:inf-randomized-impulse-control} is well-defined for any given \(\blambda\in(0,+\infty)^2\).

\section{Convergence of the Randomized Problem to the Classical Problem}\label{sec:rand2classical}

In \Cref{sec:HJB-vrf,sec:existence}, we have shown the well-definedness of the randomized impulse control problem in Definition~\ref{defn:rand-impulse} for any given randomization parameter \(\blambda\). 
In this section, we focus on the relation between this \emph{randomized} problem and the \emph{classical} impulse control problem \eqref{eq:obj-classical}. In particular, we establish a convergence result for the
\emph{randomized} impulse–control problem to its
\emph{classical} counterpart as \(\blambda\) approaches \((0,0)\).

Let us first introduce the sequence of functions $\{\psi^n\}_{n \geq 0}$, representing the classical policy iterates. This sequence converges to the classical solution $\psi$, which solves~\eqref{eq:inf-impulse-classical}. Assume that $\psi^0$ is the value function where no intervention is made. Namely, $\psi^0\equiv\psi^{\blambda,0}$ as defined in \eqref{eq:policy_improvement_psi0}. 
Starting with \(\psi^{0}\), we recursively construct the iteration with $\stp$ (recall \eqref{eq:inf-classical-impulse-control}). For any \(n\in\NN^+\),
\begin{equation}\label{eq:policy_improvement_psin_value_definition_classical}
    {\psi^{n}(x)=\stp\psi^{n-1}(x) = \inf_\tau\E[x]{\int_0^\tau e^{-rt} f(X_t)\dd t+e^{-r\tau}\nonlocal\psi^{n-1} (X_\tau) }},
\end{equation}
subject to \eqref{eq:sde-uncontrolled} and $X_0=x$.

Equivalently, one can show by classical results \cite{bensoussan2011applications,pham2009continuous} that when $n\in \NN^+$, 
$\psi^n$'s are the unique viscosity solution to the corresponding HJB equation:
\begin{equation}\label{eq:classical_PI_hjb}
\min\{\cL\psi^n-r\psi^n+f, \nonlocal\psi^{n-1}-\psi^n\}=0.
\end{equation}

We first state the following assumption, which will be active from now on.  

\begin{assumption}\label{assump:boundedness}
We impose the following boundedness conditions:
\begin{enumerate}[label=\textup{(\roman*)}]
    \item The coefficients satisfy $\|b\|_{L^\infty(\RR)} + \|\sigma\|_{L^\infty(\RR)} + \|f\|_{L^\infty(\RR)} \leq M$.\label{assump:bounded_coeffs}
    \item The continuation regions $\mathfrak{C}^n$ of the policy iterates $\{\psi^n\}_{n \geq 1}$ are uniformly bounded; that is, there exists a bounded set $W \Subset \RR$ such that $\mathfrak{C}^n \subset W$ for all $n$.\label{assump:bounded_continuation_region}
\end{enumerate}
\end{assumption}

The first condition is a common convention in the literature on impulse control problems \cite{reisinger2019penalty,reisinger2020error}.  
The second condition is satisfied in many practical settings. For instance, both our example in \Cref{subsec:experiments} and the numerical study in \cite{reisinger2020error} meet this requirement.

Under Assumption~\ref{assump:boundedness}, one can show that the classical policy iterates $\{\psi^n\}_{n\geq 0}$ locally converges to the classical solution $\psi$ at a geometric rate.
\begin{proposition}[{\cite[Theorem 3.4]{reisinger2019penalty}}  and {\cite[Proposition 4.1]{reisinger2020error}}]\label{propn:classical_iterates_bound}
For any $n\in\NN^+$, let $\psi$ and $\psi^n$ be the unique viscosity solution to \eqref{eq:inf-impulse-classical} and \eqref{eq:classical_PI_hjb} respectively. Then there exist constants $\mu\in(0,1]$ and $C> 0$ such that 
\[0\leq \psi^n-\psi\leq C(1-\mu)^n, \quad n\in\NN.\]
Consequently, the iterates $\{\psi^n\}_{n\geq 0}$ are bounded uniformly in $n$.
\end{proposition}
Given the similarity in the problem settings, the proof of Proposition~\ref{propn:classical_iterates_bound} is essentially the same as that of \cite[Proposition 4.1]{reisinger2020error}, therefore omitted.

Also, one can prove that the iterates $\{\psi^n\}_{n\geq 0}$ have a uniform Lipschitz constant.

\begin{lemma}\label{lemma:uniform_lipschitz_classical_iterates}
For any $n\in\NN$, $\psi^n\in \lips[\frac{L_f}{r-G}]$.
\end{lemma}



To better analyze the fixed point iteration of the classical stopping operator, we introduce a ``semi-randomized'' impulse control problem, where we merely randomize the jump size with $\nonlocal^{\lambda_2}\psi^n$. In specific, for any $n\in \NN$,
\begin{equation}
\begin{cases}
    \nonlocal^{\lambda_2}\phi(x)&=\inf_{\mu\in\mathcal P(\mathbb R)}\left[\int_{\mathbb{R}}\phi(x+\xi)+l(\xi)\mu(d\xi)-\lambda_2\mathcal H(\mu)\right],\\
    \tilde\stp^n\phi(x)&=\inf_{\tau}\mathbb{E}\left[\int_0^\tau e^{-rt}f(X_{t-})\dd{t}+e^{-r\tau}\nonlocal^{\lambda_2}\psi^n(X_{\tau-})\biggl|X_{0-}=x\right].
\end{cases}\label{eq:intermediate-impulse-control}\end{equation}
The corresponding value function $\tilde{\psi}^n$ for the intermediate problem would satisfy the fixed-point problem
\[\tilde\psi^n=\tilde\stp^n\tilde\psi^n\]
as well as the HJB equation
\begin{equation}\label{eq:additional_HJB}
\min\{\cL\tilde{\psi}^n-r\tilde{\psi}^n+f,\nonlocal^{\lambda_2}\psi^n-\tilde{\psi}^n\}=0.
\end{equation}

In fact, one can show by going through the same line of proof as former work \cite[Section 3,5]{pham2009continuous} that for any $n\in \NN$, the newly-introduced semi-randomized solution $\tilde\psi^n$ is the unique viscosity solution to \eqref{eq:additional_HJB}. 
Moreover, we have the following propositions.
\begin{proposition}
For any $n\in\NN$, $\tilde\psi^n$ is the unique viscosity solution to \eqref{eq:additional_HJB}. Moreover, $\tilde{\psi}^n(x)\in \lips[\frac{L_f}{r-G}]$.
\end{proposition}
This follows directly from \cite[Section 3,5]{pham2009continuous} and \cite{bensoussan2011applications} since the obstacle function $\nonlocal^{\lambda_2}\psi^n\in \lips[\frac{L_f}{r-G}]$, therefore the proof is omitted.

\begin{proposition}\label{prop:bdd}
For any $n\in\NN$, $\psi^n$, $\tilde{\psi}^n$, $\psi^{\blambda,n}$, $\stp^{\blambda}\psi^{n}$ are bounded.
\end{proposition}

We also notice the following property, which bounds the difference between the semi-randomized problem and the classical problem.

\begin{lemma}\label{lemma:semi-randomized_nonlocal_bound}
For any compact set $K$ with $W\subset K\Subset\RR$, for any $n\in\NN^+$, $\|\psi^n-\tilde\psi^{n-1}\|_{L^\infty(K)}\leq \|\nonlocal\psi^{n-1}-\nonlocal^{\lambda_2}\psi^{n-1}\|_{L^\infty(K)}$.
\end{lemma}

Then we have the following upper bounds for \(\stp^{\blambda}\psi^n\) and \(\|\stp^{\blambda}\psi^n-\tilde\psi^n\|_{L^\infty(\RR)}\), respectively.
\begin{proposition}\label{propn:inter_bound_nonlocal_stp}
There exists some constant $C>0$, such that for any $n\in \NN$, $x\in\RR$,
    \(\stp^{\blambda}\psi^{n}(x)\le
  \nonlocal^{\lambda_2}\psi^{n}(x)
  +C\lambda_1\log\frac{1}{\lambda_1}.\)
\end{proposition}

\begin{proposition}\label{propn:randomized_os_convergence}
There exists some constant $C>0$, such that for sufficiently small $\lambda_1>0$, for any $n\in\NN$, $\|\stp^{\blambda}\psi^{n}-\tilde{\psi}^{n}\|_{L^\infty(\RR)}\leq C\lambda_1\log\frac{1}{\lambda_1}$.
\end{proposition}

Finally, we present the following convergence theorem, which demonstrates that the randomized impulse control problem approximates its classical counterpart as $\blambda$ tends to \((0,0)\).

\begin{theorem}\label{thm:comparison_convergence}
For any compact set $K$ with $W\subset K$, 
any sequence $\{(\lambda_1^m,\lambda_2^m)\}_{m\ge 1}$ with
\(
\lambda_1^m\to0\)
and
\(\lambda_2^m\to0
\) as \(m\to \infty\),
there exists a subsequence $\{(\lambda_1^{m_k},\lambda_2^{m_k})\}_{k\ge 1}\subset \{(\lambda_1^m,\lambda_2^m)\}_{m\ge 1}$, such that
\(
\|\psi^{\blambda^{m_k}}(x)-\psi(x)\|_{L^\infty(K)}
\xrightarrow[]{k\to\infty}0.
\)
\end{theorem}

\section{Reinforcement Learning for Randomized Impulse Control}\label{sec:rl}

Recall that in \Cref{sec:existence}, we have provided an {iterative approach} that eventually converges to the solution to the randomized impulse control problem in \Cref{subsec:iter-construction} . In this section, we show that this iterative approach actually serves as the foundation for the design of a reinforcement learning algorithm to numerically compute the solution in Definition~\ref{defn:rand-impulse} with a temporal–difference (TD) loop. First, we establish a geometric convergence result for the iterative approach under a fixed randomization parameter \(\blambda\in(0,+\infty)\). Then, we describe the corresponding TD-error algorithm. Finally,
we use a linear benchmark problem to illustrate the performances of this algorithm for different \(\blambda\)'s.

\subsection{Geometric convergence of the policy improvement algorithm}
\Cref{thm:policy_improvement,thm:iter-value-fct} have guaranteed the convergence for the iterative approach introduced in \Cref{subsec:iter-construction} to the value function in Definition~\ref{defn:rand-impulse}. We present the following geometric convergence rate for this approach under the additional standing  Assumption~\ref{assump:boundedness}. The proof is presented in Appendix~\ref{app:a-5}.
\begin{theorem}\label{thm:geom-rate}
There exists some $q\in (0,1)$, such that for any $n\in\NN$,
\[\|\psi^{\blambda}-\psi^{\blambda,n}\|_{L^\infty(\RR)} \leq \frac{q^n}{1-q}\|\psi^{\blambda}-\psi^{\blambda,0}\|_{L^\infty(\RR)}.\]
Therefore, the policy improvement scheme \eqref{eq:policy_improvement_psi0} and \eqref{eq:policy_improvement_psin_value_definition} converge geometrically.
\end{theorem}

\subsection{A Temporal-Difference Error Method}\label{subsec:experiments}

Now, we present a reinforcement learning algorithm for the randomized impulse control problem with a TD loop. In particular, we use a TD scheme to estimate the ``policy
evaluation'' step inside the randomized impulse–control framework. Our code is publicly available at \url{https://github.com/Zhouhao-Yang/A-Two-fold-Randomization-Framework-for-Impulse-Control-Problems}.

\revise{
Throughout this section we work in a simulator-based setting: we assume the structural form of the uncontrolled dynamics and costs, but we do \emph{not} assume that the drift $b$, volatility $\sigma$, or cost functions $f$ and $l$ are known in closed form. Instead, the algorithm only interacts with the environment through simulated trajectories and pointwise cost evaluations. 
All of our updates are therefore expressed in terms of these simulated paths  and cost samples, without requiring explicit knowledge of $b$, $\sigma$, $f$, or $l$.
}

We fix a uniform step size $\Delta t>0$ and horizon
$T=M\Delta t$.  At each grid point $t_i=i\Delta t$, we propagate
\begin{align}
\label{eq:euler_maruyama}
X_{t_{i+1}} &= X_{t_i}+b\left(X_{t_i}\right)\Delta t
           +\sigma\left(X_{t_i}\right)\sqrt{\Delta t}\,\xi_i, \quad
           \xi_i\sim\mathcal N(0,1),\\
\label{eq:survival_weight}
p_{t_{i+1}} &= \bigl(1-\pi(X_{t_i})\Delta t\bigr)\,p_{t_i},
\end{align}
where \eqref{eq:euler_maruyama} is the Euler--Maruyama discretization of the
controlled SDE and \eqref{eq:survival_weight} tracks the {survival
probability} $p_t$, which penalizes premature impulses through the randomized
policy~$\pi$.

Recall that
\[\psi^{\blambda,n+1}(x) =\EE^{x}\left[\int_0^\infty e^{-rs}p_s(f(X_s)+\pi_s\nonlocal^{\lambda_2}\psi^{\blambda,n}(X_s)-\lambda_1\cR(\pi_s))\Big|p_0=1\right].\]
Assume $\psi^{\blambda,n+1}$ is approximated via 
the value network $\psi_\theta$ with parameters $\theta$. Then $\psi_\theta$ ideally satisfies the
one–step Bellman equation
\begin{equation*}
 \begin{aligned}
&p_{t_i}\psi_\theta(X_{t_i}) \\
=&\;
   p_{t_i}\Bigl[f(X_{t_i})
     +\pi_{t_i}\nonlocal^{\lambda_2}\psi^{\blambda,n}(X_{t_i})
     -\lambda_1\cR(\pi_{t_i})\Bigr]\Delta t+e^{-r\Delta t}p_{t_{i+1}}
      \EE\bigl[\psi_\theta(X_{t_{i+1}})| X_{t_i}\bigr].
\end{aligned}   
\end{equation*}
Rearranging and using \eqref{eq:survival_weight} yields the instantaneous TD error
\begin{equation}\label{eq:td_error}
\begin{aligned}
\delta_i(\theta)
&\coloneqq
  \psi_\theta(X_{t_i})
  -e^{-r\Delta t}\bigl(1-\pi_{t_i}\Delta t\bigr)
     \psi_\theta(X_{t_{i+1}})\\
&\hspace{30pt}  -\Bigl[f(X_{t_i})
+\pi_{t_i}\nonlocal^{\lambda_2}\psi^{\blambda,n}(X_{t_i})
     -\lambda_1\mathcal R(\pi_{t_i})\Bigr]\Delta t,
\end{aligned}
\end{equation}
Given $B$ simulated paths
$\{X_{t_i}^b\}_{i=0}^M$ ($b=0,1,\dots,B-1)$,
the empirical loss minimized via stochastic gradient descent is
\begin{equation}\label{eq:td_loss}
\mathcal L(\theta)
  \coloneqq
  \frac{1}{B}\sum_{b=0}^{B-1}\sum_{i=0}^{M-1}
      \bigl[\delta_i^b(\theta)\bigr]^2,\qquad
\delta_i^b\text{ given by }(\ref{eq:td_error})\text{ on path }b.
\end{equation}

Algorithm~\ref{alg:td_impulse} summarizes the complete fixed‑point / TD‑error
routine.  Each outer loop (index~$n$) recomputes the
Monte‑Carlo estimate of $\nonlocal^{\lambda_2}\psi^{\blambda,n}$ and then performs
$K$ inner gradient steps with learning rate $\eta$ to minimize
\eqref{eq:td_loss}.  The next outer loop starts from the last inner
iterate, ensuring monotone policy improvement.
\begin{algorithm}[htbp]
\caption{\textbf{Reinforcement Learning Algorithm with TD-Error for Randomized Impulse Control}}
\label{alg:td_impulse}
\begin{algorithmic}[1]
\REQUIRE
  randomization parameters $\blambda=(\lambda_1,\lambda_2)$; neural network $\psi(\cdot;\cdot)$;
  initial value function $\psi^0$;
  roll-out batch size $B$;
  time step $\Delta t$ and horizon $T=M\Delta t$;
  fixed-point iterations $N$;
  gradient steps per outer loop $K$;
  learning rate $\eta$;
  initial-state distribution $\rho$ on $[x_{\min},x_{\max}]$; SDE drift and diffusion $b$, $\sigma$; entropy function $\cR(\cdot)$.
\STATE Initialize network parameters $\theta_{0}^{0}$ with $\psi(\cdot;\theta_0^0) = \psi^0(\cdot)$.
    \FOR{$b=0, 1,\dots,B-1$}
        \STATE Sample $X_0^b \sim \rho$.
        \STATE Simulate $\{X_{t_i}^b\}_{i=0}^{M}$ via the controlled SDE drift/diffusion.

    \ENDFOR
\FOR{$n=0,1,\dots,N-1$}                             
    \STATE $\psi^{n}(\cdot) \gets \psi(\cdot;\theta_{0}^{n})$.
    \STATE For all $i$, compute $ \nonlocal^{\lambda_2}\psi^n(X_{t_i}^b)$ (Monte Carlo method).
    \FOR{$k=0,1,\dots,K-1$}
        \FOR{all $(b,i)$} \STATE $\displaystyle
            \pi_{t_i}^b = e^{-\frac{p_{t_i}^b - \psi(X_{t_i}^b;\theta_{k}^{n})}{\lambda_1}}$; clip if needed.
    \ENDFOR
        \STATE {\small{$\displaystyle
          \cL(\theta_{k}^{n})
          =\frac{1}{B}\sum_{b=0}^{B-1}\sum_{i=0}^{M-1}\Big[
            \psi(X_{t_i}^b;\theta_{k}^{n})
            -e^{-r\Delta t}(1-\pi_{t_i}^b\Delta t)\psi(X_{t_{i+1}}^b;\theta_{k}^{n})
            -\Delta t\big(f(X_{t_i}^b)+\pi_{t_i}^b\nonlocal^{\lambda_2}\psi^n(X_{t_i}^b)-\lambda_1\cR(\pi_{t_i}^b)\big)
          \Big]^2$.}}
        \STATE $\theta_{k+1}^{n} \gets \theta_{k}^{n} - \eta \nabla_\theta \cL(\theta_{k}^{n})$.
    \ENDFOR
    \STATE $\theta_{0}^{n+1} \gets \theta_{K}^{n}$ . 
\ENDFOR
\STATE \RETURN $\psi^N(\cdot)=\psi(\cdot;\theta_{0}^{N})$.
\end{algorithmic}
\end{algorithm}

\subsection{Experimental Setting}
We assess \cref{alg:td_impulse} on the following one–dimensional impulse–control problem, where the drift and volatility of the state dynamic are constants, and the running cost and transaction cost are linear functions.

Specifically, let \(b(\cdot)\equiv\mu\in\RR\) and \(\sigma(\cdot)\equiv\sigma\in \RR^+\) so that the uncontrolled state dynamic $(X_t)_{t\ge0}$ in \eqref{eq:sde-uncontrolled} satisfies
\[
  \dd X_t = \mu\dd t + \sigma\dd W_t, \qquad \sigma>0,
\]
Let $r>0$ be the discount factor. The running cost $f:\RR\to\RR$ and impulse
cost $l:\RR\to\RR_+$ are given by
\[
  f(x) =
  \begin{cases}
    h x, & x \ge 0,\\[2pt]
    -p x, & x \le 0,
  \end{cases}
  \qquad
  l(\xi) =
  \begin{cases}
    K_+ + k_+ \xi, & \xi \ge 0,\\[2pt]
    K_- - k_- \xi, & \xi < 0,
  \end{cases}
\]
with constants $h,p,K_\pm,k_\pm>0$.  Recall the classical nonlocal operator with test function \(\phi\),
\[
  \nonlocal\phi(x) := \inf_{\xi\in\RR}\{\phi(x+\xi) + l(\xi)\}.
\]
Finally, assume the nontriviality and growth conditions
\begin{equation}\label{eq:Sec5-growth}
  h - r k_- > 0,
  \qquad
  p - r k_+ > 0.
\end{equation}
Here, we specify the following choices of model parameters.
\begin{itemize}
  \item \textbf{Time grid:} we fix a total time horizon \(T = 20\) with step size
        \(\Delta t = 0.01\) (i.e., total time steps \(M = 2000\)).
  \item \textbf{Optimization:} we adopt a total number of outer fixed‑point iterations
        \(N = 30\) with inner gradient steps \(K = 400\); the optimizer \texttt{AdamW} is set with
        learning rate \(\eta = 10^{-3}\) and weight decay \(10^{-4}\).
  \item \textbf{Dynamics:} constant drift is set to be \(\mu=0.03\) and the default volatility is set to be
        \(\sigma=0.2\). The SDE dynamics is sampled via the Euler--Maruyama method:
        \[X_{t+\Delta t} -X_t= \mu \Delta t+\sigma  W_{t,\Delta t}, \quad W_{t,\Delta t}\overset{i.i.d.}{\sim}\cN(0,\Delta t).\]
  \item \textbf{Cost functions:} we use the following cost configuration,\\
        \(
          f(x) =
          \begin{cases}
            x   & x \ge 0\\
            -x  & x < 0
          \end{cases}
        ,\)
        \(
          \ell(x) =
          \begin{cases}
            2 + 0.5x  & x \ge 0\\
            2 - 0.5x  & x < 0
          \end{cases}
        ,\)
        discount rate \(r = 0.1\).
  \item \textbf{Batch size:} $B = 4096$.
  \item \textbf{Randomization:} we vary the randomization parameter 
        \(\blambda = (\lambda_1,\lambda_2)\) within the following choices:
        \(\lambda_1 \in \{1.0,\,0.5,\,0.1,\,0.05\}\) and
        \(\lambda_2 = \lambda_1\); denote \\\(\Lambda\coloneqq \{(1.0, 1.0),(0.5,0.5),(0.1,0.1),(0.05,0.05)\}\).
\end{itemize}

\revise{\paragraph{Classical model-based baseline}
For comparison, we also implement a classical, model-based baseline on the same example.
In this baseline we assume that only discrete observations of the uncontrolled diffusion
$\{X_{t_i}^b\}$ are available, and we first estimate the drift and volatility via the
standard moment estimators
\[
\widehat{b}
= \frac{1}{N\Delta t}\sum_{i=0}^{N-1}(X_{t_{i+1}}-X_{t_i}),
\qquad
\widehat{\sigma}^2
= \frac{1}{N\Delta t}\sum_{i=0}^{N-1}(X_{t_{i+1}}-X_{t_i})^2.
\]
These observations are simulated with the same time step $\Delta t$, horizon $T$ and
number of paths as those used in \cref{alg:td_impulse}, so that the baseline and the randomized
RL method operate under the same sampling budget. Given $(\widehat{b},\widehat{\sigma})$, we address the classical impulse-control problem by iteratively solving the HJB–QVI
\[
\min\bigl\{\hat\cL\psi^n - r\psi^n + f,\; \nonlocal\psi^{n-1} - \psi^n\bigr\} = 0,
\qquad n\in\mathbb{N}^+,
\]
where $\hat\cL$ is the generator associated with $(\widehat{b},\widehat{\sigma})$. We implement this policy-iteration scheme on a spatial grid using a finite-difference discretization of $\hat\cL$ and a projected Gauss–Seidel solver for the resulting obstacle problem.
}

\rerevise{For a fair comparison, both the baseline and Algorithm~\ref{alg:td_impulse} operate
under an identical sampling budget: $B = 4096$ independently simulated trajectories of
length $T/\Delta t = 2000$ steps. For the baseline, these trajectories are used once
to estimate $\hat{b}$ and $\hat{\sigma}$ and to solve the HJB--QVI on a spatial grid.
For Algorithm~\ref{alg:td_impulse}, the same $B$ paths are reused across $N \cdot K$
gradient updates, and no additional data is generated during training.}

\subsection{Experimental Results}
We now present the experimental results under the above setting. We are interested in the performance of \cref{alg:td_impulse} as \(\|\blambda\|_\infty:=\lambda_1\vee\lambda_2\) decreases and whether the convergence results in \Cref{thm:policy_improvement,thm:comparison_convergence} can be observed numerically. Besides, we are also interested in the sensitivity analysis of the volatility parameter \(\sigma\). This parameter modulates the intrinsic stochasticity of the impulse control problem, regardless of being randomized or not, and we want to demonstrate its impact on the corresponding value function as well as the optimal policy under randomization. 

Each figure reports the average over five runs runs with distinct random seeds. All simulations are implemented with \texttt{PyTorch}. 

\subsubsection{Comparison between randomized value functions and their classical counterpart}
In \cref{fig:main_figure}, we present the comparison results between the randomized and the classical value functions, as \(\blambda\) varies in \(\Lambda\). We plot the corresponding value functions in \cref{subfig:value_functions}, where the randomized value functions \(\psi^N\) are the final outputs of \cref{alg:td_impulse} and the classical value function \(\psi\) is the semi-explicit analytical solution as in \cite{constantinides1978existence,guo2009smooth}.  It can be observed that as \(\blambda\) tends to \((0,0)\), the randomized solution
approaches the classical benchmark \(\psi\), which numerically confirms the theoretical
limit
\(
  \psi^{\blambda} \xrightarrow[]{\|\blambda\|_\infty \to 0} \psi
\) as proved in \Cref{thm:comparison_convergence}.

\revise{In addition to the analytic classical value function $\psi$, \cref{subfig:value_functions} also displays a classical, model-based baseline (the curve labeled as ``Baseline''). The baseline curve exhibits a noticeable upward bias relative to the analytic classical solution, both in level and in the location of the continuation region. This bias can be traced to two main sources: (i) the inaccuracy of the HJB--QVI solver as we could only solve for a finite region; and (ii) the lack of robust convergence of the classical fixed-point policy iteration. By contrast, our method circumvents these problems, and the randomized value functions learned by \cref{alg:td_impulse}  remain much closer to the analytic solution across the state space.}

For each \(\blambda\in\Lambda\), we also monitor the evolution of \(L_2\) relative error between the output of each outer iteration \(\psi^n\) in \cref{alg:td_impulse} and the classical value function \(\psi\), namely, \(
  \| \psi^{n} - \psi \|_2 / \| \psi \|_2
\) for \(n=0,\dots,N\). In \cref{subfig:rel_L2_loss}, we observe that for each \(\blambda\), this relative error decreases and finally stabilizes as the training proceeds, and such a trend coincides with the convergence of the iterative approach in \Cref{thm:policy_improvement}. Moreover, we can see that when \(\|\blambda\|_\infty=0.05\), the lowest among \(\Lambda\), the final relative errors reach the lowest level, which once again is a numerical demonstration of \Cref{thm:comparison_convergence}; meanwhile, the evolution under \(\|\blambda\|_\infty=0.05\) displays the largest variance. Empirically, it comes from  the numerical inaccuracy of the exponential term
\(\exp(\cdot/\lambda_1)\) in policy evaluation and \(\exp(\cdot/\lambda_2)\) in Monte-Carlo estimation.
Conversely, larger \(\|\blambda\|_\infty\) may introduce
bias, but it may enhance stability at the same time. This tradeoff partially explains why \(\lambda_1 = 1\) is not the
worst performer.

\captionsetup[subfigure]{labelformat=parens} 
\renewcommand{\thesubfigure}{\alph{subfigure}} 

\begin{figure}[t]
  \centering
  \begin{subfigure}[t]{.48\textwidth}
    \centering
    \includegraphics[width=\linewidth]{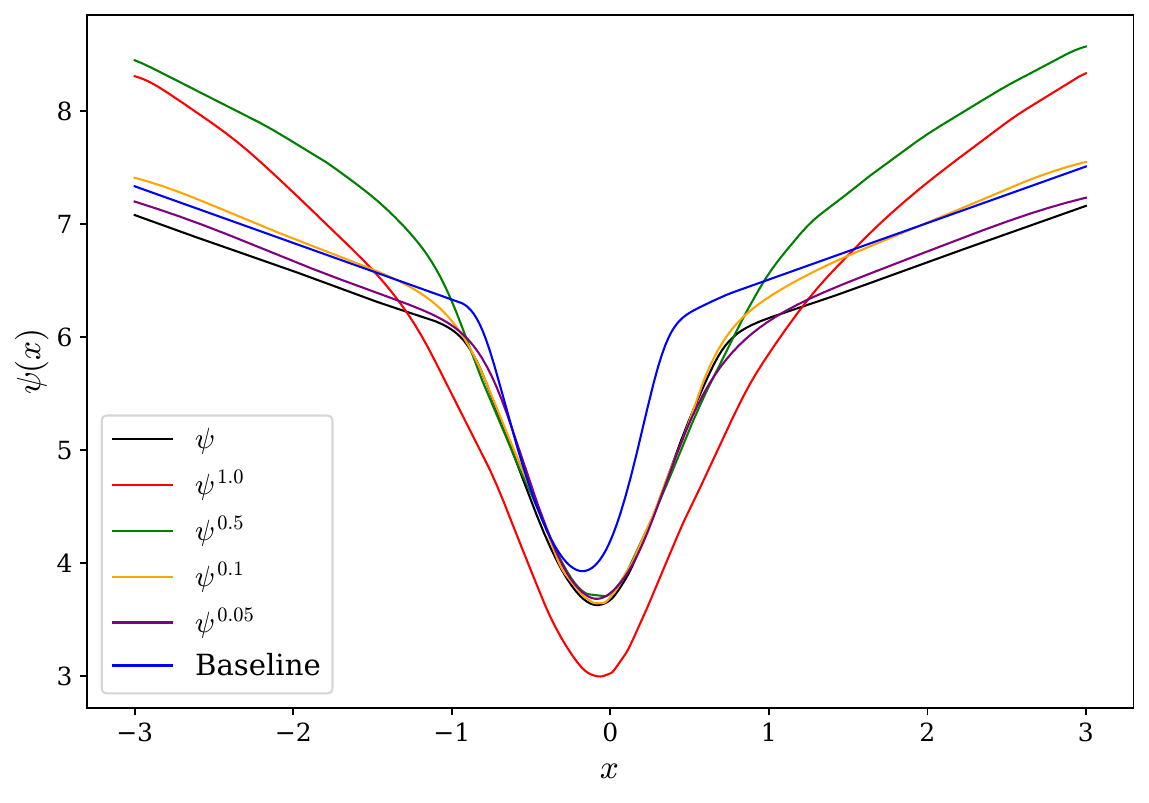}
    \caption{Value functions.}
    \label{subfig:value_functions}
  \end{subfigure}\hfill
  \begin{subfigure}[t]{.48\textwidth}
    \centering
    \includegraphics[width=\linewidth]{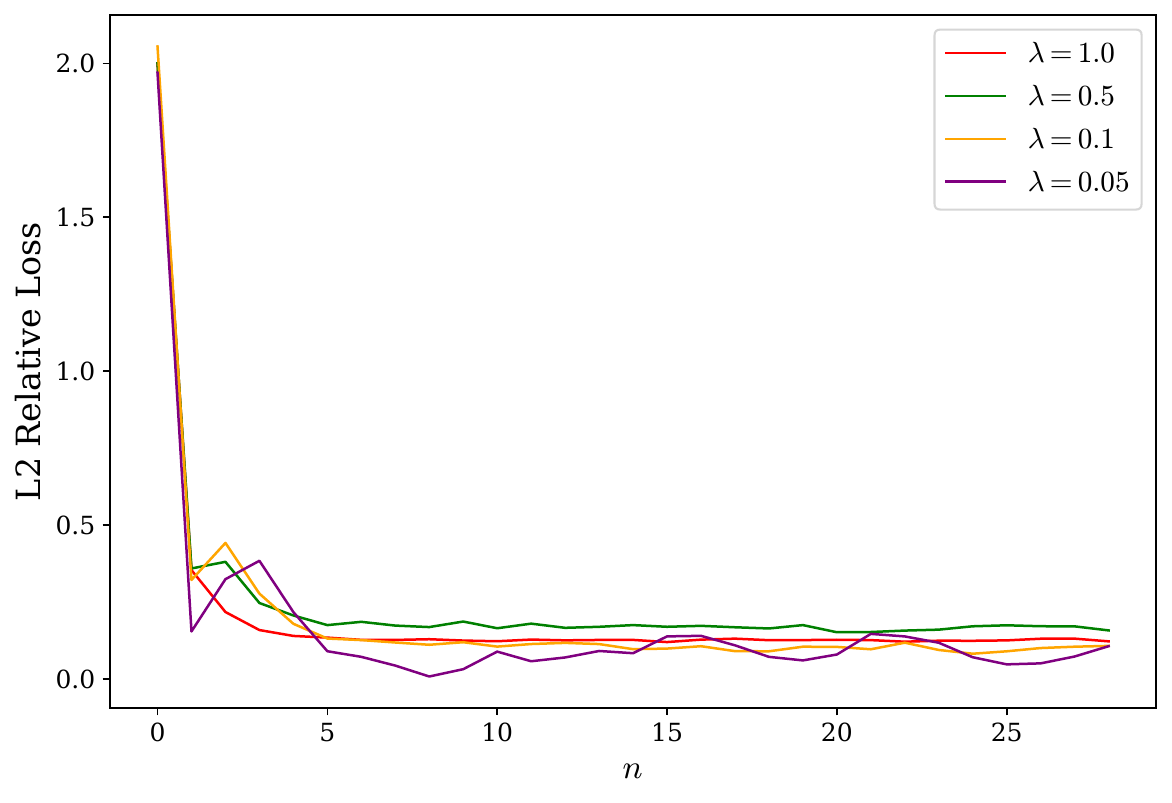}
    \caption{\(L_2\) relative errors evolution in outer iterations.}
    \label{subfig:rel_L2_loss}
  \end{subfigure}
\caption{Comparison in value functions between randomized and classical impulse control with \(\blambda\in\Lambda\).}
  \label{fig:main_figure}
\end{figure}


\subsubsection{Sensitivity Analysis}
We investigate how the diffusion volatility $\sigma$ influences the learning dynamics and the resulting randomized solutions. 
For each $\sigma \in \{0.1, 0.2, 0.3, 0.4\}$, we run \cref{alg:td_impulse} with $\blambda = (0.5,0.5)$ while
keeping all other model parameters and hyperparameters fixed, including the training budget, the optimizer, the step size, the number of Monte Carlo samples for the nonlocal term, and the evaluation grid. Results are demonstrated in \cref{fig:sigma-sensitivity}. 

For the classical impulse control in \cite{constantinides1978existence}, the size of the continuation region \(|\mathfrak{C}|=O\left(\sigma^{\frac23}+\sigma\right)\), and the jump distance \(|\xi^*(x)|=O\left(\sigma^{\frac23}\right)\) for \(x\) on the boundary of the continuation region. That is, as \(\sigma\) increases, the controller intervenes less frequently, but once hitting the action boundary from inside the continuation region, the jump size is larger. Intuitively, with a larger \(\sigma\) value, the controller may make use of the extra randomness when deciding whether to intervene, but may need to compensate for the extra uncertainty by jumping further into the continuation region during each intervention. Consequently, by accumulating more running cost and paying higher cost of intervention, the corresponding value function tends to increase along with \(\sigma\). We can observe similar trends under the randomized setting. For instance, in \cref{subfig:sensitivity_lam_0.5}, we observe that the randomized value function monotonically increases as \(\sigma\) increases. In \cref{subfig:sensitivity_pi_lam_0.5}, we visualize the optimal intensities \(\pi^*\) for the survival processes under different values of \(\sigma\). As \(\sigma\) increases, the flat region where \(\pi^*\) remains close to \(0\) (i.e., it is more likely for the controller to wait) keeps expanding. This observation coincides with the enlarging continuation region in the classical problem. 

Since \(\sigma\) indicates the intrinsic stochasticity of the model, it may require a controller to do more ``exploration'' as \(\sigma\) increases under the randomized setting. In \cref{subfig:sensitivity_pi_lam_0.5}, when \(x\) exits the flat region, the intensity of intervention drops as \(\sigma\) increases, that is, the controller becomes less committed to exploiting the intervention decision. In \cref{subfig:mu_star_-2.0,subfig:mu_star_0.0,subfig:mu_star_2.0}, we plot the optimal jump distributions at \(x\in\{-2,0,2\}\), representing the ``jump to the right'', continuation, and ``jump to the left'' regions in the classical setting, respectively. We can see that, while all the optimal jump distributions \(\mu^*\) concentrate around the corresponding optimal jumps in the classical setting, as \(\sigma\) increases, \(\mu^*\) displays a larger variance, indicating the tendency to conduct more exploration among possible jumps. 

Additional experimental results to compare with the setting $\blambda = (1.0, 1.0)$ can be found in \cite[Appendix B]{cao2025two}.

\begin{figure}[t]
  \centering
  \begin{subfigure}[htb]{.48\textwidth}
    \centering
    \includegraphics[width=\linewidth]{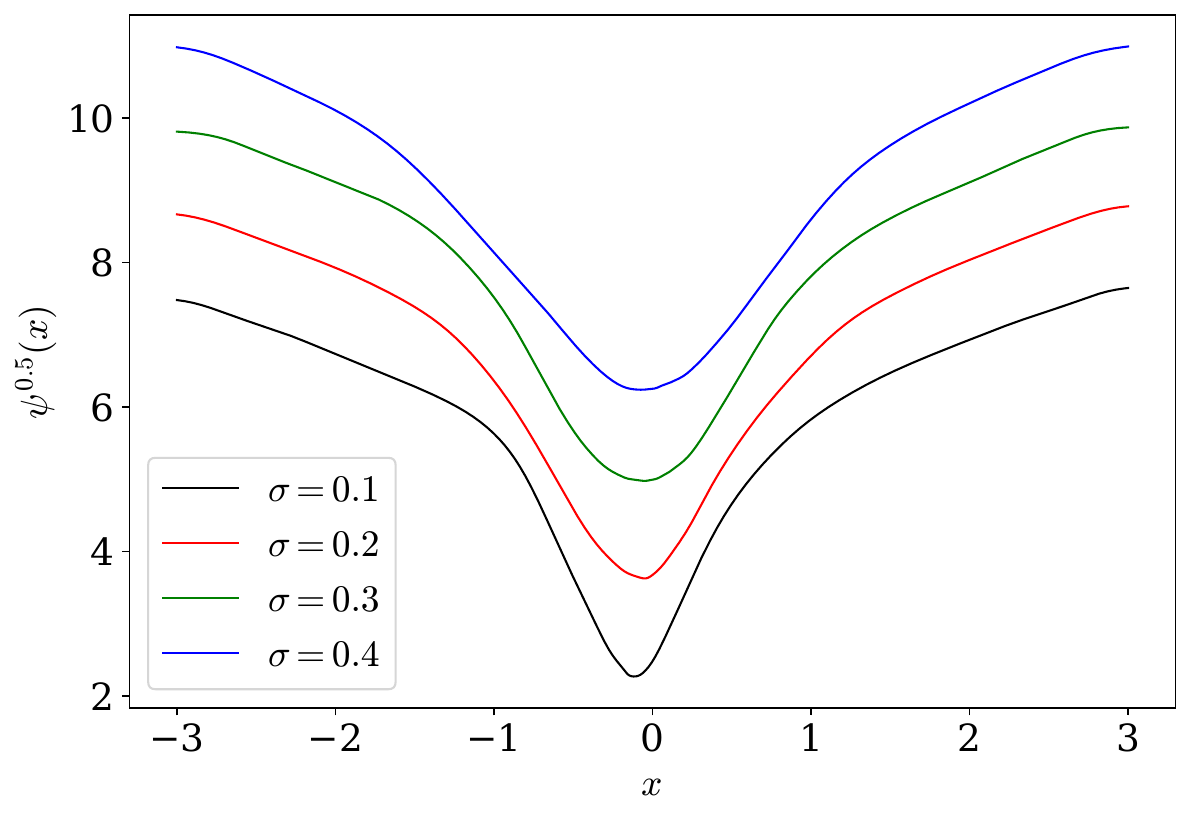}
    \caption{Value functions.}
    \label{subfig:sensitivity_lam_0.5}
  \end{subfigure}
  \hfill
  \begin{subfigure}[htb]{.48\textwidth}
    \centering
    \includegraphics[width=\linewidth]{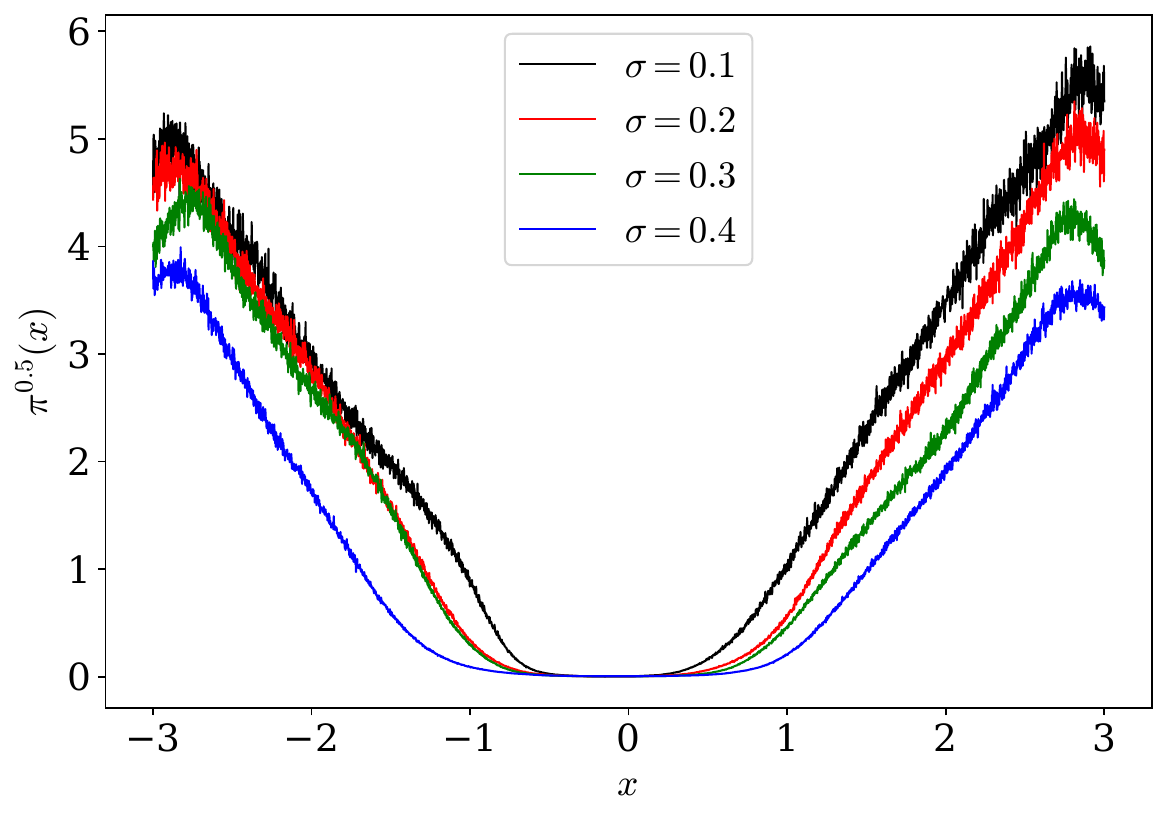}
    \caption{Intensities of interventions.}
    \label{subfig:sensitivity_pi_lam_0.5}
\end{subfigure}
\\
 \begin{subfigure}[htb]{.33\textwidth}
    \centering
    \includegraphics[width=\linewidth]{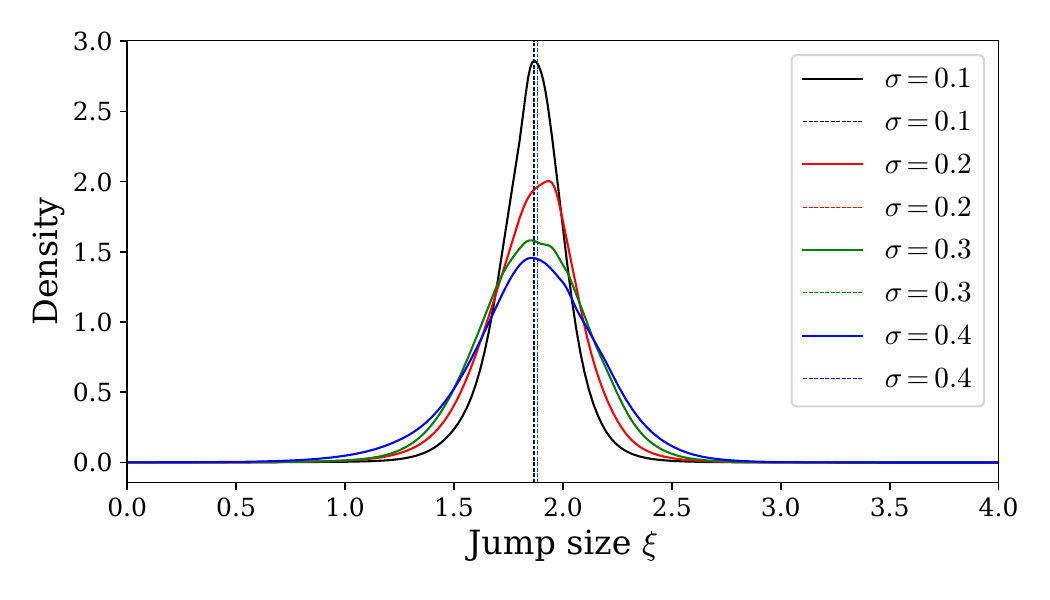}
    \caption{Jumps at $x = -2.0$.}
    \label{subfig:mu_star_-2.0}
  \end{subfigure}\hfill
  \begin{subfigure}[htb]{.33\textwidth}
    \centering
    \includegraphics[width=\linewidth]{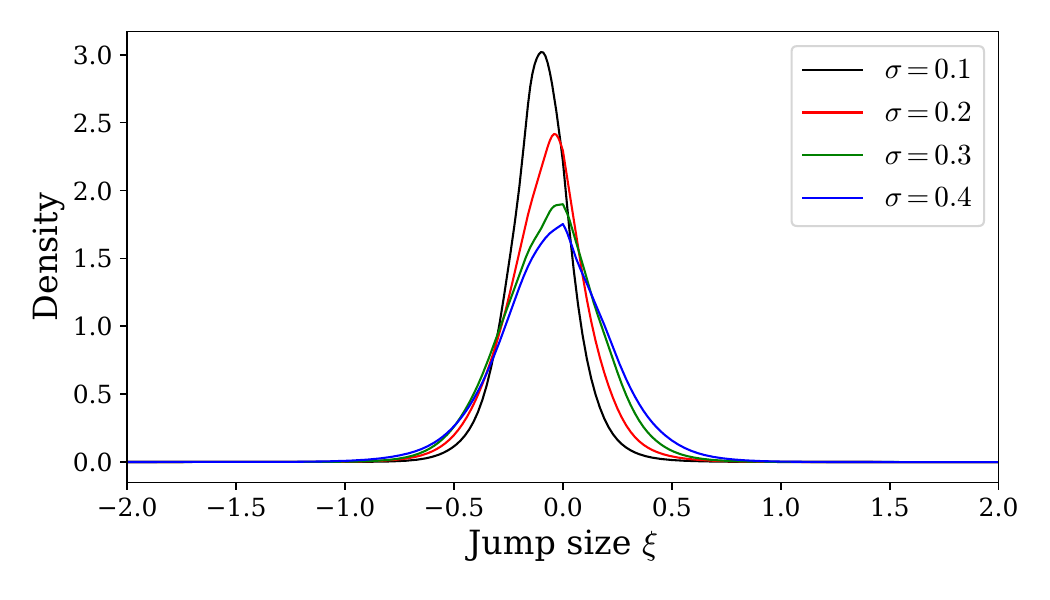}
    \caption{Jumps at $x = 0$.}
    \label{subfig:mu_star_0.0}
  \end{subfigure}\hfill
    \begin{subfigure}[htb]{.33\textwidth}
    \centering
    \includegraphics[width=\linewidth]{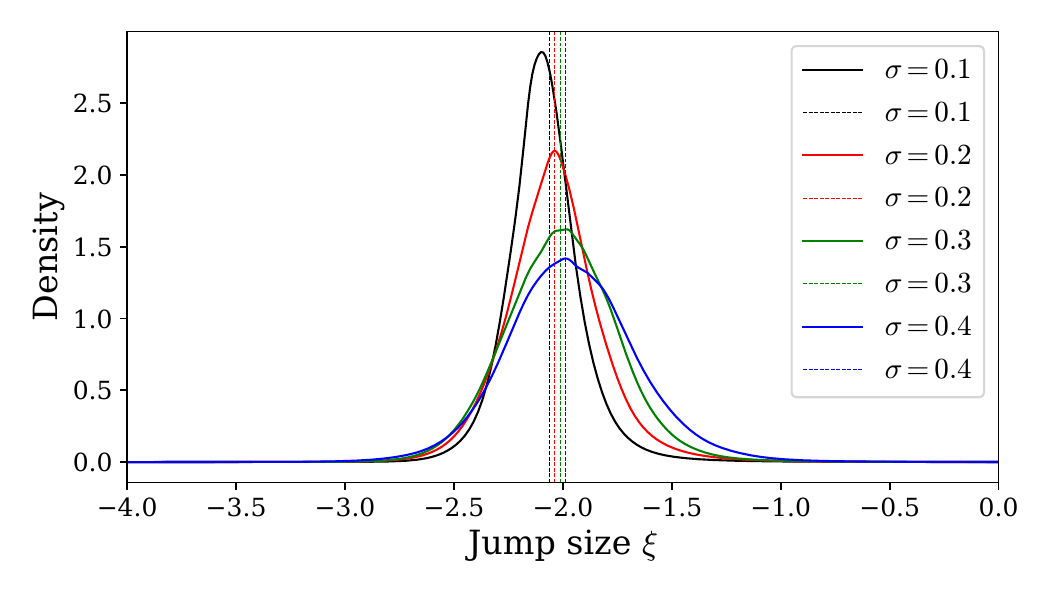}
    \caption{Jumps at $x = 2.0$.}
    \label{subfig:mu_star_2.0}
  \end{subfigure}
  \caption{Sensitivity analysis with respect to volatility $\sigma$ for Algorithm~\ref{alg:td_impulse} with \(\blambda=(0.5,0.5)\). 
  }
  \label{fig:sigma-sensitivity}
\end{figure}

\section{Conclusion}
In this work, we introduce a novel randomization framework for impulse control problems. Through a compound operator \eqref{eq:inf-randomized-impulse-control} consisting of the entropy--regularized versions of the nonlocal and the optimal stopping operators \eqref{eq:rand-nonloc} and \eqref{eq:rand-stp}, we characterized the solution to the randomized problem as a fixed point, which allows us to derive a semi-linear equation \eqref{eq:inf_st_HJB}. We rigorously prove the uniqueness of the solution via a verification theorem, thanks to an equivalent randomization scheme with a compound Poisson measure. We also show the existence of the solution using an iterative approach and prove a \(\cC^{2,\alpha}_{loc}\) regularity for the value function. The existence--uniqueness result thereby ensures the theoretical well-definedness of our randomized problem.

An important finding of this work is the convergence of the randomized solution to its classical counterpart. This, along with the regularity of the value function, shows that our framework is a robust and effective tool for approximating classical impulse control problems. In particular, the iterative approach for the existence proof naturally gives rise to \Cref{alg:td_impulse}, where the policy improvement steps enjoy a geometric convergence rate. Numerical experiments on a well-studied benchmark problem confirm that the model-free \Cref{alg:td_impulse} effectively learns a solution that closely approximates the classical solution. Furthermore, a sensitivity analysis highlights the inherent exploration--exploitation tradeoff within our randomized scheme. 

Ultimately, the randomization framework in this work provides not only a sound theoretical basis but also a practical, implementable method for solving complex impulse control problems.

\paragraph{Limitations and Future Work}
Our two-fold randomization framework and the associated RL algorithm are expected to extend to higher dimensions, while the main additional challenge concerns with the analysis of the nonlocal operator. In higher dimensions, the uncontrolled state process satisfies $X_t \in \mathbb{R}^d$ and the intervention distribution becomes a law on $\mathbb{R}^d \times [0,\infty)$, so that the jump \emph{size} and \emph{direction}, as well as the intervention time, are randomized. The corresponding randomized HJB equation has the same structural form. Under similar structural assumptions on the drift, diffusion, and cost (i.e., Lipschitz continuity and at most linear growth), we expect well-posedness of the randomized problem (i.e., existence, uniqueness, and $C^2$-regularity) to carry over to $d>1$.

One foreseeable additional difficulty arises in the convergence result (Theorem~\ref{thm:comparison_convergence}), where we show that the value function of the randomized problem converges to that of the classical impulse-control problem as the randomization parameter vanishes. Our current proof relies on the one-dimensional setting, in particular the structure of the randomized nonlocal operator and the local estimates in Lemma~\ref{lem:rand-nonloc-attain}, which do not immediately extend to $\mathbb{R}^d$. In higher dimension, one must control both jump directions and geometry-induced effects in the nonlocal term, and hence extending Lemma~\ref{lem:rand-nonloc-attain} and the subsequent convergence argument to higher dimensions would require more delicate estimates for multidimensional operators.



\appendix
\section{Proofs of the Theoretical Results}\label{app:a}
\subsection{Proofs in Section~\ref{sec:prelim}}\label{app:a-1}
\begin{proof}[Proof of Lemma~\ref{lem:rand-nonloc-basic}]
    Notice that for any \(\phi\in\test\), since \(\phi\) is bounded from below and is at most of polynomial growth,
    \[-\infty<\inf_{x\in\RR}\phi(x)+K\leq\nonlocal^\lambda\phi(x)\leq \E[\xi\sim\Phi_x]{\phi(x+\xi)+l(\xi)}<\infty,\quad\forall x\in\RR.\]
    For any \(x\in\RR\) and any \(\epsilon>0\), \(\exists\mu_\epsilon\in\pcls{\RR}\) such that 
    \[\nonlocal^\lambda\phi_2(x)\leq\E[\xi\sim\mu_\epsilon]{\phi_2(x+\xi)+l(\xi)}+\lambda\kl{\mu_\epsilon}{\Phi_x}<\nonlocal^\lambda\phi_2(x)+\epsilon,\]
    and note that \(\phi_1-\phi_2\in L^\infty(\RR)\),
    \[
    \begin{aligned}\nonlocal^\lambda\phi_1(x)-\nonlocal^\lambda\phi_2(x)&<\E[\xi\sim\mu_\epsilon]{\phi_1(x+\xi)+l(\xi)}-\E[\xi\sim\mu_\epsilon]{\phi_2(x+\xi)+l(\xi)}+\epsilon\\
    &\leq\left\|\phi_1-\phi_2\right\|_{L^\infty(\RR)}+\epsilon.\end{aligned}\]
    Similarly, \(\nonlocal^\lambda\phi_2(x)-\nonlocal^\lambda\phi_1(x)<\left\|\phi_1-\phi_2\right\|_{L^\infty(\RR)}+\epsilon\) for any \(x\in\RR\). Thus, 
    \[\left\|\nonlocal^\lambda\phi_1-\nonlocal^\lambda\phi_2\right\|_{L^\infty(\RR)}\leq\left\|\phi_1-\phi_2\right\|_{L^\infty(\RR)}+\epsilon,\quad\forall\epsilon>0. \]
Sending $\epsilon\to0$ yields \(\left\|\nonlocal^\lambda\phi_1-\nonlocal^\lambda\phi_2\right\|_{L^\infty(\RR)}\leq\left\|\phi_1-\phi_2\right\|_{L^\infty(\RR)}\).
    The remaining results follow directly from the definition and particularly the linearity of the operator \(\nonlocal^\lambda\).
\end{proof}

\begin{proof}[Proof of Lemma~\ref{lem:rand-nonloc-attain}]
       As \(\phi\in\test\), \(\nonlocal^\lambda\phi(x)\in\RR\) for all \(x\in\RR\) by \ref{finite}. By \eqref{eq:rand-nonloc}, 
    \begin{equation}
        \label{eq:rand-nonloc-rho}
        \nonlocal^\lambda\phi(x)=\inf_{\rho\geq0,{\E[\zeta\sim\Phi_x]{\rho(\zeta)}=1}}\E[\zeta\sim\Phi_x]{\rho(\zeta)\left(\phi(x+\zeta)+l(\zeta)\right)+\lambda\rho(\zeta)\log\rho(\zeta)},\ \ \forall x\in\RR.
    \end{equation}
    Let us find the infimum of \eqref{eq:rand-nonloc-rho}. Notice that the mapping \((0,+\infty)\ni r\mapsto cr+\lambda r\log{r}\) is convex for any constant \(c\in\RR\), {whose minimum is attained at $r^* = \exp(-\tfrac{c+\lambda}{\lambda})$}. 
  
    For any \(x,\zeta\in\RR\), denote
    \[\tilde\rho_x(\zeta):=e^{-\frac{[\phi(x+\zeta)-\phi(0)]+[l(\zeta)-l(-x)]}{\lambda}},\quad\hat\rho_x(\zeta):=e^{-\frac{\phi(x+\zeta)+l(\zeta)}{\lambda}}=e^{-\frac{\phi(0)+l(-x)}{\lambda}}\tilde\rho_x(\zeta).\]
    Since \(\phi,l\in\lips\), 
    \[\tilde  Z(x)=\E[\zeta\sim\Phi_x]{\tilde\rho_x(\zeta)}=\E[\zeta\sim\cN(0,1)]{e^{-\frac{[\phi(\zeta)-\phi(0)]+[l(\zeta-x)-l(-x)]}{\lambda}}}\in(0,+\infty),\quad\forall x\in\RR.\]
    For any \(x\in\RR\), with \(\hat Z(x)=\E[\zeta\sim\Phi_x]{\hat\rho_x(\zeta)}=e^{-\frac{\phi(0)+l(-x)}{\lambda}}\tilde {Z}(x)\), we must have
    \[
    \begin{aligned}\nonlocal^\lambda\phi(x)&=\E[\zeta\sim\Phi_x]{\rho_x(\zeta)\left(\phi(x+\zeta)+l(\zeta)\right)+\lambda\rho_x(\zeta)\log\rho_x(\zeta)}\\
    &=-\lambda\log\E[\zeta\sim\Phi_x]{\exp\left\{-\frac{\phi(x+\zeta)+l(\zeta)}{\lambda}\right\}},\end{aligned}\]
    where \(\rho_x=\frac{\hat\rho_x}{\hat Z(x)}=\frac{\tilde\rho_x}{\tilde {Z}(x)}\). Denoting the Lipschitz constant of \(\phi\) as \(L_\phi>0\), for any \(x\in\RR\) we have
    \[\begin{aligned}&e^{-\frac{L_\phi+L_l}{\lambda}|x+\zeta|}\leq\tilde\rho_x(\zeta)\leq e^{\frac{L_\phi+L_l}{\lambda}|x+\zeta|},\quad\forall\zeta\in\RR;\\
    &\tilde{Z}(x)\geq\underline{Z}\coloneqq\E[\zeta\sim\cN(0,1)]{e^{-\frac{L_\phi+L_l}{\lambda}|\zeta|}}>0.\end{aligned}\]
    Take \(\mu_x\in\pcls{\RR}\) such that \(\dd{\mu_x}=\rho_x\dd{\Phi_x}\). By Markov's inequality, for any \(t\geq0\), we have
    \[\begin{aligned}\mu_x\left([-x-t,-x+t]^\complement\right)&\leq2\frac{e^{\left(\frac{L_\phi+L_l}{\sqrt{2}\lambda}\right)^2}}{\underline{Z}}\int_{t}^{\infty}\frac{e^{-\frac{1}{2}\left(x-\frac{L_\phi+L_l}{\lambda}\right)^2}}{\sqrt{2\pi}}\dd{\zeta}\\
    &\leq 2\frac{e^{-\frac{t^2}{2}+\frac{L_\phi+L_l}{\lambda}t}}{\underline{Z}}\leq2\frac{e^{\frac{1}{2-4\kappa}\left(\frac{L_\phi+L_l}{\lambda}\right)^2}}{\underline{Z}}e^{-\kappa t^2},\end{aligned}\]
    for any \(\kappa\in\left(0,\frac{1}{2}\right)\). Thus, given \(\kappa_2\in\left(0,\frac{1}{2}\right)\) and \(\kappa_1\geq\frac{e^{\frac{1}{2-4\kappa_2}\left(\frac{L_\phi+L_l}{\lambda}\right)^2}}{\underline{Z}}\), \(\mu_x\in\subG{}\) as in \eqref{eq:sub-gaussian}. Notice that the choices of \((\kappa_1,\kappa_2)\) primarily depend on the Lipschitz constants \(L_\phi\) and \(L_l\).

Moreover,
    by \eqref{eq:rand-nonloc-value},
    \[\nonlocal^\lambda\phi(x)=-\lambda\log{\E[\zeta\sim\cN(0,1)]{\exp\left\{-\frac{\phi(\zeta)+l(\zeta-x)}{\lambda}\right\}}},\quad x\in\RR.\]
    For any \(x,y\in\RR\), since \(l\in\lips[L_l]\), 
    \[l(\zeta-y)\leq l(\zeta-x)+L_l|y-x|\text{ and }l(\zeta-x)\leq l(\zeta-y)+L_l|x-y|,\quad\forall\zeta\in\RR.\]
    Therefore,
    \begin{align*}
        \nonlocal^\lambda\phi(y)&=-\lambda\log{\E[\zeta\sim\cN(0,1)]{\exp\left\{-\frac{\phi(\zeta)+l(\zeta-y)}{\lambda}\right\}}}\\
        &\leq -\lambda\log{\E[\zeta\sim\cN(0,1)]{\exp\left\{-\frac{\phi(\zeta)+l(\zeta-x)+L_l|y-x|}{\lambda}\right\}}}\\
        &=\nonlocal^\lambda\phi(x)+L_l|y-x|,
    \end{align*}
    and similarly, \(\nonlocal^\lambda\phi(x)\leq \nonlocal^\lambda\phi(y)+L_l|x-y|\). Thus, \(\nonlocal^\lambda\phi\in\lips[L_l]\).

Furthermore,
    Since \(\phi_1,\phi_2\in\lips\), there exists \(L_\phi>0\) such that \(\phi_1,\phi_2\in\lips[L_\phi]\). Following the arguments above, there exist \(\kappa_1,\kappa_2>0\) and \(\mu_{1,\cdot},\mu_{2,\cdot}\in\subG{\cdot}\) as in \eqref{eq:sub-gaussian} such that
    \[\nonlocal^\lambda\phi_i(x)=\E[\xi\sim\mu_{i,x}]{\phi_i(x+\xi)+l(\xi)}+\lambda\kl{\mu_{i,x}}{\Phi_x},\quad\forall x\in\RR,\,i\in\{1,2\}.\]
    Define \(\Delta_R\coloneqq \left(R+\frac{1}{2\kappa_2R}\right)\delta_R\) for any \(R>0\). Observe that \(\lim_{R\to\infty}\Delta_R=0\). For any \(x\in\RR\) and any \(\nu\in\subG{}\), 
    \begin{equation}\label{eq:sub-gaussian-2}
    \EE_{\xi\sim\nu}\left[|x+\xi|\mathbbm{1}\{|x+\xi|>R\}\right]\leq \Delta_R,\quad\forall R>0.
\end{equation}
Then, for any \(x\in\RR\) and \(R>0\),
\begin{align*}
    &\nonlocal^\lambda\phi_1(x)-\nonlocal^\lambda\phi_2(x)\leq\E[\xi\sim\mu_{2,x}]{\phi_1(x+\xi)-\phi_2(x+\xi)}\\
    =&\E[\xi\sim\mu_{2,x}]{\left(\mathbbm{1}\left\{|x+\xi|\leq R\right\}+\mathbbm{1}\left\{|x+\xi|>R\right\}\right)\left(\phi_1(x+\xi)-\phi_2(x+\xi)\right)}\\
    \leq& \mu_{2,x}\left([-R-x,R-x]\right)\|\phi_1-\phi_2\|_{L^\infty([-R,R])}\\
    &\hspace{15pt}+ \mu_{2,x}\left([-R-x,R-x]^\complement\right)|\phi_1(0)-\phi_2(0)|+2L_\phi\E[\xi\sim\mu_{2,x}]{|x+\xi|\mathbbm{1}\left\{|x+\xi|>R\right\}}\\
    \leq& \|\phi_1-\phi_2\|_{L^\infty([-R,R])}+2L_\phi\Delta_R;
\end{align*}
similarly, \(\nonlocal^\lambda\phi_2(x)-\nonlocal^\lambda\phi_1(x)\leq\|\phi_1-\phi_2\|_{L^\infty([-R,R])}+2L_\phi\Delta_R\),
and therefore,
\begin{equation}
    \label{eq:loc-diff-nonlocal}
    \left\|\nonlocal^\lambda\phi_1-\nonlocal^\lambda\phi_2\right\|_{L^\infty(\RR)}\leq\|\phi_1-\phi_2\|_{L^\infty([-R,R])}+2L_\phi\Delta_R.
\end{equation}
The conclusion follows with \(\mathscr{E}_R=2L_\phi\Delta_R\), \(\forall R>0\).
\end{proof}

\begin{proof}[Proof of \ref{lem:rand-nonloc-monotone}]
    Fix any $0<\lambda<\lambda'$. We know that \(\exists\mu^\lambda,\mu^{\lambda'}\in\pcls{\RR}\) such that 
    \[\nonlocal^\gamma\phi(x)=\E[\xi\sim\mu^\gamma]{\phi(x+\xi)+l(\xi)}+\gamma\kl{\mu^\gamma}{\Phi_x},\quad\forall x\in\RR,\]
    with \(\gamma\in\{\lambda,\lambda'\}\). Thus,
    \[\nonlocal^\lambda\phi\leq\E[\xi\sim\mu^{\lambda'}]{\phi(x+\xi)+l(\xi)}+\lambda\kl{\mu^{\lambda'}}{\Phi_x}\leq\nonlocal^{\lambda'}\phi(x),\quad\forall x\in\RR.\]
\end{proof}

\begin{proof}[Proof of \ref{lem:compare-classical}]
    For any \(y\in\RR\), \(\delta_y\in\pcls{\RR}\). Thus, 
    \[\nonlocal\phi(x)=\inf_{\xi}[\phi(x+\xi)+l(\xi)]=\inf_{\mu\in\pcls{\RR}}\int_{\RR}\phi(x+\xi)+l(\xi)\mu(\dd{\xi}).\]
    Since \(\{\mu\in\pcls{\RR}:\mu\ll\Phi_x\}\subset\pcls{\RR}\) and \(\kl{\mu}{\Phi_x}\geq0\) for all \(x\in\RR\), the conclusion follows.
\end{proof}

\begin{proof}[Proof of \ref{lemma:nonlocal_convergence}]
By \ref{lem:rand-nonloc-monotone} and \ref{lem:compare-classical}, it suffices to show that for any \(x\in \RR\) and any \(u>\nonlocal\phi(x)\), there exists \( \lambda_0>0\) such that \(\nonlocal^{\lambda_0}\phi(x)\leq u\). Denote \(\mathcal I=\{\xi\in\RR:\phi(x+\xi)+l(\xi)<u\}\). By continuity of \(\phi\) and \(l\), \(\mathcal{I}\neq\emptyset\) and \(\exists [a,b]\in\mathcal I\) with \(-\infty<a<b<+\infty\). Let \(\ell=\max_{\xi\in[a,b]}\phi(x+\xi)+l(\xi)\), then \(\ell<u\). Denote \(c=\frac{a+b}{2}\) and \(\delta=b-c\). By Lipschitz condition of \(\phi\) and \(l\), \(\exists\bar{L}>0\) such that
\(\phi(x+\xi)+l(\xi)<\ell+\bar{L}|\xi-c|,\,\forall\xi\not\in[a,b]\).

Let \(\sigma=\frac{u-\ell}{2\bar L}\sqrt{\frac{\pi}{2}}>0\) and \(\Phi_\sigma=\cN(c,\sigma)\). Then \(\EE_{\xi\sim\Phi_\sigma}[\phi(x+\xi)+l(\xi)]\leq \ell+L\sigma\sqrt{\frac{2}{\pi}}=\frac{u+\ell}{2}\). Then, for \(\lambda_0=\frac{u-\ell}{2}\left(-\log{\sigma}+\frac{\sigma^2+(c+x)^2}{2}\right)^{-1}>0\),
\[{\lambda_0 D_{\mathrm{KL}}(\Phi_\sigma\Vert\Phi_x)
=\lambda_0\,
   \left(-\log\sigma + \frac{ \sigma^2 + (c+x)^2-1}{2}\right)
\leq \frac{u-\ell}{2}.}\]
And hence,
\[\nonlocal^{\lambda_0}\phi(x)\leq \EE_{\xi\sim\Phi_\sigma}[\phi(x+\xi)+l(\xi)]+\lambda_0 D_{KL}(\Phi_\sigma\|\Phi_x)\leq u,\]
which implies \(\lim_{\lambda\to0^+}\nonlocal^\lambda\phi(x)=\nonlocal\phi(x)\) for any \(x\in K\).

From \cite[Lemma 2.2-(3)]{guo2009smooth}, we know that \(\nonlocal\phi\in\lips\). Then the conclusion follows from Dini's theorem.
\end{proof}

\begin{proof}[Proof of Lemma~\ref{lem:rand-stp-lips}]
    Fix any \(x,y\in\RR\). By Lemma~\ref{lem:l1-error-propagation}, for any intensity process \(\pi=\{\pi_t\}_{t\geq0}\),
    \begin{align*}
        &\left|J^{\lambda,\phi}(x;\pi)-J^{\lambda,\phi}(y;\pi)\right|\\
       \leq &\E{\int_0^\infty e^{-rt}p_t\left(\left|f(X^x_t)-f(X^y_t)\right|+\pi_t\left|\phi(X^x_t)-\phi(X^y_t)\right|\right)\dd{t}}\\
        \leq& \E{\int_0^\infty e^{-rt}p_t\left(L_f+\pi_tL_\phi\right)\left|X^x_t-X^y_t\right|\dd{t}}\\
        \leq& \E{\int_0^\infty e^{-(r-G)t}p_t\left(L_f+\pi_tL_\phi\right)\dd{t}}|x-y|\\
       \leq &\max\left\{\frac{L_f}{r-G},L_\phi\right\}|x-y|\E{\int_0^\infty e^{-(r-G)t}p_t\left(r-G+\pi_t\right)\dd{t}}\\
       =&\max\left\{\frac{L_f}{r-G},L_\phi\right\}|x-y|\E{\int_0^\infty \dd \left(e^{-(r-G)t}p_t\right)}\\
       = &\max\left\{\frac{L_f}{r-G},L_\phi\right\}|x-y|.
    \end{align*}
    For any \(\epsilon>0\), there exists an intensity process \(\pi^y=\{\pi^y_t\}_{t\geq0}\) such that 
    \[T^\lambda\phi(y)\leq J^{\lambda,\phi}(y;\pi^y)<T^\lambda\phi(y)+\epsilon.\]
    Denoting \(p^y_t=\exp\left\{-\int_0^t\pi^y_s\dd{s}\right\}\) for all \(t\geq0\), we have
    \begin{align*}
        T^\lambda\phi(x)-T^\lambda\phi(y)&< J^{\lambda,\phi}(x;\pi^y)-J^{\lambda,\phi}(y;\pi^y)+\epsilon\leq\max\left\{\frac{L_f}{r-G},L_\phi\right\}|x-y|+\epsilon;
    \end{align*}
    Similarly, \(T^\lambda\phi(y)-T^\lambda\phi(x)<\max\left\{\frac{L_f}{r-G},L_\phi\right\}|x-y|+\epsilon\). Thus,
    \[|T^\lambda\phi(x)-T^\lambda\phi(y)|<\max\left\{\frac{L_f}{r-G},L_\phi\right\}|x-y|+\epsilon,\quad\forall\epsilon>0.\]
    Sending $\epsilon\to 0$  closes the proof.
\end{proof}

Theorems~\ref{thm:reg-hjb} and \ref{thm:reg-verification} that we are about to introduce will immediately imply Theorem~\ref{thm:classical-control-sol}. 
\begin{theorem}
    \label{thm:reg-hjb}
    Suppose \(\phi\in\lips\) with lower bound \(l_\phi\geq0\). There exists a solution \(V^{\lambda,\phi}\) to \eqref{eq:hjb-regular-test} such that \(V^{\lambda,\phi}\in\CCal_{loc}^{2,\alpha}(\RR)\) for any \(\alpha\in(0,1)\), and \(\lim\limits_{|x|\to\infty}\left|\frac{V^{\lambda,\phi}(x)}{x}\right|<\infty\). 
\end{theorem}
\begin{proof}
    First, we fix any \(N,M>0\). Let \(\Omega_N=(-N,N)\) and \(\phi_M\) be a smooth cut-off function such that \(\phi(x)=e^{x}\) when \(x\leq M\) and \(\phi_M^{x}=e^{M+1}\) when \(x>M+1\). Hence, \(\phi_M\) is bounded and Lipschitz. 

    Then, we consider the following elliptic equation,
    \begin{equation}
        \label{eq:hjb-N}
        \begin{cases}
            (\cL-r)v+f-\lambda\phi_M\left(-\frac{\phi-v}{\lambda}\right)=0,&x\in\Omega_N,\\
            v(x)=\phi(x),&x\not\in\Omega_N,
        \end{cases}
    \end{equation}
    for \(v:\RR\to\RR\). Let \(B(x,y,z)=b(x)z-ry-\lambda\phi_M\left(-\frac{\phi(x)-y}{\lambda}\right)+f(x)\) for any \((x,y,z)\in\bar\Omega_N\times\RR\times\RR\). We can rewrite \eqref{eq:hjb-N} into\begin{equation}
        \label{eq:hjb-quasilinear}\frac{\sigma(x)^2}{2}v_N''(x)+B\left(x,v_N,v_N'\right)=0.
    \end{equation}
    Together with the standing Assumptions~\ref{ass:dynamic} and \ref{assump:basic}, we can apply existence theorems for quasi-linear elliptic equations (see \cite[Theorem 15.10.]{gilbarg2013elliptic} for instance) as well as the comparison principle (see \cite[Theorem 10.1]{gilbarg2013elliptic} for instance) to conclude that \eqref{eq:hjb-N} admits a unique classical solution \(v_N\in\CCal_{loc}^{2,\alpha}\cap\CCal\) for any \(\alpha\in(0,1)\). Notice that on \(\bar\Omega_N\times\RR\times\RR\),
    \[-{\rm sign}(y)\lambda\phi_M\left(-\frac{\phi(x)-y}{\lambda}\right)\leq\begin{cases}
        0,&\text{if }y\geq0;\\
        \lambda e^{-\frac{l_\phi}{r}},&\text{otherwise},
    \end{cases}\]
    and therefore,
    \[B(x,y,z){\rm sign}(y)=f(x){\rm sign}(y)-r|y|-{\rm sign}(y)\lambda\phi_M\left(-\frac{\phi(x)-y}{\lambda}\right)+b(x)|y|.\]
    Thus,
    \[\frac{B(x,y,z){\rm sign}(y)}{\frac{\sigma(x)^2}{2}z^2}\leq \frac{\sigma_0^2}{2}\frac{|z|\cdot\sup_{x\in\Omega_N}b(x)+[\lambda e^{-\frac{l_\phi}{\lambda}}+\sup_{x\in\Omega_N}f(x)]}{z^2}.\]
    Notice that the constants \(C_1=\frac{\sigma_0^2}{2}\cdot\sup_{x\in\Omega_N}b(x)\) and \(C_2=\frac{\sigma_0^2}{2}[\lambda e^{-\frac{l_\phi}{\lambda}}+\sup_{x\in\Omega_N}f(x)]\geq0\) depend only on \((L,\sigma_0,L_f,l_\phi,\lambda,N)\). By maximum principle \cite[Theorem 10.3]{gilbarg2013elliptic}, \(v_N\) is bounded uniformly for all \(M>0\). Therefore, given \(M\) sufficiently large, \(v_N\) satisfies \eqref{eq:hjb-regular-test} on \(\Omega_N\).

    Now, choosing any \(q>1\) and any \(\Gamma\Subset\Omega_N\) and applying interior estimate (for example, \cite[Theorem 9.11.]{gilbarg2013elliptic}), we have 
    \[\|v_N\|_{\cW^{2,q}(\Gamma)}\leq c\left[\|v_N\|_{L^q(\Gamma)}+\|\lambda e^{-\frac{\phi-v_N}{\lambda}}\|_{L^q(\Gamma)}\right],\]
    where \(\|v_N\|_{L^q(\Gamma)}=\left(\int_{\Gamma}|v_N(x)|^q\dd x\right)^{\frac{1}{q}}\), and \(\|v_n\|_{\cW^{2,q}(\Gamma)}=\left(\sum_{k=0}^2\|v_N^{(k)}\|_{L^{q}(\Gamma)}^q\right)^{\frac{1}{q}}\), with some constant \(c>0\). By the boundedness of \(v_N\) on \(\Omega_N\), we have that 
    \[\|v_N\|_{\cW^{2,q}(\Gamma)}\leq c'\|v_N\|_{L^\infty(\Gamma)},\]
    with some constant \(c'>0\). By Sobolev embedding theorem (see \cite[Section 5, Theorem 6]{Evans1998} for instance), we have that 
    \[\|v_N\|_{\fcls[1,\alpha]{\Gamma_N}}\leq c''\|v_N\|_{\cW^{2,q}(\Gamma_N)},\]
    for any \(\alpha\leq1-\frac{1}{q}\). Thus, \(\{v_N\}_{N\in\NN^+}\) is equicontinuous locally. Thus, upon taking subsequences, \(\{v_N\}_{N\in\NN^+}\) as a locally uniform limit \(V^{\lambda,\phi}\), which is also a weak limit in \(\cW^{2}_{loc}\). Thus, \(V^{\lambda,\phi}\in\CCal_{loc}^{2,\alpha}\) for any \(\alpha\in(0,1)\), since \(q>0\) can be arbitrarily large. Finally, it is straightforward to verify that \(V^{\lambda,\phi}\) satisfies \eqref{eq:hjb-regular-test} with at most linear growth. 
\end{proof}

\begin{theorem}
    \label{thm:reg-verification}
    Suppose \(v\in\CCal^2\) is a classical solution to \eqref{eq:hjb-regular-test} with at most linear growth. Then, \(v=T^{\lambda}\phi\).
\end{theorem}
\begin{proof}
    For any intensity process \(\pi=\{\pi_t\}_{t\geq0}\) and the corresponding \(p=\{p_t=e^{-\int_0^t\pi_s\dd{s}}\}_{t\geq0}\), applying It\^{o}'s formula to \(f(t,X_t)\) subject to \eqref{eq:sde-uncontrolled} with \(X_0=x\in\RR\) with \(f(t,y)=e^{-rt}p_tv(y)\) for all \(t\geq0\) and \(y\in\RR\), we have
    \[e^{-rt}p_tv(X_t)=v(x)+\int_0^te^{-rs}p_s\left(\cL-r-\pi_s\right)v(X_s)\dd{s}+\int_0^te^{-rs}p_sv'(X_s)\sigma(X_s)\dd{W_s}.\]
    Since \(v\) solves \eqref{eq:hjb-regular-test},
    and note that \(\{\int_0^te^{-rs}p_sv'(X_s)\sigma(X_s)\dd{W_s}\}_{t\geq0}\) is a martingale, we have by taking expectation
    \[v(x)\leq \E[x]{\int_0^te^{-rs}p_s\left[f(X_s)-\lambda\cR(\pi_s)+\pi_s\phi(X_s)\right]\dd{s}}+\E[x]{e^{-rt}p_tv(X_t)}.\]
    By \ref{discount}, \(\lim_{t\to\infty}\E[x]{e^{-rt}p_tv(X_t)}=0\). Therefore, for any \(\pi\),
    \[v(x)\leq J^{\lambda,\phi}(x;\pi),\ \ \forall x\in\RR\implies v\leq T^{\lambda}\phi.\]
    Let \(\bar\pi=\exp\left\{-\frac{\phi-v}{\lambda}\right\}\) and \(\pi^*_t=\bar\pi(X_t)\) for all \(t\geq0\). Then we have
    \[v(x)=J^{\lambda,\phi}(x;\pi^*),\ \ \forall x\in\RR\implies v=T^{\lambda}\phi.\]
\end{proof}

\begin{proof}[Proof of \ref{lem:rand-lips}]
    Notice that \(\nonlocal^{\lambda_2}\phi\in\lips[L_l]\). By \ref{discount}, \(\frac{L_f}{r-G}=\max\left\{\frac{L_f}{r-G}, L_l\right\}\). The conclusion follows from Lemma~\ref{lem:rand-stp-lips}.
\end{proof}

\begin{proof}[Proof of \ref{lem:global-rand-nonexp}]
    Without loss of generality, assume that \(\|\phi_1-\phi_2\|_{L^\infty(\RR)}<\infty\). Fix any \(x\in\RR\). By \ref{global-non-exp}, for any intensity process \(\pi\), 
    \begin{align*}
        \left|J^{\lambda_1,\nonlocal^{\lambda_2}\phi_1}(x;\pi)-J^{\lambda_1,\nonlocal^{\lambda_2}\phi_1}(x;\pi)\right|
        &\leq\E[x]{\int_0^\infty e^{-rt}p_t\pi_t\left|\nonlocal^{\lambda_2}\phi(X_t)-\nonlocal^{\lambda_2}\phi(X_t)\right|\dd{t}}\\
        &\leq\|\phi_1-\phi_2\|_{L^\infty(\RR)} \E{\int_0^\infty e^{-rt}p_t\pi_t\dd{t}}\\
        &\leq \|\phi_1-\phi_2\|_{L^\infty(\RR)}.
    \end{align*}
    Similar to the argument in the proof of Lemma~\ref{lem:rand-stp-lips}, we have
    \[|\stp^{\blambda}\phi_1(x)-\stp^{\blambda}\phi_2(x)|\leq \|\phi_1-\phi_2\|_{L^\infty(\RR)}.\]
    The conclusion follows.
\end{proof}

\subsection{Proofs in Section~\ref{sec:HJB-vrf}}\label{app:a-2}
\begin{proof}[Proof of Theorem~\ref{thm:fx-pnt-hjb-sol}]
    We know that \(V\in\lips\) with a finite lower bound. By Theorem~\ref{thm:classical-control-sol}, \(\nonlocal^{\lambda_2}V\in\lips\) with a finite lower bound. Recall that \(\stp^{\blambda}V=T^{\lambda_1}[\nonlocal^{\lambda_2}V]\), hence \(\stp^{\blambda}V\) is a unique classical solution in \(\CCal^2\cap\lips\) to 
    \begin{equation}
        \label{eq:temp-hjb}
        (\cL-r)v+f-\lambda_1\exp\left\{-\frac{\nonlocal^{\lambda_2}V-v}{\lambda_1}\right\}=0,
    \end{equation}
    for \(v:\RR\to\RR\). Since \(V=\stp^{\blambda}V\), we must have \(V\in\CCal^2\) and it satisfies in a classical sense that
    \[(\cL-r)V+f-\lambda_1\exp\left\{-\frac{\nonlocal^{\lambda_2}V-V}{\lambda_1}\right\}=0.\]
    The conclusion follows.
\end{proof}

\begin{proof}[Proof of Lemma~\ref{lem:equiv-hjb}]
    Following the classical dynamical programming principle argument in, for instance, \cite{oksendal2007applied}, we have the following the HJB equation associated with \eqref{eq:mdp-jump-diff}, where for any \(x\in\RR\),
    \begin{equation}
        \label{eq:mdp-hjb-1}
        \begin{aligned}
        0&=(\cL-r)v(x)+f(x)+\inf_{\pi\geq0;\,\rho\geq0,\EE_{\zeta\sim\Phi_{x},\rho(\zeta)}=1}\\
    &\hspace{40pt} \big\{\pi(x)\big[\EE_{\zeta\sim\Phi_x}\big[\rho(\zeta)\big(v(x+\zeta)+l(\zeta)+\lambda_2\log{\rho(\zeta)}\big)-v(x)\big]-\lambda_1\cR(\pi(x))\big\},
        \end{aligned}
    \end{equation}
    for \(v:\RR\to\RR\). By linearity in \(\pi\), \eqref{eq:mdp-hjb-1} is equivalent to
    \begin{equation}
        \label{eq:mdp-hjb-2}
        \begin{aligned}
        0&=(\cL-r)v(x)+f(x)+\inf_{\pi\geq0}\Bigg\{\pi(x)\cdot\\
        &\inf_{\rho\geq0,\E[\zeta\sim\Phi_{x}]{\rho(\zeta)}=1}\left[\E[\zeta\sim\Phi_x]{\rho(\zeta)\left(v(x+\zeta)+l(\zeta)+\lambda_2\log{\rho(\zeta)}\right)}-v(x)\right]-\lambda_1\cR(\pi(x))\Bigg\}\\
        &=(\cL-r)v(x)+f(x)+\inf_{\pi\geq0}\left\{\pi(x)\left[\nonlocal^{\lambda_2}v(x)-v(x)\right]-\lambda_1\cR(\pi(x))\right\},
        \end{aligned}
    \end{equation}
    where the last equality is by \eqref{eq:rand-nonloc-rho}. The equivalence with \eqref{eq:inf_st_HJB} follows from the concavity of \(\cR\) on \([0,+\infty)\).
\end{proof}

\begin{proof}[Proof of Theorem~\ref{thm:rand-verification}]
    For any \((\pi,\bmu)\in\cA\), applying It\^{o}'s formula (e.g., \cite[Theorem 1.14.]{oksendal2007applied}) to \(g(t,X_t)\) subject to \eqref{eq:jump-diff}, with \(g(t,y)=e^{-rt}V(y)\) for any \((t,y)\in[0,+\infty)\times\RR\), we have
    \[\begin{aligned}
        e^{-rt}V(X_t)&=V(x)+\int_0^t\hat{\cL}^{\pi,\nu}[e^{-rs}V(X_{s-})]\dd{s}+\int_0^te^{-rs}V'(X_{s-})\sigma(X_{s-})\dd{W_s}\\
        &=V(x)+\int_0^te^{-rt}\left\{(\cL-r)V(X_{s-})+\pi_s\E[\xi\sim\mu_t]{V(X_{s-}+\xi)-V(X_{s-})}\right\}\dd{s}\\
        &\hspace{60pt}+\int_0^te^{-rs}V'(X_s)\sigma(X_{s-})\dd{W_s},\quad\forall t\geq0.
    \end{aligned}\]
    Since \(V\) is a classical solution to \eqref{eq:inf_st_HJB}, by Lemma~\ref{lem:equiv-hjb}, we have that 
    \[\begin{aligned}
        e^{-rt}V(X_t)&\geq V(x)+\int_0^te^{-rs}V'(X_s)\sigma(X_{s-})\dd{W_s}\\
        &\hspace{10pt}-\int_0^t e^{-rs}\left\{f(X_{s-})-\lambda_1\cR(\pi_s)+\pi_s\left[\E[\xi\sim\mu_s]{l(\xi)}-\lambda_2\cH(\mu_s;X_{s-})\right])\right\}\dd{s}.
    \end{aligned}\]
    Since \(\{\int_0^te^{-rs}V'(X_s)\sigma(X_{s-})\dd{W_s}\}_{t\geq0}\) is a martingale, taking expectation on both sides of the inequality, we have
    \[\begin{aligned}
      V(x)&\leq \E[x]{\int_0^t e^{-rs}\left\{f(X_{s-}-\lambda_1\cR(\pi_s)+\pi_s\left[\E[\xi\sim\mu_s]{l(\xi)}-\lambda_2\cH(\mu_s;X_{s-})\right])\right\}\dd{s}}\\
      &\hspace{50pt}+\E[x]{e^{-rt}V(X_t)},\quad\forall t\geq0.  
    \end{aligned}\]
    By \ref{discount}, \(\lim_{t\to\infty}\E[x]{e^{-rt}V(X_t)}=0\). Therefore, for any \((\pi,\bmu)\in\cA\),
    \[V(x)\leq J^{\blambda}(x;\pi,\bmu),\ \ \forall x\in\RR\implies V\leq \tilde{\psi}^{\lambda}.\]
    Since \(V\in\lips\) with a lower bound, by Theorem~\ref{thm:classical-control-sol}, \(\nonlocal^{\lambda_2}V(x)\) is attained at \(\mu_x^*\) such that \(\rho^*_x=\frac{d\mu^*_x}{d\Phi_x}=\frac{\exp\left\{-\frac{V(x+\xi)+l(\xi)}{\lambda_2}\right\}}{\E[\zeta\sim\Phi_x]{\exp\left\{-\frac{V(x+\zeta)+l(\zeta)}{\lambda_2}\right\}}}\). Denote \(\bar\pi=\exp\left\{-\frac{\nonlocal^{\lambda_2}V-V}{\lambda_1}\right\}\). Let \(\pi^*_t=\bar\pi(X_{t-})\), and \(\mu^*_t=\mu^*(X_{t-})\) for all \(t\geq0\). By Lemma~\ref{lem:equiv-hjb} again, we have
    \[V(x)=J^{\blambda}(x;\pi^*,\bmu^*), \ \ \forall x\in\RR\implies V=\tilde{\psi}^{\blambda}.\]
\end{proof}

\subsection{Proofs in Section~\ref{sec:existence}}
\begin{proof}[Proof of Lemma~\ref{lemma:uniform_lipschitz}]
Fix any \(x,x'\in\RR\). By Lemma~\ref{lem:l1-error-propagation}, for \(\psi^{\blambda,0}\), we have 
\[\begin{aligned}\left|\psi^{\blambda,0}(x)-\psi^{\blambda,0}(x')\right|&\leq \EE\left[\int_0^\infty e^{-rt}|f(X_t)-f(X_t')|\dd t\right] \\
&\leq L_f\EE\left[\int_0^\infty e^{-rt}|X_t-X_t'|\dd t\right]\le\frac{L_f}{r-G}|x-x'|.\end{aligned}\]
Therefore, $\psi^{\blambda,0}\in\lips[\frac{L_f}{r-G}]$. Suppose by induction that \(\psi^{\blambda,n-1}\in\lips[\frac{L_f}{r-G}]\) for \(n\in\NN^+\). {By \Cref{lem:rand-nonloc-attain} and \ref{discount}, $L_l\leq \frac{L_f}{r-G}$, and hence \(\nonlocal^{\lambda_2}\psi^{\blambda,n-1}\in\lips[\frac{L_f}{r-G}]\).} Furthermore, one can conclude that \(\psi^{\blambda,n}=T^{\lambda_1}[\nonlocal^{\lambda_2}\psi^{\blambda,n-1}]\in\lips[\frac{L_f}{r-G}]\) based on Lemma~\ref{lem:rand-stp-lips}, which closes the proof.
\end{proof}

\begin{proof}[Proof of Theorem~\ref{thm:policy_improvement}]
By definition, for any \(x\in\RR\),
\begin{align*}
    \psi^{\blambda,1}&=\inf_{\pi}\EE^x\left[\int_0^\infty e^{-rt}p_t\left(f(X_t)-\lambda_1\mathcal R(\pi_t)+\pi_t\nonlocal^{\lambda_2}\psi^{\blambda,0}(X_t)\right)\dd t\right]\\
   &\leq J^{\lambda_1,\nonlocal^{\lambda_2}\psi^{\blambda,0}}(x;\pi^0)= \EE^x\left[\int_0^\infty e^{-rt}f(X_t)\dd t\right]=\psi^{\blambda,0}(x).
\end{align*}
Suppose by induction that $\psi^{\blambda,n}\leq \psi^{\blambda,n-1}$ holds on $\RR$ for \(n\in\NN^+\). By \eqref{eq:policy_improvement_psin_value_definition}, for any $x\in\RR$,
\[
\begin{aligned}
\psi^{\blambda,n+1}(x)&=\inf_\pi\EE\left[\int_0^\infty e^{-rt} p_t\left(f(X_t)+\nonlocal^{\lambda_2}\psi^{\blambda,n} (X_t) \pi_t-\lambda_1\cR(\pi_t) \right)\dd t \Big| X_0=x\right] \\
&\leq \EE\left[\int_0^\infty e^{-rt} p^n_t\left(f(X_t)+\nonlocal^{\lambda_2}\psi^{\blambda,n} (X_t) \pi^n_t-\lambda_1\cR(\pi^n_t) \right)\dd t \Big| X_0=x\right]\\
&\leq \EE\left[\int_0^\infty e^{-rt} p^n_t\left(f(X_t)+\nonlocal^{\lambda_2}\psi^{\blambda,n-1} (X_t) \pi^n_t-\lambda_1\cR(\pi^n_t) \right)\dd t \Big| X_0=x\right]\\
&=\psi^{\blambda,n}(x),
\end{aligned}\]
where the first inequality comes from the sub-optimality of $\pi^n$, and the second inequality is by \ref{increase}. The proof by induction is complete. 

With a natural lower bound \(0\), the existence of a pointwise limit \(\hat{\psi}^{\blambda}\geq0\) directly follows from monotone convergence theorem. To see the Lipschitz property, we fix any \(x,y\in\RR\). For any \(\epsilon>0\), there exists \(N\in\NN^+\) such that \(|\hat{\psi}^{\blambda}(x)-{\psi}^{\blambda,N}(x)|+|\hat{\psi}^{\blambda}(y)-{\psi}^{\blambda,N}(y)|<\epsilon\). By Lemma~\ref{lemma:uniform_lipschitz}, we have
\[\begin{aligned}
|\hat{\psi}^{\blambda}(x)-\hat{\psi}^{\blambda}(y)|&\leq|\hat{\psi}^{\blambda}(x)-{\psi}^{\blambda,N}(x)|+|{\psi}^{\blambda,N}(x)-{\psi}^{\blambda,N}(y)|+|\hat{\psi}^{\blambda}(y)-{\psi}^{\blambda,N}(y)|\\&<\frac{L_f}{r-G}|x-y|+\epsilon.    
\end{aligned}\]
Since \(\epsilon>0\) and \(x,y\in\RR\) are arbitrarily fixed, \(\hat{\psi}^{\blambda}\in\lips[\frac{L_f}{r-G}]\).

Moreover, fix any compact set $K\subset\RR$. By Dini's theorem, monotone pointwise convergence of continuous functions to a continuous limit is uniform.
Hence,
\[
\sup_{x\in K}|\psi^{\blambda,n}(x)-\hat{\psi}^{\blambda}(x)|\xrightarrow[]{n\to\infty}0.
\]
Because $K$ was arbitrary, the convergence is locally uniform on $\RR$.
\end{proof}

\begin{proof}[Proof of Theorem~\ref{thm:iter-value-fct}]
    For \(\hat{\psi}^{\blambda}\), \(0\) is a natural lower bound. From Theorem~\ref{thm:policy_improvement}, we know that \(\hat{\psi}^{\blambda}\in\lips[\frac{L_f}{r-G}]\). {Following a similar line of proof as \Cref{lem:rand-nonloc-attain}}, we know that for each \(x\in\RR\), \(\nonlocal^{\lambda_2}\hat{\psi}^{\blambda}\) is attained at and only at \(\mu^*_x\). 
    
    By Theorem~\ref{thm:classical-control-sol}, \(\nonlocal^{\lambda_2}\hat{\psi}^{\blambda}\in\lips[\frac{L_f}{r-G}]\), and the following optimal control problem,
    \[\mathcal T^{\blambda}\hat\psi^{\blambda}(x)=\inf_{\pi}\EE^x\left[\int_0^\infty e^{-rt}p_t\left(f(X_t)-\lambda_1\mathcal R(\pi_t)+\pi_t\nonlocal^{\lambda_2}\hat\psi^{\blambda}(X_t)\right)\dd t\right],\quad\forall x\in\RR,\]
    subject to \eqref{eq:sde-uncontrolled} and \eqref{eq:survival} with \(p_0=1\), is well-posed. In particular, \(\mathcal T^{\blambda}\hat\psi^{\blambda}\) is the unique \(\cC^2\) solution to  
    \begin{equation}
        \label{eq:hjb-fx-pnt}
        (\cL-r)u+f+\inf_{\pi\geq0}\left\{\pi\left(\nonlocal^{\lambda_2}\hat{\psi}^{\blambda}-u\right)-\lambda_1\cR(\pi)\right\}=0,
    \end{equation}
    and the problem admits the optimal control policy \(\pi^*=\{\pi^*_t\}_{t\geq0}\) such that for any \(t\geq0\), \(\pi^*_t=\hat\pi(X_t)\), where 
    \[\hat\pi(x)=\exp\left\{-\frac{\nonlocal^{\lambda_2}\hat\psi^{\blambda}(x)-\mathcal T^{\blambda}\hat\psi^{\blambda}(x)}{\lambda_1}\right\},\quad\forall x\in\RR;\]
    denote \(\hat p_t=\exp\left\{-\int_0^t\hat\pi_sds\right\}\) for all \(t\geq0\). By Theorem~\ref{thm:policy_improvement} and \eqref{eq:loc-diff-nonlocal} with \(L_\phi=\frac{L_f}{r-G}\), for any \(n\in\mathbb N\) and any \(R>0\), 
    \begin{align*}
        & \hat\psi^{\blambda}(x)-\mathcal T^{\blambda}\hat\psi^{\blambda}(x)\leq\psi^{\blambda,n+1}(x)-\mathcal T^{\blambda}\hat\psi^{\blambda}(x)\\
        \leq& \;\EE^x\left[\int_0^\infty e^{-rt}\hat p_t\left(f(X_t)-\lambda_1\mathcal R(\hat\pi_t)+\hat\pi_t\nonlocal^{\lambda_2}\psi^{\blambda,n}(X_t)\right)\dd t\right]-\mathcal T^{\blambda}\hat\psi^{\blambda}(x)\\
        =&\;\EE^x\left[\int_0^\infty e^{-rt}\hat p_t\hat\pi_t\left(\nonlocal^{\lambda_2}\psi^{\blambda,n}(X_t)-\nonlocal^{\lambda_2}\hat\psi^{\blambda}(X_t)\right)\dd t\right]\\
        \leq& \;\left(\|\psi^{\blambda,n}-\hat\psi^{\blambda}\|_{L^\infty(\Omega_R)}+\frac{2L_f}{r-G}\Delta_R\right)\EE^x\left[\int_0^\infty e^{-rt}\hat p_t\hat\pi_t\dd t\right]\\
        \leq&\; \|\psi^{\blambda,n}-\hat\psi^{\blambda}\|_{L^\infty(\Omega_R)}+\frac{2L_f}{r-G}\Delta_R.
    \end{align*}
    Fix an arbitrary \(\epsilon>0\). By \eqref{eq:sub-gaussian} and \eqref{eq:sub-gaussian-2}, there exists \(R_\epsilon>0\) such that \(\Delta_{R_\epsilon}<\frac{r-G}{4L_f}\epsilon\). By Theorem~\ref{thm:policy_improvement}, there exists \(N_\epsilon\in\mathbb N\) such that \( \|\psi^{\blambda,n}-\hat\psi^{\blambda}\|_{L^\infty(\Omega_{R_\epsilon})}<\frac{\epsilon}{2}\) for all \(n>N_\epsilon\). Thus,
    \[\forall \epsilon>0,\, \hat\psi^{\blambda}(x)-\mathcal T^{\blambda}\hat\psi^{\blambda}(x)<\epsilon\implies \hat\psi^{\blambda}(x)\leq\mathcal T^{\blambda}\hat\psi^{\blambda}(x).\]
    By Theorem~\ref{thm:policy_improvement} and \ref{lem:rand-increasing}, we have that for any \(x\in\RR\),
    \[\forall n\in\mathbb N,\,\mathcal T^{\blambda}\hat\psi^{\blambda}(x)\leq\mathcal T^{\blambda} \psi^{\blambda,n}(x)=\psi^{\blambda,n+1}(x)\implies\mathcal T^{\blambda}\hat\psi^{\blambda}(x)\leq\hat\psi^{\blambda}(x).\]
    Thus, we must have \(\hat\psi^{\blambda}=\mathcal T^{\blambda}\hat\psi^{\blambda}\) on \(\RR\). Replacing \(\stp^{\blambda}\hat{\psi}^{\blambda}\) by \(\hat{\psi}^{\blambda}\) in \(\hat\pi\), we get \(\bar\pi\).
\end{proof}

\subsection{Proofs in Section~\ref{sec:rand2classical}}
\begin{proof}[Proof of Lemma~\ref{lemma:uniform_lipschitz_classical_iterates}]
Fix any \(x,x'\in\RR\). By Lemma~\ref{lemma:uniform_lipschitz}, $\psi^{0}\equiv\psi^{\blambda,0}\in\lips[\frac{L_f}{r-G}]$.

Inductively, suppose that $\psi^{n-1}\in \lips[\frac{L_f}{r-G}]$. Define the cost functional for optimal stopping problem at step $n$ corresponding to \eqref{eq:policy_improvement_psin_value_definition_classical}: For any $x\in\RR$, $\tau$ as $\mathbb{F}$-stopping time,
\[J^n(x;\tau):=\E[x]{\int_0^\tau e^{-rt} f(X_t)\dd t+e^{-r\tau}\nonlocal\psi^{n-1} (X_\tau) }.\]
Then by \cite[Lemma 2.2-(3)]{guo2009smooth}, $\nonlocal\psi^{n-1}\in \lips[\frac{L_f}{r-G}]$. Subsequently, using Lemma~\ref{lem:l1-error-propagation}, the following holds for any $\tau$:
\[
\begin{aligned}
&|J^n(x;\tau)-J^n(x';\tau) |
\\
\leq& {\EE\left[\sup_{s\geq 0} e^{-Gs} |X_s-X_s'|\right]\cdot \left(\int_0^\tau e^{-(r-G)s}L_f\dd s + e^{-(r-G)\tau}\frac{L_f}{r-G}\right)
\leq \frac{L_f}{r-G}|x-x'|}.
\end{aligned}\]
By arbitrariness of $\tau$, $\psi^n\in \lips[\frac{L_f}{r-G}]$.
\end{proof}

\begin{proof}[Proof of Proposition~\ref{prop:bdd}]
Fix any $x\in\RR$.
Setting $\tau'= \infty$ in the definition of $\psi^n$ and $\tilde\psi^n$, we have
\begin{equation*}
\begin{aligned}
0&\leq \psi^n(x) = \inf_\tau\EE_x\left[\int_0^\tau e^{-rt}f(X_t)\dd t+ e^{-r\tau}\nonlocal\psi^{n-1}(X_\tau)\right]\\
&\leq \EE_x\left[\int_0^\infty e^{-rt}f(X_t)\dd t\right]\leq \frac{\|f\|_{L^\infty(\RR)}}{r}\leq \frac{M}{r}.
\end{aligned}
\end{equation*}
Similarly, replacing $\psi^n$ by $\tilde\psi^n$, and $\nonlocal\psi^{n-1}$ by $\nonlocal^{\lambda_2}\psi^{n-1}$ in the above inequality, we have $
0\leq \tilde\psi^n(x)\leq \frac{M}{r}$.

Setting $\pi' \equiv 0$, $p'\equiv 1$ in the definition of $\psi^{\blambda,n+1}$ and $\stp^{\blambda}\psi^n$, we have
\begin{equation*}
\begin{aligned}
0&\leq\psi^{\blambda,n+1}(x)=\stp^{\blambda}\psi^{\blambda,n}(x) \leq \EE_x\left[\int_0^\infty e^{-rt}f(X_t)\dd t\right]\leq \frac{\|f\|_{L^\infty(\RR)}}{r}\leq \frac{M}{r}.
\end{aligned}
\end{equation*}
\end{proof}

\begin{proof}[Proof of Lemma~\ref{lemma:semi-randomized_nonlocal_bound}]
For any $n\in \NN^+$, $x\in K$, let $\tau^*(x)$ be an optimal stopping time for $\psi^n(x)$ as defined in \eqref{eq:policy_improvement_psin_value_definition_classical}.
Then either $\tau^*(x)=0$, or $x\in\mathfrak{C}^n$ and
$\tau^*(x)=\inf\{t\ge0:X_t\notin\mathfrak{C}^n\}$, hence
\(X_{\tau^*(x)} \in K\).
Using $\tau^*(x)$ as an admissible (possibly suboptimal) time for $\tilde\psi^{\,n-1}(x)$,
\[\begin{aligned}
\tilde\psi^{n-1}(x)
&\le \EE_x\left[\int_0^{\tau^*(x)} e^{-rt} f(X_t)\dd t + e^{-r\tau^*(x)}\nonlocal^{\lambda_2}\psi^{n-1}(X_{\tau^*(x)})\right]\\
&= \psi^n(x) + \EE_x\left[e^{-r\tau^*(x)}\left(\nonlocal^{\lambda_2}\psi^{n-1}(X_{\tau^*(x)})-\nonlocal\psi^{n-1}(X_{\tau^*(x)})\right)\right]\\
&\leq \psi^n(x)+\left\|\nonlocal\psi^{n-1}-\nonlocal^{\lambda_2}\psi^{n-1}\right\|_{L^\infty(K)}.
\end{aligned}
\]
Meanwhile, for any $x\in\RR$,  we know $\psi^n\leq\tilde{\psi}^{n-1}$ by substituting the optimal stopping time associated with $\tilde\psi^{n-1}$ into the value function definition of $\psi^n$, and using $\nonlocal\psi^n\leq\nonlocal^{\lambda_2}\psi^n$. This closes the proof.
\end{proof}

\begin{proof}[Proof of Proposition~\ref{propn:inter_bound_nonlocal_stp}]
Fix $x\in\RR$, then
with bounded $b,\sigma$ by $M$, It\^o's formula and BDG's inequality give
\[
\EE_{x}|X_t-x|
 \le \|b\|_{L^\infty(\RR)}t+\|\sigma\|_{L^\infty(\RR)}\sqrt{t}
 \le M(t+\sqrt{t}).
\]
Subsequently, by $\tfrac{L_f}{r-G}$--Lipschitzness of $\nonlocal^{\lambda_2}\psi^n$,
\begin{equation}\label{eq:lipschitz_nonlocal_bound}
\bigl|\nonlocal^{\lambda_2}\psi^{n}(X_t)
       -\nonlocal^{\lambda_2}\psi^{n}(x)\bigr|\le
\frac{L_f}{r-G}|X_t-x|
\le \frac{L_fM}{r-G}(t+\sqrt{t}).
\end{equation}
{Recall that  $\stp^{\blambda}\psi^n$, by the definition of the randomized compound optimal stopping operator $\stp^{\blambda}$, solves the problem 
\[\mathcal T^{\blambda}\psi^{n}(x)=\inf_{\pi}\EE_x\left[\int_0^\infty e^{-rt}p_t\left(f(X_t)-\lambda_1\mathcal R(\pi_t)+\pi_t\nonlocal^{\lambda_2}\psi^{n}(X_t)\right)\dd t\right],\quad\forall x\in\RR,\]
    subject to \eqref{eq:sde-uncontrolled} and \eqref{eq:survival} with \(p_0=1\). Moreover, \(\mathcal T^{\blambda}\psi^{n}\) is the unique \(\cC^2\) solution to  the HJB equation
    \[
        (\cL-r)u+f-\lambda_1\exp\left(-\frac{\nonlocal^{\lambda_2}\psi^n-u}{\lambda_1}\right)=0.
    \]}
Taking the constant suboptimal control $\bar\pi = \frac{1}{\lambda_1^2}$ for the problem $\stp^{\blambda}\psi^n(x)$ yields
{\begin{equation}\label{eq:randomized_stp_nonlocal_bound}
\begin{aligned}
&\stp^{\blambda}\psi^{n}(x) \\
\leq&\; \EE_x\left[\int_0^\infty e^{-(r+\lambda_1^{-2})}\left(f(X_t)+\frac{1}{\lambda_1^2}\nonlocal^{\lambda_2}\psi^{n}(X_t)-\lambda_1\cR\left(\frac{1}{\lambda_1^{2}}\right)\right)\dd t\right]\\
\leq &\; \EE_x\bigg[\int_0^\infty e^{-(r+\lambda_1^{-2})}\bigg(\frac{1}{\lambda_1^2}\left(\nonlocal^{\lambda_2}\psi^{n}(x)+\nonlocal^{\lambda_2}\psi^{n}(X_t)-\nonlocal^{\lambda_2}\psi^{n}(x)\right)+M\bigg)\dd t\bigg]\\
&\qquad \qquad+\frac{\lambda_1\left(1-2\log(\lambda_1^{-1})\right)}{1+r\lambda_1^{2}}\\
\leq&\; \nonlocal^{\lambda_2}\psi^n(x)+M\lambda_1^2+\EE_x\left[\int_0^\infty e^{-(r+\lambda_1^{-2})t}\left(\frac{1}{\lambda_1^2}\left(\nonlocal^{\lambda_2}\psi^{n}(X_t)-\nonlocal^{\lambda_2}\psi^{n}(x)\right)\right)\dd t\right]\\
&\qquad \qquad+C\lambda_1\log\frac{1}{\lambda_1}\\
\leq & \; \nonlocal^{\lambda_2}\psi^n(x)+M\lambda_1^2+ \frac{CL_{f}M}{r-G}\lambda_1+C\lambda_1\log\frac{1}{\lambda_1},
\end{aligned}
\end{equation}}
where the second inequality comes from $\|f\|_{L^\infty(\RR)}\leq M$ by Assumption~\ref{assump:boundedness}, the third inequality is due to $\int_0^\infty e^{-(r+\lambda_1^{-2})t}\dd t=\frac{\lambda_1^2}{1+r\lambda_1^2}\leq C\lambda_1^2$, and the fourth inequality can be derived by:
{\[\begin{aligned}
&\EE_x\left[ \int_0^\infty e^{-(r+\lambda_1^{-2})t}\frac{1}{\lambda_1^2}\left(\nonlocal^{\lambda_2}\psi^{n}(X_t)- \nonlocal^{\lambda_2}\psi^{n}(x)\right)
\dd t\right]\\
\leq&
 \int_0^\infty e^{-(r+\lambda_1^{-2})t}\frac{1}{\lambda_1^2}\frac{L_fM}{r-G}(t+\sqrt{t})
\dd t\\
=&\frac{1}{\lambda_1^2}\frac{L_fM}{r-G}\left(\frac{1}{(r+\lambda_1^{-2})^2}+\frac{\sqrt{\pi}}{(r+\lambda_1^{-2})^{3/2}}\right)\\
\le& \frac{CL_fM}{r-G}
\lambda_1
\end{aligned}\]
for some constant $C>0$. }

Finally, it can be concluded from \eqref{eq:randomized_stp_nonlocal_bound} that for sufficiently small $\lambda_1>0$, we have   \(\stp^{\blambda}\psi^{n}(x)\le
  \nonlocal^{\lambda_2}\psi^{n}(x)
  +C\lambda_1\log\frac{1}{\lambda_1}.\)
\end{proof}

\begin{proof}[Proof of Proposition~\ref{propn:randomized_os_convergence}]
{Recall that $\stp^{\blambda}\psi^{n}$  solves \eqref{eq:value_improvement} with obstacle function $\nonlocal^{\lambda_2}\psi^n$:    
\[
        (\cL-r)\stp^{\blambda}\psi^{n}+f-\lambda_1\exp\left(-\frac{\nonlocal^{\lambda_2}\psi^n-\stp^{\blambda}\psi^{n}}{\lambda_1}\right)=0.
    \]
Meanwhile, $\tilde{\psi}^{n}$ solves  \eqref{eq:additional_HJB}:
\[\min\{\cL\tilde{\psi}^n-r\tilde{\psi}^n+f,\nonlocal^{\lambda_2}\psi^n-\tilde{\psi}^n\}=0.\]}
\noindent \textbf{Step 1:} We prove that $\stp^{\blambda}\psi^{n}\geq \tilde{\psi}^{n}-C\lambda_1$ on $\RR$, where $C$ solves $rC=\exp(-C)$.

Let $w=\tilde{\psi}^{n}-C\lambda_1$, then
\[\cL w-rw+f\geq -rC\lambda_1 = \lambda_1\exp(-C)\geq \lambda_1\exp\left(-\frac{\nonlocal^{\lambda_2}\psi^{n}-w}{\lambda_1}\right),\]
where the last inequality is due to 
\(\nonlocal^{\lambda_2}\psi^{n}-w = \nonlocal^{\lambda_2}\psi^{n}-\tilde{\psi}^{n}+C\lambda_1\geq C\lambda_1\) by the second part of \eqref{eq:additional_HJB}.
Applying the maximum principle yields $\stp^{\blambda}\psi^{n}\geq w$.

\noindent \textbf{Step 2:} We prove that $\stp^{\blambda}\psi^{n}\leq \tilde{\psi}^{n}+C\lambda_1\log\frac{1}{\lambda_1}$ for some constant $C>0$.

Suppose by contradiction that $\sup_{x} (\stp^{\blambda}\psi^{n}(x)- \tilde{\psi}^{n}(x)-C\lambda_1\log\frac{1}{\lambda_1})>0$. Then there exists $x'$ such that $\stp^{\blambda}\psi^{n}(x')- \tilde{\psi}^{n}(x')-C\lambda_1\log\frac{1}{\lambda_1}>0$. By choosing a sufficiently small $\epsilon>0$, we have $\stp^{\blambda}\psi^{n}(x')- \tilde{\psi}^{n}(x')-C\lambda_1\log\frac{1}{\lambda_1}-\frac{\epsilon}{2}|x'|^2>0$. It further holds that $\sup_x (\stp^{\blambda}\psi^{n}(x)- \tilde{\psi}^{n}(x)-C\lambda_1\log\frac{1}{\lambda_1}-\frac{\epsilon}{2}|x|^2)>0$. By coercivity, we may assume that the maximum is attained at $x^*$.

For simplicity, we set $v(x) =\stp^{\blambda}\psi^{n}(x)- \tilde{\psi}^{n}(x)-C\lambda_1\log\frac{1}{\lambda_1}-\frac{\epsilon}{2}|x|^2$. By first- and second-order optimality conditions, we have that
\[\partial_{xx} v(x^*) = \partial_{xx}(\stp^{\blambda}\psi^{n}- \tilde{\psi}^{n})(x^*)-\epsilon\leq 0; \;  \partial_{x} v(x^*) = \partial_{x}(\stp^{\blambda}\psi^{n}- \tilde{\psi}^{n})(x^*)-\epsilon x^*=0.\]
Hence,
\[\cL v(x^*) = \left(b\partial_x+\frac{\sigma^2}{2}\partial_{xx}\right)v(x^*) = \cL(\stp^{\blambda}\psi^{n}-\tilde{\psi}^{n})(x^*) - \frac{\sigma^2\epsilon}{2}-b\epsilon x^* \leq 0. \]
Sending $\epsilon\to 0$ gives $\cL(\stp^{\blambda}\psi^{n}-\tilde{\psi}^{n})(x^*)\leq 0$. 
Meanwhile, combining \(v(x^*)>0\) with Proposition~\ref{propn:inter_bound_nonlocal_stp}, 
\begin{equation}\label{eq:stp_tilde_psi_bound}\tilde{\psi}^{n}(x^*)+C\lambda_1\log\frac{1}{\lambda_1}+\frac{\epsilon}{2}|x^*|^2<\stp^{\blambda}\psi^{n}(x^*)\leq\nonlocal^{\lambda_2}\psi^{n}(x^*)+C_1\lambda_1\log\frac{1}{\lambda_1},\end{equation}
Then sending $\epsilon\to 0$  yields $\tilde{\psi}^{n}(x^*)<\nonlocal^{\lambda_2}\psi^{n}(x^*)$. Hence, $\cL \tilde{\psi}^{n}-r\tilde{\psi}^{n}+f=0$ at $x^*$.
By \eqref{eq:value_improvement} with $\nonlocal^{\lambda_2}\psi^{\blambda,n}$ replaced by $\nonlocal^{\lambda_2}\psi^{n}$, we have
\[\cL(\stp^{\blambda}\psi^{n}-\tilde{\psi}^{n})(x^*) = r(\stp^{\blambda}\psi^{n}-\tilde{\psi}^{n})(x^*) +\lambda_1\exp\left(\frac{\nonlocal^{\lambda_2}\psi^{n}(x^*)-\stp^{\blambda}\psi^{n}(x^*)}{\lambda_1}\right).\]
By \eqref{eq:stp_tilde_psi_bound}, the above term is larger than $0$, which yields the contradiction.
\end{proof}

\begin{proof}[Proof of Theorem~\ref{thm:comparison_convergence}]
\textbf{Lower bound:} 
We show that $\psi^{\blambda,n}\geq \psi^n-\frac{\lambda_1}{r}$ on $\RR$ by induction. 

When $n=0$, it holds that $\psi^{\blambda,0} \equiv \psi^0$.
Suppose $\psi^{\blambda,n}\geq \psi^n-\frac{\lambda_1}{r}$. Set $w^n = \psi^n-\frac{\lambda_1}{r}$. By the value iteration \eqref{eq:classical_PI_hjb} and \ref{lem:compare-classical}, we have
\[\nonlocal^{\lambda_2}w^n =\nonlocal^{\lambda_2} \left(\psi^n-\frac{\lambda_1}{r}\right) =\nonlocal^{\lambda_2}\psi^n- \frac{\lambda_1}{r} \geq \nonlocal\psi^n- \frac{\lambda_1}{r}\geq \psi^{n+1}- \frac{\lambda_1}{r}= w^{n+1},\]
which implies
\(\lambda_1\exp\left(-\frac{\nonlocal^{\lambda_2}w^n-w^{n+1}}{\lambda_1}\right) \leq \lambda_1.\)
Combining the above inequality with \eqref{eq:classical_PI_hjb}, we further have
\[\begin{aligned}
&\cL w^{n+1}-rw^{n+1}+f = \cL\psi^{n+1} -r\psi^{n+1}+f+\lambda_1\geq \lambda_1\\
\geq& \lambda_1\exp\left(-\frac{\nonlocal^{\lambda_2}w^n-w^{n+1}}{\lambda_1}\right) \geq \lambda_1\exp\left(-\frac{\nonlocal^{\lambda_2}\psi^{\blambda,n}-w^{n+1}}{\lambda_1}\right),
\end{aligned}
\]
where the last inequality is due to the induction assumption.
Furthermore, recall by \eqref{eq:value_improvement} that
\[   \cL\psi^{\blambda,n+1}-r\psi^{\blambda,n+1}+f -\lambda_1\exp\left(-\frac{\nonlocal^{\lambda_2}\psi^{\blambda,n}-\psi^{\blambda,n+1}}{\lambda_1}\right)=0.\]
By comparison principle, we have $\psi^{\blambda,n+1}\geq w^{n+1} = \psi^{n+1}-\frac{\lambda_1}{r}$, which closes the proof.

Returning to the lower bound, let $n\to \infty$ in the claim, we have $\psi^{\blambda}\geq \psi-\frac{\lambda_1}{r}$ on $\RR$.

\noindent \textbf{Upper bound:} Fix some  $x\in K$.  By Proposition~\ref{propn:classical_iterates_bound}, $\psi^n-\psi\leq C(1-\mu)^n$ on $\RR$. Moreover, the policy iteration gives $\psi^{\blambda}\leq \psi^{\blambda,n}$. Hence,
\begin{equation}\label{eq:int_psi_lambda_bound}
\begin{aligned}
\psi^{\blambda}(x)-\psi(x)&= (\psi^{\blambda}(x)-\psi^{\blambda,n}(x)) + (\psi^{\blambda,n}(x)-\psi^{n}(x)) +(\psi^n(x)-\psi(x)) \\
&\leq (\psi^{\blambda,n}(x)-\psi^{n}(x)) + C(1-\mu)^n.
\end{aligned}
\end{equation}
For every \(n\in\NN\) and every pair \(\blambda=(\lambda_1,\lambda_2)\), we have
\begin{equation}\label{eq:psi_lambdan-psi_n_bound}
\begin{aligned}
&\psi^{\blambda,n}(x)-\psi^{n}(x)\\
\leq&\; \|\stp^{\blambda}\psi^{\blambda,n-1}-\stp^{\blambda}\psi^{n-1}\|_{L^{\infty}(K)}+ \|\stp^{\blambda}\psi^{n-1}-\tilde\psi^{n-1}\|_{L^{\infty}(K)}+\|\tilde\psi^{n-1}-\psi^n\|_{L^{\infty}(K)}\\
\leq&\; \|\psi^{\blambda,n-1}-\psi^{n-1}\|_{L^{\infty}(K)}+ C\lambda_1\log\frac{1}{\lambda_1}+\|\nonlocal^{\lambda_2}\psi^{n-1}-\nonlocal\psi^{n-1}\|_{L^{\infty}(K)},
\end{aligned}
\end{equation}
where the first inequality comes from triangular inequality, and the second inequality is due to \ref{lem:global-rand-nonexp}, Proposition~\ref{propn:randomized_os_convergence}, and Lemma~\ref{lemma:semi-randomized_nonlocal_bound}.

Iteratively applying the arguments of \eqref{eq:psi_lambdan-psi_n_bound}, and note that $\psi^{\blambda,0} =\psi^0$, we have
\begin{equation}
\begin{aligned}
&\psi^{\blambda,n}(x)-\psi^n(x)\\
\leq&\; \|\psi^{\blambda,0}-\psi^{0}\|_{L^{\infty}(K)}+nC\lambda_1\log\frac{1}{\lambda_1}+\sum_{i=1}^n\|\nonlocal^{\lambda_2}\psi^{i-1}-\nonlocal\psi^{i-1}\|_{L^{\infty}(K)}\\
=&\; nC\lambda_1\log\frac{1}{\lambda_1}+\sum_{i=1}^n\|\nonlocal^{\lambda_2}\psi^{i-1}-\nonlocal\psi^{i-1}\|_{L^{\infty}(K)}.
\end{aligned}
\end{equation}

For each integer $j\ge1$ set
\(\Delta_j(\lambda_2):=
    \max_{1\le i\le j}
    \bigl\|\nonlocal^{\lambda_2}\psi^{i-1}-\nonlocal\psi^{i-1}\bigr\|_{L^\infty(K)}.\)
  
\ref{lemma:nonlocal_convergence} implies
\(\Delta_j(\lambda_2)\to0\) as \(\lambda_2\to0\) for fixed $j$.
Choose a decreasing sequence \(\eta_j\downarrow0\) such that $\lambda_2\le\eta_j$, we have $\Delta_j(\lambda_2)<2^{-j}$.
Set
$n(\lambda_1):=\Bigl\lceil \frac1{|\log(1-\mu)|}\log\frac1{\lambda_1}\Bigr\rceil$, let $\lambda_2\le\eta_{n(\lambda_1)}$, we then have
\[
  \sum_{i=1}^{n(\lambda_1)}
    \bigl\|\nonlocal^{\lambda_2}\psi^{i-1}-\nonlocal\psi^{i-1}\bigr\|_{L^\infty(K)}
\leq n(\lambda_1)2^{-n(\lambda_1)}.
\]

Subsequently, for sufficiently large $C'>0$, one can derive from  \eqref{eq:psi_lambdan-psi_n_bound} that
\[\psi^{\blambda,n(\lambda_1)}(x)-\psi^{n(\lambda_1)}(x)
\leq n(\lambda_1)C\lambda_1\log\frac{1}{\lambda_1}+ n(\lambda_1)2^{-n(\lambda_1)}.
\]

Substituting back to \eqref{eq:int_psi_lambda_bound}, we have
\begin{equation}\label{eq:three-terms}
  \psi^{\blambda}(x)-\psi(x)
\leq
    n(\lambda_1)C'\lambda_1\log\frac1{\lambda_1}
   +n(\lambda_1)2^{-n(\lambda_1)}
   +C(1-\mu)^{n(\lambda_1)} .
\end{equation}

Put $\alpha:=\lambda_1\log(1/\lambda_1)$ and let
$\widetilde\mu:=\max\{1-\mu,\tfrac12\}<1$.  Both exponential terms on the RHS of \eqref{eq:three-terms} satisfy
$2^{-n(\lambda_1)},(1-\mu)^{n(\lambda_1)}\le\widetilde\mu^{\,n(\lambda_1)}$.
Applying \cite[Lemma 4.9]{reisinger2020error}  with $\gamma=1$ yields a constant $C_0$ such that
\[
  \alpha n(\lambda_1)+\widetilde\mu^{n(\lambda_1)}\le
  C_0\alpha\bigl(-\log\alpha\bigr)
  \quad(\lambda_1\to0).
\]
Hence the right–hand side of \eqref{eq:three-terms} is bounded by
$C\alpha(-\log\alpha)$, i.e.
\[
  \psi^{\blambda}(x)-\psi(x)
  \le
  C\lambda_1\log\frac1{\lambda_1}
    \log\Bigl(\frac1{\lambda_1\log(1/\lambda_1)}\Bigr),
\]
whenever $\lambda_2\le\eta_{n(\lambda_1)}$, and the definition of $\eta_{n(\lambda_1)}$ guarantees that $\lambda_1\to 0$, which simultaneously induces a sequence $\lambda_2\to 0$. Together with the lower bound
$\psi^{\blambda}\ge\psi-\lambda_1/r$ completes the proof.
\end{proof}

\subsection{Proof in Section~\ref{sec:rl}}
\label{app:a-5}
\begin{proof}[Proof of \Cref{thm:geom-rate}]
Denote $w^n(x)=\psi^{\blambda,n+1}(x)-\psi^{\blambda,n}(x)$, {$g^n(x) = \nonlocal^{\lambda_2}\psi^{\blambda,n}(x)-\nonlocal^{\lambda_2}\psi^{\blambda,n-1}(x)$}. Fix $n\in\NN^+$, for any $x\in\RR$,
\[\begin{aligned}\
&\cL w^n(x)-rw^n(x)\\
=&\; \lambda_1\exp\left(-\frac{\nonlocal^{\lambda_2}\psi^{\blambda,n}(x)-\psi^{\blambda,n+1}(x)}{\lambda_1}\right)-\lambda_1\exp\left(-\frac{\nonlocal^{\lambda_2}\psi^{\blambda,n-1}(x)-\psi^{\blambda,n}(x)}{\lambda_1}\right)\\
=&\;\exp\left(-\frac{c(x)}{\lambda_1}\right)\left(\psi^{\blambda,n+1}(x)-\psi^{\blambda,n}(x)-(\nonlocal^{\lambda_2}\psi^{\blambda,n}(x)-\nonlocal^{\lambda_2}\psi^{\blambda,n-1}(x))\right),
\end{aligned}\]
where the second equality holds by mean value theorem, with $c(x)$ pointwisely lying between $\nonlocal^{\lambda_2}\psi^{\blambda,n}(x)-\psi^{\blambda,n+1}(x)$ and $\nonlocal^{\lambda_2}\psi^{\blambda,n-1}(x)-\psi^{\blambda,n}(x)$.

{By \ref{increase} and \ref{improve},} both $\nonlocal^{\lambda_2}\psi^{\blambda,n-1}(x)-\psi^{\blambda,n}(x)$ and $\nonlocal^{\lambda_2}\psi^{\blambda,n}(x)-\psi^{\blambda,n+1}(x)$ are pointwisely no less than $\nonlocal^{\lambda_2}\psi^{\blambda}(x)-\psi^{\blambda,0}(x)$. Therefore, $c(x)\geq \nonlocal^{\lambda_2}\psi^{\blambda}(x)-\psi^{\blambda,0}(x)$. We further find a uniform lower bound for $\nonlocal^{\lambda_2}\psi^{\blambda}(x)-\psi^{\blambda,0}(x)$ to argue that $C_0(x)=:\exp(-(\nonlocal^{\lambda_2}\psi^{\blambda}(x)-\psi^{\blambda,0}(x))/\lambda_1)$ is uniformly upper bounded,  which will be invoked later. 

To derive the lower bound of $\nonlocal^{\lambda_2}\psi^{\blambda}$, notice that it is equivalent to \[\psi^{\blambda}(0)+l(-x)-\lambda_2\log \EE_{\xi\sim\Phi_x}\!\left[\exp\!\left(-\frac{\psi^{\blambda}(x+\xi)+l(\xi)-\psi^{\blambda}(0)-l(-x)}{\lambda_2}\right)\right].\]
As $\psi^{\blambda}\in \lips[\frac{L_f}{r-G}]$, we have \(\exp\!\left(-\frac{\psi^{\blambda}(x+\xi)+l(\xi)-\psi^{\blambda}(0)-l(-x)}{\lambda_2}\right) \le \exp\left(\frac{\left(\frac{L_f}{r-G}+L_l\right)|x+\xi|}{\lambda_2}\right)\).
As $x+\xi \sim \Phi_0$, one can derive that its expectation is bounded. This gives the uniform lower boundedness of $\nonlocal^{\lambda_2}\psi^{\blambda}$. Combining with the upper bound on $\psi^{\blambda,0}$ by \cref{prop:bdd} leads to a uniform upper bound on $C_0(x)$, denoted as $C_0$. Namely,
\[\cL w^n(x) - rw^n(x)\leq C_0 \left(\psi^{\blambda,n+1}(x)-\psi^{\blambda,n}(x)-(\nonlocal^{\lambda_2}\psi^{\blambda,n}(x)-\nonlocal^{\lambda_2}\psi^{\blambda,n-1}(x))\right).\]

In the following, we prove that
\[
  \|w^n\|_{L^\infty(\RR)}
  \leq \frac{C_0}{C_0+r}\,
       \|\nonlocal^{\lambda_2}\psi^{\blambda,n}
          -\nonlocal^{\lambda_2}\psi^{\blambda,n-1}\|_{L^\infty(\RR)}
  =: K_n.
\]
Let us first show $\sup_{x\in\RR}w^n(x)\leq K_n$.
Since $w^n$ is bounded, for each $\epsilon>0$ the function
\(
  v^{n,\epsilon}(x) := w^n(x)-\frac{\epsilon}{2} x^2
\)
satisfies $v^{n,\epsilon}(x)\to -\infty$ as $|x|\to\infty$ and hence attains its
maximum at some point $\overline{x}_\epsilon\in\RR$:
\(v^{n,\epsilon}(\overline{x}_\epsilon)
  = \max_{x\in\RR} v^{n,\epsilon}(x).\)

Recall that
\[
  \cL w^n(x) - r w^n(x)
  = a^n(x)\bigl(w^n(x) - g^n(x)\bigr),
\]
where
\[
  a^n(x) := \exp\!\Big(-\tfrac{c(x)}{\lambda_1}\Big)\in(0,C_0],
  \qquad
  g^n(x) := \nonlocal^{\lambda_2}\psi^{\blambda,n}(x)
            - \nonlocal^{\lambda_2}\psi^{\blambda,n-1}(x),
\]
and set $B_n := \|g^n\|_{L^\infty(\RR)}$ so that
$K_n = \tfrac{C_0}{r+C_0} B_n$.
At the maximum point $\overline{x}_\epsilon$ we have
\[
  (v^{n,\epsilon})'(\overline{x}_\epsilon)=0,
  \quad
  (v^{n,\epsilon})''(\overline{x}_\epsilon)\le 0,
  \quad
    (w^n)'(\overline{x}_\epsilon) = \epsilon\,\overline{x}_\epsilon,
  \quad
  (w^n)''(\overline{x}_\epsilon) \le \epsilon.
\]
By Assumption~\ref{assump:boundedness}, 
\[
  \cL w^n(\overline{x}_\epsilon)
  = \frac12\sigma^2(\overline{x}_\epsilon)(w^n)''(\overline{x}_\epsilon)
    + b(\overline{x}_\epsilon)(w^n)'(\overline{x}_\epsilon)
  \le \frac12\bar\sigma^2\,\epsilon + \bar b\,\epsilon|\overline{x}_\epsilon|
  \le C_1\,\epsilon\,(1+|\overline{x}_\epsilon|),
\]
for some constant $C_1>0$ independent of $n$ and $\epsilon$. Evaluating the
identity for $w^n$ at $x=\overline{x}_\epsilon$ and combining with the above
estimate yields
\[
  r w^n(\overline{x}_\epsilon)
  + \exp\left(-\frac{c}{\lambda_1}\right)\bigl(w^n(\overline{x}_\epsilon)
                                   - g^n(\overline{x}_\epsilon)\bigr)
  = \cL w^n(\overline{x}_\epsilon)
  \le C_1\,\epsilon\,(1+|\overline{x}_\epsilon|).
\]
Hence,
\begin{equation}\label{eq:wn-upper-xeps}\begin{aligned}
  w^n(\overline{x}_\epsilon)
  &\le \frac{\exp\left(-\frac{c}{\lambda_1}\right)}{r+\exp\left(-\frac{c}{\lambda_1}\right)} \|g^n\|_{L^\infty(\RR)}
       + \frac{C_1}{r}\epsilon(1+|\overline{x}_\epsilon|)\\
  &\le \frac{C_0}{r+C_0}\|g^n\|_{L^\infty(\RR)} + C_2\epsilon(1+|\overline{x}_\epsilon|)\\
  &= K_n + C_2\epsilon(1+|\overline{x}_\epsilon|),
  \end{aligned}
\end{equation}
for some constant $C_2>0$ independent of $n,\epsilon$.

We now show that $\epsilon(1+|\overline{x}_\epsilon|)\to 0$ and that
$\sup_{\RR} w^n$ is bounded by $K_n$. Denote
\(m_\epsilon := v^{n,\epsilon}(\overline{x}_\epsilon)
  = \max_{x\in\RR} v^{n,\epsilon}(x)
  = w^n(\overline{x}_\epsilon) - \frac{\epsilon}{2}\,\overline{x}_\epsilon^2.\)
Since $\overline{x}_\epsilon$ maximizes $v^{n,\epsilon}$, we have
$v^{n,\epsilon}(\overline{x}_\epsilon)\ge v^{n,\epsilon}(0)$, that is,
\[
  w^n(\overline{x}_\epsilon) - \frac{\epsilon}{2}\,\overline{x}_\epsilon^2
  \ge w^n(0),
\]
and hence
\[
  \frac{\epsilon}{2}\,\overline{x}_\epsilon^2
  \le w^n(\overline{x}_\epsilon) - w^n(0)
  \le |w^n(\overline{x}_\epsilon)| + |w^n(0)|
  \le 2\|w^n\|_{L^\infty(\RR)}.
\]
Thus $\overline{x}_\epsilon = O(\epsilon^{-1/2})$, and
  \(\epsilon\,|\overline{x}_\epsilon|
  \to0 \) as \(\epsilon \to 0\).

Let $M^* := \sup_{x\in\RR} w^n(x)$. For each $\epsilon>0$,
\[
  m_\epsilon
  = \max_{x\in\RR} v^{n,\epsilon}(x)
  \le \sup_{x\in\RR} w^n(x) = M^*.
\]Conversely, for any $\eta>0$
there exists $x_\eta\in\RR$ such that $w^n(x_\eta)\ge M^*-\eta$, and hence
\[
  m_\epsilon
  \ge v^{n,\epsilon}(x_\eta)
  = w^n(x_\eta) - \frac{\epsilon}{2}x_\eta^2
  \ge M^* - \eta - \frac{\epsilon}{2}x_\eta^2.
\]
Letting $\epsilon\downarrow 0$ for fixed $\eta$ yields
$\liminf_{\epsilon\downarrow 0} m_\epsilon \ge M^* - \eta$, and since
$\eta>0$ is arbitrary we conclude that $m_\epsilon\to M^*$ as
$\epsilon\downarrow 0$. Moreover, $w^n(\overline{x}_\epsilon)\ge v^{n,\epsilon}(\overline{x}_\epsilon)
= m_\epsilon$, so
\[
  \limsup_{\epsilon\downarrow 0} w^n(\overline{x}_\epsilon)
  \;\ge\; \lim_{\epsilon\downarrow 0} m_\epsilon
  = M^*.
\]
On the other hand, from \eqref{eq:wn-upper-xeps} and
$\epsilon(1+|\overline{x}_\epsilon|)\to 0$ we obtain
\[
  \limsup_{\epsilon\downarrow 0} w^n(\overline{x}_\epsilon)
  \le K_n.
\]
Combining the two displays gives $M^*\le K_n$, i.e.
\[
  \sup_{x\in\RR} w^n(x) \le K_n.
\]

An entirely analogous argument applied to $-w^n$ in place of $w^n$
gives $-\inf_{x\in\RR} w^n(x)\le K_n$. Therefore
\[
  \|w^n\|_{L^\infty(\RR)}
  = \sup_{x\in\RR} |w^n(x)|
  \le K_n
  = \frac{C_0}{r+C_0}\,
    \|\nonlocal^{\lambda_2}\psi^{\blambda,n}
      -\nonlocal^{\lambda_2}\psi^{\blambda,n-1}\|_{L^\infty(\RR)},
\]
which is the desired contraction estimate.

Combining with non-expansiveness of $\nonlocal^{\lambda_2}$ yields \(\|\psi^{\blambda,n+1}-\psi^{\blambda,n}\|_{L^\infty(\RR)}\leq \frac{C_0}{r+C_0}\|\psi^{\blambda,n}-\psi^{\blambda,n-1}\|_{L^\infty(\RR)}\).
For simplicity, let $q\coloneqq\frac{C_0}{r+C_0}$. 
Fix some $n\in\NN^+$, for any integer $m>n$, applying the above inequality iteratively, we have
\[\|\psi^{\blambda,m}-\psi^{\blambda,n}\|_{L^\infty(\RR)}\leq \sum_{k=n}^{m-1}\|\psi^{k+1}-\psi^k\|_{L^\infty(\RR)}\leq \frac{q^n}{1-q}\|\psi^{\blambda,1}-\psi^{\blambda,0}\|_{L^\infty(\RR)}. \] 
By \Cref{thm:policy_improvement}, $\lim_{m\to \infty} \psi^{\blambda,m} = \psi^{\blambda}$, and $\psi^{\blambda}\leq \psi^{\blambda,1}\leq \psi^{\blambda,0}$. Sending $m\to \infty$, we conclude that for any $n\in \NN$,
\[\|\psi^{\blambda}-\psi^{\blambda,n}\|_{L^\infty(\RR)}\leq \frac{q^n}{1-q}\|\psi^{\blambda,1}-\psi^{\blambda,0}\|_{L^\infty(\RR)} \leq \frac{q^n}{1-q}\|\psi^{\blambda}-\psi^{\blambda,0}\|_{L^\infty(\RR)}.\]
\end{proof}
\section{Additional Experimental Results}\label{app:b}
In \Cref{subfig:psi-lambda-1,subfig:psi-lambda-0.5}, we display the impact of \(\sigma\) on the learned value function with respect to two different \(\blambda\) values, i.e., \(\blambda = (1.0,1.0), (0.5,0.5)\).
For each \(\blambda\), the curves retain the same shape and monotone ordering in \(\sigma\) (flatter near the origin and higher overall as \(\sigma\) increases), reflecting a broader effective continuation region and a reduced marginal value of aggressive impulses. This observation suggests that the impact of \(\sigma\) on the randomized value functions enjoys a general consistency across different \(\blambda\) values. 
Comparing the two \(\blambda\) instances, stronger entropy-regularization with \(\blambda=(1.0,1.0)\) yields a smoother, slightly higher cost profile near the origin, consistent with more conservative interventions; weaker regularization with \(\blambda=(0.5,0.5)\) permits more aggressive corrections and correspondingly lower minima.

In \Cref{subfig:pi-lambda-1,subfig:pi-lambda-0.5}, we compare the impact of \(\sigma\) on the intervention intensity under the two entropy settings. Similar to the observation for value functions, we also observe the general consistency in the impact of \(\sigma\) on the intensity across different values of \(\sigma\). 
For any fixed \(\sigma\), the regularization effects under different \(\blambda\) values are clear. When regularization is weaker, we can see that the flat region where the intensity \(\pi^*\) remains close to \(0\) visibly enlarges, representing a larger waiting region; when the state variable escapes away from the flat region, the increase in the intensity gets more rapid, representing for a more aggressive intervention. 
\begin{figure}[htb]
  \centering
  \begin{subfigure}[htb]{.48\textwidth}
    \centering
    \includegraphics[width=\linewidth]{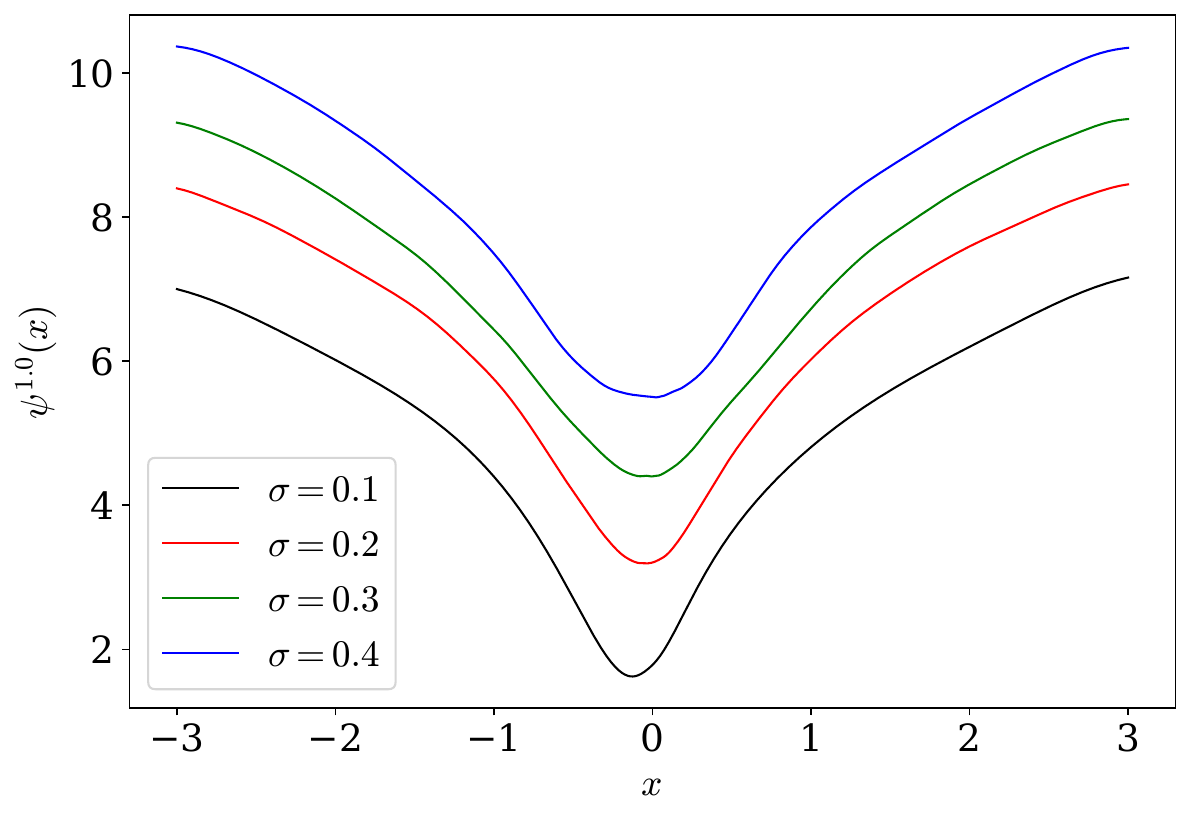}
    \caption{Value functions at \(\blambda = (1.0,1.0)\).}
    \label{subfig:psi-lambda-1}
  \end{subfigure}
  \hfill
  \begin{subfigure}[htb]{.48\textwidth}
    \centering
    \includegraphics[width=\linewidth]{figures/psi_sigma_comparison_lamb_0.5}
    \caption{Value functions at \(\blambda = (0.5,0.5)\).}
    \label{subfig:psi-lambda-0.5}
\end{subfigure}\\
  \begin{subfigure}[htb]{.48\textwidth}
    \centering
    \includegraphics[width=\linewidth]{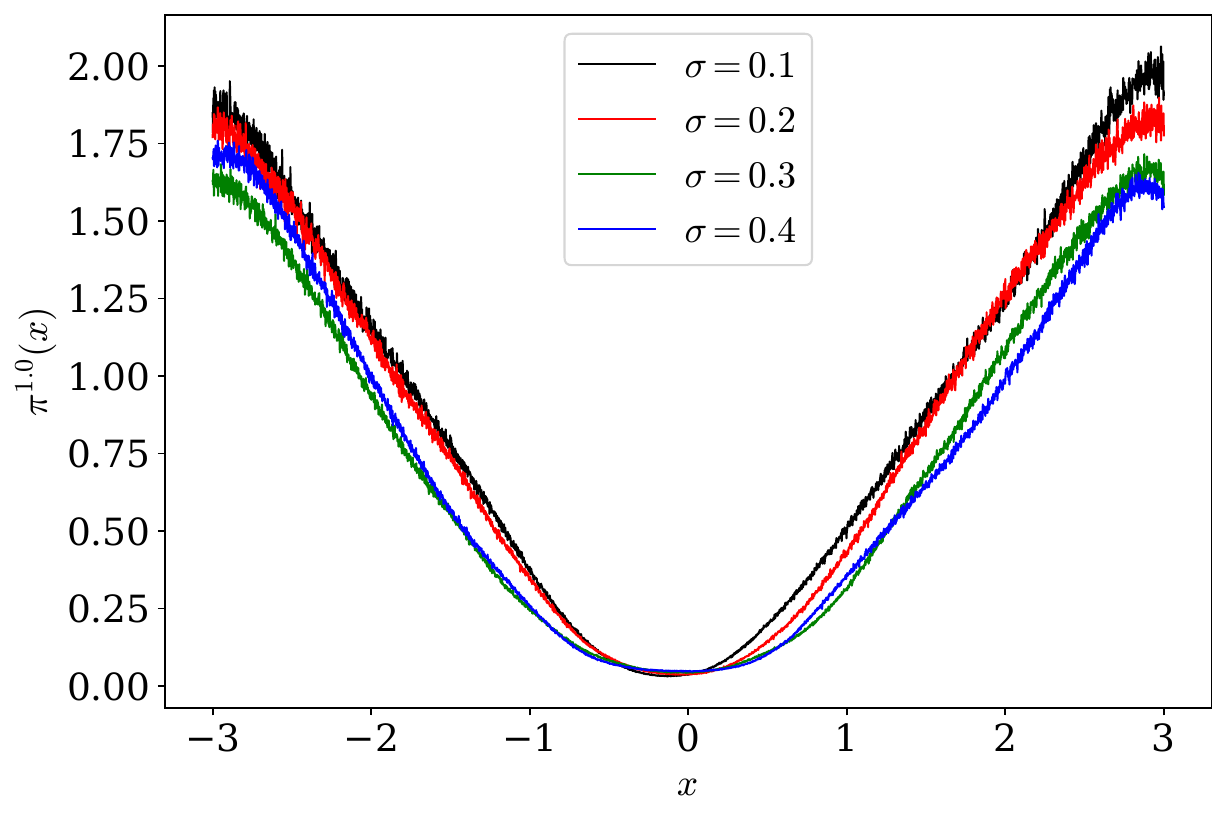}
    \caption{Intensities of interventions at \(\blambda = (1.0,1.0)\).}
    \label{subfig:pi-lambda-1}
  \end{subfigure}
  \hfill
  \begin{subfigure}[htb]{.48\textwidth}
    \centering
    \includegraphics[width=\linewidth]{figures/pi_sigma_comparison_lamb_0.5}
    \caption{Intensities of interventions at \(\blambda = (0.5,0.5)\).}
    \label{subfig:pi-lambda-0.5}
\end{subfigure}
\caption{Sensitivity analysis with respect to volatility $\sigma$ for Algorithm~\ref{alg:td_impulse} with \(\blambda=(1.0,1.0),(0.5,0.5)\).}
\label{fig:comparison}
\end{figure}




\bibliographystyle{plain}
\bibliography{ref}

\end{document}